\documentclass[a4paper,12pt,reqno]{amsart}

\usepackage{manfnt}
\usepackage{CJKutf8}
\def\myfnt{\ifx\protect\@typeset@protect\expandafter\footnote\else\expandafter\@gobble\fi}
\makeatother

\usepackage{amsfonts}
 \usepackage{amssymb}
 \usepackage{amsmath,amsxtra,amsthm}
 \usepackage{mathtools} 
 \usepackage{mleftright} 
 \usepackage{stmaryrd} 

 \makeatletter
\DeclareFontEncoding{LS1}{}{}
\DeclareFontSubstitution{LS1}{stix}{m}{n}
\DeclareMathAlphabet{\mathcal}{LS1}{stixscr}{m}{n}
\usepackage[T1]{fontenc}
\usepackage{fancyhdr}
\usepackage[usenames]{color}
\usepackage{tikz} 
\usepackage{tikz-3dplot}
\usepackage{pgfplots}
\usepgfplotslibrary{colormaps}
\usepackage{pst-solides3d}
 
\usetikzlibrary{cd}
\usepackage[mathscr]{eucal}

 \usepackage{dsfont}

\usepackage{hyperref}

\usepackage[utf8]{inputenc}
\usepackage{wrapfig}

\usepackage[all,cmtip]{xy}
\usepackage{leftindex}
\newtheorem{theorem}{Theorem}[section]

\newtheorem{lemma}[theorem]{Lemma}
\newtheorem{cor}[theorem]{Corollary}
\newtheorem{prop}[theorem]{Proposition}

\theoremstyle{definition}
\newtheorem{deff}[theorem]{Definition}

\newtheorem{questionIntro}{Question}
\newtheorem{example}[theorem]{Example}

\theoremstyle{remark}
\newtheorem{rem}[theorem]{Remark}

\newcommand{\pairing}{outer holonomy}

\newcommand{\data}{leaf data}

\newcommand{\datas}{leaf data}
\newcommand{\Datas}{Leaf data}
\definecolor{blue}{rgb}{0,0,0}

\expandafter\let\expandafter\oldproof\csname\string\proof\endcsname
\let\oldendproof\endproof
\renewenvironment{proof}[1][\proofname]{%
  \oldproof[\bfseries #1.]%
}{\oldendproof}

\newcommand{\bqa}{\begin{eqnarray}}
\newcommand\eqa {\end{eqnarray}}
\newcommand{\beq}{\begin{eqnarray}}
\newcommand{\beqn}{\begin{eqnarray}\nonumber}
\newcommand{\eeq}{\end{eqnarray}}
\newcommand{\be}{\begin{array}}
\newcommand{\ee}{\end{array}}

   \newcommand\vf\varphi

 \newcommand{\Aut}{\mathrm{Aut}}

 \newcommand{\cT}{{\mathcal T}}

 \newcommand{\cG}{{\mathcal{G}}}
 
 \newcommand{\cF}{{\mathcal{F}}}
\newcommand{\cU}{{\mathcal U}}

 \newcommand{\makecommand}[3]{%
    \foreach \i in #3 {%
        \expandafter\xdef\csname #1\i\endcsname{\noexpand#2{\unexpanded\expandafter{\i}}}%
    }%
}
 \newcommand{\latinalphabet}{A,a,B,b,C,c,d,D,E,e,F,f,G,g,H,h,I,i,J,j,K,k,L,l,M,m,N,n,O,o,P,p,Q,q,R,r,S,s,T,t,U,u,V,v,W,w,X,x,Y,y,Z,z} 

\makecommand{rm}{\mathrm}{\latinalphabet}

 \newcommand{\R}{{\mathbb R}}

   \def\t{\theta}

   \def\i{\imath}

\DeclareFontFamily{OT1}{pzc}{}
\DeclareFontShape{OT1}{pzc}{m}{it}{<-> s * [1.15] pzcmi7t}{}
\DeclareMathAlphabet{\mathpzc}{OT1}{pzc}{m}{it}

\newcommand\spiral{}
\def\spiral[#1][#2][#3:#4:#5]
{
\pgfmathsetmacro{\domain}{pi*#3/180+#4*2*pi}
\draw [#1,shift={(#2)}, domain=0:\domain,variable=\t,smooth,samples=int(\domain/0.08)] plot ({\t r}: {#5*\t/\domain})
}

\newcommand{\leaf}{L}

\newcommand{\vertical}{\pi \colon T \rightarrow \leaf}
\newcommand{\normal}{T}

\definecolor{SaemannsBlau}{rgb}{0.15,0.25,0.45}
\hypersetup{
    colorlinks=true,
    citecolor=SaemannsBlau,
    linkcolor=SaemannsBlau,
    urlcolor=SaemannsBlau,
    pdflang={en-GB},
    breaklinks=true
}

\pgfplotsset{compat=1.18}

\renewcommand{\tocsection}[3]{%
  \indentlabel{\@ifnotempty{#2}{\bfseries\ignorespaces#1 #2\quad}}\bfseries#3}
\renewcommand{\tocsubsection}[3]{%
  \indentlabel{\@ifnotempty{#2}{\ignorespaces#1 #2\quad}}#3}

  \def\l@subsection{\@tocline{2}{0pt}{2.5pc}{5pc}{}}

\begin{document}

\tdplotsetmaincoords{70}{0}
\tikzset{declare function={torusx(\u,\v,\R,\r)=cos(\u)*(\R + \r*cos(\v)); 
torusy(\u,\v,\R,\r)=(\R + \r*cos(\v))*sin(\u);
torusz(\u,\v,\R,\r)=\r*sin(\v);
vcrit1(\u,\th)=atan(tan(\th)*sin(\u));
vcrit2(\u,\th)=180+atan(tan(\th)*sin(\u));
disc(\th,\R,\r)=((pow(\r,2)-pow(\R,2))*pow(cot(\th),2)+%
pow(\r,2)*(2+pow(tan(\th),2)))/pow(\R,2);
umax(\th,\R,\r)=ifthenelse(disc(\th,\R,\r)>0,asin(sqrt(abs(disc(\th,\R,\r)))),0);
}}

\bibliographystyle{amsplain}

\title{A classification of neighborhoods around leaves of a singular foliation}

\author[Simon-Raphael Fischer]{Simon-Raphael Fischer}
\address{Simon-Raphael Fischer: National Center for Theoretical Sciences (\begin{CJK*}{UTF8}{bkai}國家理論科學研究中心\end{CJK*}), Mathematics Division, National Taiwan University 
\newline
No.\ 1, Sec.\ 4, Roosevelt Rd., Taipei City 106, Room 407, Cosmology Building, Taiwan
\newline
\begin{CJK*}{UTF8}{bkai}臺北市羅斯福路四段1號 (國立臺灣大學次震宇宙館407室)\end{CJK*}
\newline
\textit{ORCiD}: \href{https://orcid.org/0000-0002-5859-2825}{0000-0002-5859-2825}}
\email{\href{mailto:sfischer@ncts.tw}{sfischer@ncts.tw}}

\author[Camille Laurent-Gengoux]{Camille~Laurent-Gengoux}
\address{Camille~Laurent-Gengoux: Institut Elie Cartan de Lorraine (IECL), UMR 7502 --  3 rue Augustin Fresnel, 57000 Technop\^ole Metz, France
\newline
\textit{ORCiD}: \href{https://orcid.org/0000-0002-9363-869X}{0000-0002-9363-869X}}
\email{\href{mailto:camille.laurent-gengoux@univ-lorraine.fr}{camille.laurent-gengoux@univ-lorraine.fr}}


\begin{abstract} 
We classify singular foliations admitting a given leaf and a given transverse singular foliation.
\end{abstract}


\maketitle

\tableofcontents

\section*{Introduction}

Although much less studied than regular foliations, singular foliations are extremely common in differential geometry. After an intense debate in the 1970s about what their very definition should be \cite{Lavau18}, there is an almost-consensus to define a singular foliation as a locally finitely generated sub-module of the module of vector fields which is stable under Lie bracket, following Hermann \cite{Hermann}, Stefan \cite{stefan1974accessible, Stepan1,Stepan2}, and Sussmann \cite{Sussmann1, Sussmann2}.
We will use this widely accepted definition, but we think that most of the work which is done here should hold true with minor adaptations using the alternative definition of singular foliations as in \cite{Miyamoto}.
Non-commutative geometry, presently the most active area of research that makes use of such singular foliations, uses examples associated to differential operators  (see, \textit{e.g.}\ \cite{AMY,AS,Debord}), sometimes through the adiabatic Lie algebroid constructions, (see \textit{e.g.}\ \cite{zbMATH07217274,VictorNistor}). Singular foliations are also widely studied in complex geometry, see \textit{e.g.}\ \cite{zbMATH07124409} for straightening theorems in this context. They also appear in Hamiltonian systems since  Poisson structures, and more generally Lie algebroids, induce singular foliations \cite{crainic2021lectures,LPV,zbMATH06745657}. Last, they appear in mathematical physics \cite{LLS,NahariStrobl} in relation with the notion of dg-manifold. 

Singular foliations deserve their name, because they induce a partition of the manifold by immersed submanifolds called \emph{leaves} \cite{Hermann,Stepan1,Sussmann1}. 
We intend here to complete a  classification of all singular foliations in a neighborhood of a leaf $L$. This is a continuation of two articles of the second author with Leonid Ryvkin. In \cite{LGR}, an equivalent of the monodromy for singular leaves was constructed, while \cite{Ryvkin2} addressed the formal linearization problem, \textit{i.e.}\  Levi-Malcev theorems. 
Also, a classification in the codimension one case was recently completed by Francis in \cite{PhD_Francis,francis2023singular}, and a classification for vector fields tangent up to order $k$ of a given submanifold was recently completed by Bischoff, del Pino, and Witte in \cite{BDW}. We recover these classifications, at least their formal equivalent.
\begin{blue}
There is probably also a relation with \cite{macdonald2021hierarchies, macdonald2021holonomy}: we construct (albeit formally) the diffeologies that appear in these articles, but we do not need to insist on the holonomy groupoid itself.
\end{blue}

To understand what a classification in a neighborhood of a leaf can achieve, recall that a basic result for singular foliations states that "if one travels along a leaf, one always sees the same landscape through the window".
More precisely, there is a splitting theorem: For a leaf of codimension $d$, there is a well-defined (germ or $\infty $-jet of) singular foliation $\mathcal T_0 $ on a neighborhood of $0$ in $\mathbb R^{d} $ such that, near every point of the leaf, $\mathcal F $ is simply the direct product of this leaf with that transverse singular foliation $\mathcal T_0 $. This $\mathcal T_0 $ is constant along the leaf, and does not 
 depend on any choice. 
We call the (germ or $\infty $-jet of) singular foliation $\mathcal T_0 $ the \emph{transverse model}.
We address the following question:

\begin{questionIntro}\label{EssentialQuestions}
Given a manifold $L$ and a singular foliation $ \mathcal T_0$ near $0$ on $ \mathbb R^d $, how many singular foliations admit $L$ as leaf and $ \mathcal T_0$ as a transverse model? How to classify them?
\end{questionIntro}

The question "how many?" may seem naïve and one may think that the answer is "quite a lot", so that the best one could hope would be an answer of the type "as much as we have elements in some-huge-set". But there are strong indications that this set is not that "huge". For $L$ an open ball, it was proven in \cite{AS} that a singular foliation is uniquely determined by $L$ and $\cT_0$, so that the answer to Question \ref{EssentialQuestions} is "one and only one" (which then has to be the trivial product). We will extend this result to any contractible $L$. Moreover, the second author and Leonid Ryvkin have proven in \cite{Ryvkin2} that the answer is "one and only one" again provided that $L$ is simply connected and $\mathcal T_0 $ is made of vector fields that vanish at least quadratically at $0$. Said otherwise, despite the huge variety and the possibly rich geometry of the transverse models $\mathcal T_0 $, there are cases where there are much less singular foliations near a given leaf than naively thought. 

Also, there are cases for which the answer to Question \ref{EssentialQuestions} is known. If the chosen transverse model $ \mathcal T_0$ is the smallest possible one, \emph{i.e.} the trivial singular foliation of vector fields all equal to $0$, then we are indeed looking for regular foliations admitting $L$ as a leaf. It is known that there are as many such regular foliations on a formal neighborhood of $L$ as there are conjugacy classes of group morphisms
 $$  \pi_1(L) \longrightarrow \mathrm{Diff}(\mathbb R^d, 0)  $$
 where $ \mathrm{Diff}(\mathbb R^d, 0) $ stands for the group of germs of diffeomorphisms of $\mathbb R^d $ mapping $0$ to itself \cite{GroupoidBasedPrincipalBundles}.   If the chosen transverse model $ \mathcal T_0$ is the largest possible one, \emph{i.e.}  the singular foliation of all vector fields on $ \mathbb R^d$ vanishing at $0$, then we are indeed classifying rank $d$ vector bundles over $L$, \textit{i.e.}\ we are classifying principal $\mathrm{GL}_d(\mathbb R) $-bundle over $L$. As we will see in Theorem \ref{thm:classificationSimplified}, our classification has these two ingredients above: a group morphism from $\pi_1(L) $ to some group of symmetries, and a finite dimensional principal bundle.

\begin{blue}For leaves of dimension one, \textit{i.e.}\ $ \mathbb{S}^1$, it is also relatively elementary  to see that a local classification is given by conjugacy classes of outer symmetries of the transverse model, see Ex.\ \ref{ex:S1asleaf}.\end{blue} For the sphere 
$ \mathbb{S}^2$, the problem of classification is not that trivial. An obvious sign that a general theory was missing is that for  the torus $ \mathbb{T}^2$, the question was open.  Consider for instance the case where the transverse model is given by one of the following two singular foliations on $\mathbb R^2 $:

\begin{center}
\begin{tabular}{|c|c|c|}
\hline
Name &Concentric circles & Spirals \\
\hline
Generator  & $x \partial_y - y \partial_x $  & $x \partial_y - y \partial_x + (x^2+y^2)(x \partial_x + y \partial_y) $ 
\\ \hline 
\begin{tabular}{c}Picture\\ of \\ the \\ leaves  \\ \\ \\   \end{tabular} &  \begin{tabular}{c} \\
\begin{tikzpicture}
\filldraw[color=red!60, fill=red!0, very thick](0,0) circle (1.5);
\filldraw[color=green!60, fill=red!0, very thick](0,0) circle (1.25);
\filldraw[color=blue!60, fill=red!0, very thick](0,0) circle (1);
\filldraw[color=green!60, fill=red!0, very thick](0,0) circle (0.75);
\filldraw[color=red!60, fill=red!0, very thick](0,0) circle (0.5);
\filldraw[color=black!100, fill=black!100, very thick](0,0) circle (0.15);
\end{tikzpicture}
\end{tabular}
 &
 \begin{tabular}{c} \\
 \begin{tikzpicture}
 \node[rotate=90] at (0.3,0.2) {
    \begin{tikzpicture}[node distance = 0cm, auto](0,0)\spiral[red][0,0][0:2:2] ;
    \end{tikzpicture}};
 \node[rotate=180] at (-0.1,0.2) {
    \begin{tikzpicture}[node distance = 0cm, auto](0,0)\spiral[green][0,0][0:2:2] ;
    \end{tikzpicture}};
    
\spiral[blue][0,0][0:2:2];
  \filldraw[color=black!100, fill=black!100, very thick](0,0) circle (0.15);
\end{tikzpicture}
 \end{tabular}
\\ \hline

\end{tabular}
\end{center}

Here are natural questions:
\begin{enumerate}
\item[Q2] {\textbf{a}:} Can the concentric circle singular foliation be the transverse model of a singular foliation on $\rmT \mathbb{S}^2$ admitting $ \mathbb{S}^2$ as a leaf?

\noindent {\textbf{b}:} What about the spiral case?
\item[Q3] {\textbf{a}:} Given two symmetries $ \phi,\psi$ of these singular foliations such that $\phi \circ \psi (x) \sim \psi \circ \phi (x) $, where $ \sim $ means "to be in the same leaf", can we construct a singular foliation admitting the $2$-torus $ \mathbb{T}^2$ as a leaf and the concentric circles as transverse models?

\noindent {\textbf{b}:} What about spirals? 

\noindent {\textbf{c}:} If yes, is it unique?
\end{enumerate}
The answer to \textbf{Q2a} is "yes", and it is relatively easy to find an example: {\color{blue} Vector fields $X$ on $\rmT \mathbb{S}^2$ that are zero when applied to the square of the euclidean norm. The answer to {\textbf{Q2b}} is "no": See Example \ref{ex:IntroQuestionsAbout2SphereAsLeaf} It could be seen as a consequence of the hairy ball theorem. However, we will give a systematic description that answers all such questions at once.}
For questions, {\textbf{Q3a}}-{\textbf{Q3b}}, the answer would be "yes" if $ \phi$ and $\psi$ commuted: One could then make a suspension construction, using some $\mathbb Z^ 2 $-action. We will see in Examples \ref{ex:torustriple}-\ref{ex:IntroQuestionsAboutTorusAsLeaf} that even if $ \phi$ and $ \psi$ do not commute and no $\mathbb Z^2 $-action exists, the answers to {\textbf{Q3a}}-{\textbf{Q3b}} are "yes, it exists" but the meaning of the relation $\sim$ has to be made more precise. Again the heavy machinery of our classification is required there. The answer of {\textbf{Q3c}} is that it will only be unique in the case of spirals. In the case of circles, there are as many of them as there are elements in $\mathbb Z $. In most cases above, the complete answer, as we will see, goes through our classification and therefore goes through some infinite dimensional geometry.

In fact, we only give a classification and therefore answer these questions at formal level. The manifold $ L$ will be for us an ordinary smooth manifold, but the transverse model $\mathcal T_0 $ will be only a \textit{formal} singular foliation, {\color{blue} i.e.\ the infinite jet of a singular foliation - a precise description will be given below.  This is a debatable choice, so let us justify it. This means that we only classify infinite jets of a smooth singular foliation along a given smooth embedded submanifold. This is of course a first step, exactly as Conn's linearization theorem \cite{conn1984normal, conn1985normal}, or its extensions for Lie algebroids \cite{Zung, crainic2011geometric}. All these results start by a formal linearization.}
 Also we will see that using infinite jets allows one to reduce some infinite dimensional principal bundles to ordinary finite-dimensional ones, which would not be easy using, e.g., germs.

 Of course, we believe that under some real-analyticity  condition and/or compactness conditions, most results can be enlarged, but this will not be discussed here. 

\begin{blue}
Also, we relate this theorems with the formalism of \textit{Yang-Mills connections} (=YMCs). In fact, in the first version of the article, those played a crucial role in the proof  of the classification. 
In the present version, following the suggestions of the referees, we decided to reduce the role of YMCs, and to delay the discussion about it to the Appendix, where we generalize our classification to YMCs with a centerless group. Unlike the first version, where YMCs were used to prove the classification, it does not appear explicitly in the proof anymore, although it may help to have them in mind, so let us say a few words about them, following the presentation of \cite{SRFCYM}, and to its history \cite{OriginofCYMH, mayer2009lie, CurvedYMH, My1stpaper, MyThesis} (with similar problems appearing in \cite{samann2020towards, Kim:2019owc, rist2022explicit, MR2157566}). \end{blue} 
In one sentence, given a fibre bundle $\normal \to L $ on which a Lie group bundle (LGB) $\mathcal G \to L $ acts on the left, YMCs are Ehresmann connections with a certain compatibility w.r.t.\ the $\cG$-action and 
whose parallel transports\footnote{Let us however point out that we are considering a particular class of YMC that appeared in \cite{SRFCYM}, which we call here \textit{multiplicative} YMCs, because we are in the particular case where the group action is faithful. As we will see in a future work, what we do for singular foliations extends to several geometrical situations (\textit{e.g.}\ dg-manifolds or Poisson structures), but for singular foliations YMCs will be automatically multiplicative. We do not have to consider curved Yang-Mills gauge theories as in \cite{SRFCYM}: We do not have a curving $3$-form valued in some center (\cite{My1stpaper, MyThesis}), nor an adjustment term.}
over any contractible loop are given by the action of some element in the identity component of $\mathcal G $.
Several examples are listed in \cite{SRFCYM}.

What is the relation between YMCs and singular foliations? 
Geometrically, the existence of a YMC is linked to the following phenomenon (to our knowledge, first noticed by \cite{Dazord}). 
Consider a path $\gamma $ on the leaf $\leaf $ of a singular foliation $\mathcal F $ with transverse singular foliation $\mathcal T_0 $. There is a way to define parallel transportation along this path, through a particular type of Ehresmann connections called $\mathcal F $-connections in \cite{LGR}, such that integral curves are always in the same leaf of $ \mathcal F$. If $\gamma $ is a loop, it means that each leaf of $ \mathcal T_0$ is mapped to a leaf of $\mathcal T_0 $. But it may not be the same leaf. However,  if the loop is contractible, it has to be the same leaf. 
That is, it is a Yang-Mills connection with respect to a subgroup bundle of symmetries. Of course, since we work in the formal case, such an interpretation should be considered as an heuristic idea only. But it can be made precise:
\begin{blue}
We will show (in the body of the article) that any formal $\cF$-connection induces a Yang-Mills connection, and that choosing a different formal $\cF$-connection amounts to what is called field redefinitions of the corresponding Yang-Mills connection (\cite{My1stpaper, MyThesis, SRFCYM}).
In the Appendix, we will classify centerless Yang-Mills bundles, a classification very similar to the one of singular foliations along a leaf.
\end{blue}

Although the diffeologies involved in the proofs are far from easy to deal with, our final answer to Question \ref{EssentialQuestions} uses only very classical differential geometry.
We show that any singular foliation admitting a given leaf $L$ and a given transverse model $\mathcal T $ induces a triple as follows, which we also call the \emph{\data{}}: 
\begin{enumerate}
\item a Galois cover $L' $ of $L$ which has as many elements as there will be outer symmetries of $\mathcal T_0 $ coming through "wandering along the leaf and coming back",
\item the group extension $ G \hookrightarrow H \hookrightarrow \pi_1(L',L)  $ of all such symmetries obtained through "wandering along the leaf and coming back",
\item some $H$-principal bundle over $L$ on the top of $L'$.
\end{enumerate}
We show that, at the formal level, such triples entirely classify singular foliations admitting $ L$ as a leaf (see Thm.\ \ref{thm:classification}). The precise meaning of "formal level" will be given in the text.
 This result can be considerably simplified, by showing that, at formal level, \datas{} can even be replaced by an even simpler description (see Theorem \ref{thm:classificationSimplified}):
\begin{enumerate}
\item a group morphism $\phi \colon \pi_1(L) \longrightarrow {\mathrm{Out}}(\mathcal T_0)  $ (called \pairing{}), with ${\mathrm{Out}}(\mathcal T_0)$ the group of outer symmetries of $ \mathcal T_0$, and
\item some principal $R$-bundle extending the Galois cover associated to the kernel of $ \phi$, with $R$ some finite dimensional Lie group to be described in the text.
\end{enumerate}
 Theorems  \ref{thm:classification}  and \ref{thm:classificationSimplified} are, in our opinion, the two most important results of the present article.

The result obtained by the second author and Ryvkin \cite{Ryvkin2} mentioned before is an immediate consequence. If $L$ is simply connected, there is no non-trivial Galois cover and the two first items in the \data{} disappear. If the transverse model singular foliation $ \mathcal T_0$ is made of vector fields that vanish at least quadratically, then $H={\mathrm{Inner}}(\mathcal T_0)$ is nilpotent-like and its group of inner symmetries is contractible, so that there is no non-trivial principal bundle.
As corollaries of the main construction, we show that:
\begin{enumerate}
    \item If the leaf $L$ is simply connected, then all possible formal singular foliations are simply classified by some finite-dimensional principal bundle (see Cor. \ref{coro:simply-connected-case}).
    \item If the transverse singular foliation is made of vector fields that vanish at least quadratically at $0$, then we show that formal singular foliations are the same as the data consisting of a Galois bundles over $L$ together with a group extensions of a certain type. Alternatively, it means that they are entirely given by the \pairing{} (see Cor.\  \ref{cor:transvQuadra}). 
    \item If a flat $\cF$-connection exists, then the \pairing{} can be lifted to a group morphism $\pi_1(L) \to \mathrm{Sym}(\cT_0)$, where $\mathrm{Sym}(\cT_0)$ is the group of symmetries of $\cT_0$. In this case, one can give an even more explicit description of the \data{} (see Prop.\ \ref{prop:FlatFoliationClassification}).
    \item As an application, we describe all possible formal singular foliations admitting a $2$-dimensional torus as a leaf (see Subsection \ref{subsec:Torus}) - a surprisingly difficult problem.
    \item We also give systematic answers to the miscellaneous questions in the introduction about concentric circles and spirals.
\end{enumerate}

\textbf{Conventions:}
\leavevmode\newline
$\bullet$
Lie groupoids are denoted by $\Gamma \rightrightarrows M $, with $M $ the set of units. Their source and target are denoted by $s$ and $t$, respectively. We use the short hand notation $\Gamma_m$, $\Gamma^n$, $\Gamma_m^{n} $ for $s^{-1}(m)$, $t^{-1}(n)$, $s^{-1}(m) \cap t^{-1}(n)$, respectively.

$\bullet$ A \emph{Galois cover} $ \tilde{L}$  over a connected manifold $\leaf $ is a $G$-principal bundle which satisfies the two following properties: \emph{(1)} it is connected, and \emph{(2)} the projection $\tilde{L} \to L $ is a local diffeomorphism. It is easily proved that this implies that $ G \simeq \pi_1(L)/K$ with $K$ a normal subgroup of the fundamental group $\pi_1(\leaf) $ and that $\tilde{L}= L^u/K $, with $L^u $ the universal cover of $L $. As a consequence, \emph{there is a one-to-one correspondence between Galois covers of $L$ and normal subgroups of $\pi_1(L) $}. It is convenient to denote the group $G=\pi_1(L)/K$ as $\pi_1(\tilde{L},L) $ and to call it the \emph{Galois group} of  $\tilde{L} $.  Notice that for any flat complete Ehresmann connection $\mathbb H $ on a bundle $X \to L $, the leaves of $\mathbb H$ (seen as a regular foliation on $X$) are Galois bundles over $L$. 

$ \bullet$ For $ K \subset H$ a normal Lie subgroup of a Lie group, and $R \to L $  an $ H/K$-principal bundle over $L$, we say that an $H$-principal bundle $P \to L $ such that $ P/K \cong R$ as principal bundles \begin{blue} (in particular, the natural map $P \to P/K$ is a morphism of principal bundles)\end{blue} is an \emph{extension of $R$ to an $H$-principal bundle} \cite{KirillExtension}.\footnote{Notice that a given $H/K$-principal bundle $R$ may have no extension to a $H$-principal bundle.
For instance, when $K $ is the center of $H $, there are obstructions in the \v{C}ech cohomology group ${H}^2(L,Z) $: This point is at the root of the theory of Abelian gerbes \cite{zbMATH00691244,zbMATH05306115}.}

\section{Formal singular foliations along a leaf}

\subsection{Formal neighborhoods and formal singular foliations}  

In this section, we introduce the notion of formal singular foliation along a leaf $\leaf$, following \cite{Ryvkin2}. 
The notion is justified by the fact that, as we will show, that any singular foliation on a manifold $M$ admitting $\leaf $  as a leaf induces one, namely its infinite jet along the leaf $\leaf$. Last, we equip (inner,  outer) symmetries of a formal singular foliation along $ \leaf$ with relevant diffeologies.

Let $\mathcal A $ be a commutative algebra and $\mathcal A[[x_1, \dots, x_d]] $ be the algebra of \emph{formal power series in $d$ variables}, \textit{i.e.}\ the algebra of formal sums
 $$   F = \sum_{i_1, \dots, i_d \geq 0} F_{i_1, \dots, i_d} x_1^{i_1} \cdots x_d^{i_d}   $$ 
with $ F_{i_1, \dots, i_d} \in \mathcal A $. This algebra is filtered by the decreasing sequence $ \left(\mathcal A_{\geq n}\right)_{n \in \mathbb N}  $ with $\mathcal A_{\geq n} $ being the ideal of formal functions of valuation $ \geq n$, \textit{i.e.}\ the ideal generated by monomials of 
degree $ \geq n$. It is a complete algebra with respect to the topology induced by this filtration.

We call a \emph{formal neighborhood  of a manifold $\leaf$} a sheaf of filtered commutative algebras 
$$ \mathcal U \mapsto C^{\mathrm{formal}}(\mathcal U) $$ over $ \leaf$ that satisfies the property that each point $\ell \in \leaf $ admits a neighborhood $\mathcal U \subset \leaf $ such that, as filtered algebras,
\begin{align}\label{CformalLocal}
C^{\mathrm{formal}}(\mathcal U) \simeq C^\infty(\mathcal U) [[x_1, \dots, x_d]]. 
\end{align}
Sections of this sheaf shall be called \emph{formal functions}. This sheaf is by assumption a sheaf of filtered algebras,  with respect to a decreasing filtration
given by the valuation of the formal power series. In other words, the $k$-th ideal of this filtration is made of formal functions which in any neighborhood $  \mathcal U$ as in \eqref{CformalLocal} read: 
$$ \sum_{i_1 + \dots +i_d \geq k} f_{i_1, \dots, i_d} x_1^{i_1} \dots x_d^{i_d}$$
with $f_{i_1, \dots, i_d}$ being smooth functions on~$\mathcal U$. We denote by $C^{\mathrm{formal}}_{\geq k} $  this ideal and call it the \emph{ideal of formal functions vanishing along $\leaf $ at order at least $k$}. In particular,  $\I_L := C^{\mathrm{formal}}_{\geq 1} $ is the kernel of the natural projection $C^{\mathrm{formal}}(\leaf) \longrightarrow C^\infty(\leaf)  $ and will be referred to as the \emph{ideal of formal functions vanishing along $\leaf $}.

\begin{blue}
\begin{example}
\label{ex:obvious}
\normalfont
For every manifold $\mathcal U$ (later on, it will be an open subset of the manifold $ \leaf$, hence the notation), and every integer $d$, the sheaf of formal functions of the form $
\sum_{i_1 + \dots +i_d \geq 0} f_{i_1, \dots, i_d} x_1^{i_1} \dots x_d^{i_d}$
with $f_{i_1, \dots, i_d}$ being smooth functions on~$\mathcal U$ is a formal neighborhood of $ \mathcal U$. We call it the \emph{trivial formal neighborhood}.
\end{example}
\end{blue}
 
\begin{example}
\label{decomp:step1}
\normalfont
For any embedded submanifold $\leaf$ of a manifold $ M$, the algebra of infinite jets of smooth functions along $\leaf $
(\textit{i.e.}\ the adic completion of $C^\infty(M) $ with respect to the ideal of functions vanishing along $\leaf $)
is the space of global sections of a formal neighborhood of $\leaf $.
\end{example}

Provided that $\leaf $ is connected, the integer $d$ is constant along $\leaf $, and the quotient space $ C^{\mathrm{formal}}_{\geq 1}/C^{\mathrm{formal}}_{\geq 2} $ is (by Serre-Swann theorem) the sheaf of sections of a vector bundle of rank $d$. We denote its dual bundle by $ \vertical$ and call it the \emph{normal bundle}.

\begin{example}
\label{decomp:step2}
\normalfont
For the formal neighborhood of Example \ref{decomp:step1}, the normal bundle $\normal$ is the usual normal bundle $\mathrm{T}M|_{\leaf} / \mathrm{T} \leaf $.
\begin{blue}For Example \ref{ex:obvious} it is the trivial vector bundle of rank $d$.\end{blue}
\end{example}

Example \ref{decomp:step1} is in fact "generic", in view of the following classical lemma:

\begin{lemma}
\label{lem:normalstructure}
\cite{Iglesias-Zemmour}
Any formal neighborhood of a connected manifold $\leaf $  is isomorphic (as a sheaf of filtered algebra) to infinite jets along $\leaf $  of smooth functions on the normal bundle $\vertical$.
\end{lemma}

\begin{deff}
\label{def:normalstructure}
We call a \emph{formal tubular neighborhood} an isomorphism as in Lemma \ref{lem:normalstructure}.
\end{deff}

\begin{rem}
\normalfont
Although the normal bundle $\vertical$ is canonically associated to a formal neighborhood
$C^{\mathrm{formal}} $, the filtered algebra isomorphism in the lemma \ref{lem:normalstructure} above is \emph{not} canonical.
\end{rem}

\begin{example}
\label{decomp:step2.5}
\normalfont
For the formal neighborhood of Example \ref{decomp:step1}, the choice of a tubular neighborhood, \textit{i.e.}\ a diffeomorphism between a neighborhood of $\leaf $ in $M $ and a neighborhood of the zero section in $\rmT M|_\leaf/\rmT\leaf $, induces a tubular neighborhood in the sense of Definition \ref{def:normalstructure}, i.e. a sheaf isomorphism as in Lemma \ref{lem:normalstructure}. 
\end{example}

We call \emph{vector field on a formal neighborhood $C^{\mathrm{formal}}$ of $\leaf $} a derivation of the sheaf $C^{\mathrm{formal}} $. The  space $\mathfrak X^{\mathrm{formal}} $  of all vector fields is a sheaf of Lie algebras and of  $C^{\mathrm{formal}}$-modules - both structures being related by the usual Leibniz identity:
 $$ [X,FY] =F [X,Y] + X[F] \, Y $$
 for all $F \in C^{\mathrm{formal}} $, and $X,Y \in  \mathfrak X^{\mathrm{formal}}$.
\begin{example}
\label{decomp:step3}
\normalfont
For the formal neighborhood of Example \ref{decomp:step1}, any vector field on $M$ induces a formal vector field, i.e. its infinite jet along $L$, and there is a natural Lie algebra morphism 
$\mathfrak X(M) \longrightarrow \mathfrak X^{\mathrm{formal}} $. 
\end{example}
 
 We say that a formal vector field $X \in  \mathfrak X^{\mathrm{formal}} $ is \emph{tangent to $L$} if it preserves the ideal $I_L = C^{\mathrm{formal}}_{\geq 1}$. Such a vector field can be evaluated at a point $\ell \in \leaf $ to give an element in the tangent space $T_\ell \leaf $ that we denote by $X_{\ell} $.

\begin{example}
\label{decomp:step4}
\normalfont
For the formal neighborhood of Example \ref{decomp:step1}, 
any vector field $X$ on $M$ tangent to $\leaf$ induces a formal vector field tangent to $\leaf $ and its induced element in $X_\ell \in T_\ell \leaf  $ coincides with the value $X(\ell) \in T_\ell L $ of $X  $ at $ \ell$. 
\end{example}

Let us define singular foliations in this context.

\begin{deff}
\label{def:formalsf}
Let $\leaf$ be a connected manifold.
We call a \textit{formal singular foliation on a formal neighborhood $ C^{\mathrm{formal}}$ of the manifold $ \leaf$} a subsheaf $ \mathcal F$ of $\mathfrak X^{\mathrm{formal}} $ that satisfies the three following conditions:
\begin{enumerate}
\item[1)] $ \mathcal F$ is stable under Lie bracket, 
\item[2)] $ \mathcal F$ is stable under multiplication by $C^{\mathrm{formal}} $,
\item[3)] $\cF$ is a locally finitely generated $C^{\mathrm{formal}} $-module.
\end{enumerate}
When, in addition, 
\begin{enumerate}
\item[a)]  every $ X \in \mathcal F$ is tangent to $\leaf $
\item[b)] for every $u \in T_\ell \leaf$, there exists $ X \in \mathcal F$ such that $X_\ell =u $,  
\end{enumerate} 
then we speak of a \emph{formal singular foliation along the leaf $\leaf $}.
\end{deff}
\begin{rem}
\normalfont
In fact, in the definition of a formal singular foliation along the leaf $\leaf$, the assumption "locally finitely generated" could be dropped: conditions 1) + 2) + a) + b) imply condition 3) in Definition \ref{def:formalsf}. This follows from the fact that formal power series form a Noetherian ring \cite{MR240826}.
\end{rem}

\begin{example}
\label{decomp:step5}
Consider the formal neighborhood of $ \leaf$ in a manifold $M$ of Example \ref{decomp:step1}. Any singular foliation $ \mathcal F$ (defined as in Hermann \cite{Hermann}, Dazord, Stefan, Sussmann \cite{Dazord,Stepan1,Stepan2,Sussmann1}, or non-commutative geometers, \textit{e.g.}\ Androulidakis, Debord, Mohsen, Skandalis, Yuncken, Zambon \cite{AS,Debord2,AMY,MZ16} as a locally finitely generated sub-sheaf of the module of vector fields stable under Lie bracket, see \cite{LLR} for a review of the matter) induces a singular foliation on the corresponding formal neighborhood of an embedded submanifold $ \leaf$. If $\leaf$ is a leaf of $\mathcal F $, then this induced formal singular foliation is a formal singular foliation along the leaf $\leaf$.
 \end{example}

\begin{example}
\label{ex:ppalbundleSF}
\normalfont
Here is an example of formal singular foliations admitting a given manifold $L $ as a leaf. Given the following pair:
\begin{enumerate}
    \item a finite-dimensional Lie group $G$  acting on $\mathbb{R}^d$. We do not assume that the action is by linear maps, but we assume that   $0$ is a fixed point.
    \item A principal $G$-bundle $Q$ over $L$.
\end{enumerate}
Consider the associated vector bundle over $ \leaf$ given by
 $$ T:=(Q \times \mathbb{R}^d)/G.$$
The manifold $T$ comes with a natural singular foliation that we now describe in two different ways. It may be seen as the unique singular foliation on $T$ whose pull-back on $Q \times \mathbb{R}^d$ is the direct product of the singular foliation $\mathfrak X(Q) $ of all vector fields on $Q$ with the singular  foliation associated with the $G$-action on $ \mathbb R^d$. It can also be seen as follows: There is a natural action of the gauge groupoid
 $$ \frac{Q\times Q}{G} \rightrightarrows L  $$
 on $T$. Its infinitesimal  action is an action of the Atiyah algebroid $ \rmT Q/G \to L$ (= the Lie algebroid of the gauge groupoid) on $T$.  
 The anchor map for this Lie algebroid action equips $T$ with a singular foliation. Since the zero section is an orbit of the gauge groupoid action, $L$ is a leaf of this singular foliation.
Now, if one is initially given a \emph{formal} action of $G$ on $\mathbb R^d $  preserving $0$, the construction goes through and yields a  \emph{formal} singular foliation along the leaf $L $. 
 Since the action of the isotropies of the gauge groupoid on fibers of $T$ are isomorphic to the action of $G$ on $\mathbb R^d $, the transverse singular foliation (to be defined below) is the one associated to the infinitesimal action of the Lie algebra $\mathfrak g $ of $G$ on $\mathbb R^d $.
\end{example}

\begin{example}\label{ex:HopfFibration}
\normalfont
An example of the recipe above is provided by \cite[\S 7, Ex.\ 7.9]{SRFCYM} (based on the example provided \cite[Example 7.3.20; page 287]{MR2157566}) with the help of the Hopf fibration. Consider
\begin{itemize}
	\item $L \coloneqq \mathbb{S}^4$
	\item $G = \mathbb{S}^3 \cong \mathrm{SU}(2)$
	\item $Q$ the Hopf fibration $\mathbb{S}^7 \to \mathbb{S}^4$.
\end{itemize}
Then $G$ acts faithfully on $\mathbb{C}^2$, preserving $0$, so that, following the construction of Example \ref{ex:ppalbundleSF}, we have a natural foliation $\cF$ on $T \coloneqq (Q \times \mathbb{C}^2)/H$ with transverse model $(\mathbb{C}^2, \cT_0)$, where $\cT_0$ consists of concentric spheres.

\end{example}

\begin{blue}
\begin{example}\label{ex:trivialsingularfoliation}
\normalfont
For any manifold $\mathcal U$, any integer $d$, and
any formal singular foliation $\mathcal T_0 $ near $0 \in \mathbb R^d $ whose generators $ X_1, \dots, X_r$ are assumed to vanish at $0$,  we call  \emph{trivial singular foliation on $\cU$  with transverse model $ \mathcal T_0$} the formal singular foliation defined on the trivial formal neighborhood of Example \ref{ex:obvious}, whose sections are made of formal vector fields of the form $ \sum_{i=1}^s g_i Y_i+\sum_{j=1}^r f_j X_j $
with $ Y_i$ smooth vector fields on $\cU$,and $ f_i,g_j$ formal functions on the trivial neighborhood. For $ \cU = \mathbb R^s$ 
with variables $ y_1, \dots, y_s$ it is generated by the vector fields 
 $$ \left(\frac{\partial}{\partial y_1}, \dots, \frac{\partial}{\partial y_s}, X_1, \dots, X_r \right) $$
 as a module over formal functions.
 \end{example}
\end{blue}

\subsection{Formal diffeomorphisms}\label{sec:FormalDiffeo} \cite{MR0212575}
Let us now work on $\mathbb{R}^d$ with formal functions $\mathbb R[[x_1 , \dots, x_d]] $ in $d$ variables. We recall a few facts.
First, for formal vector fields, the exponential map 
 $$ {\mathrm{exp}} : \mathfrak X^{\mathrm{formal}}_0 \longrightarrow {\mathrm{Diff}}^{\mathrm{formal}} \mleft(\mathbb R^d, 0\mright) $$
 is well-defined and maps the Lie algebra $\mathfrak{X}^{\mathrm{formal}}_0$ of formal vector fields vanishing at zero to the group $\mathrm{Diff}^{\mathrm{formal}}\mleft(\mathbb R^d,0\mright) $ of \emph{formal diffeomorphisms} of $\mathbb R^d $ fixing $0 \in \mathbb R^d$. Recall that the latter is by construction the group of filtered algebra isomorphisms of the ring of formal functions. 
This latter group admits a filtration by normal subgroups:
  $$ \cdots \hookrightarrow  \mathrm{Diff}^{\mathrm{formal}}\mleft(\mathbb R^d,0\mright)_{\geq 3}  \hookrightarrow  \mathrm{Diff}^{\mathrm{formal}}\mleft(\mathbb R^d,0\mright)_{\geq 2}  \hookrightarrow \mathrm{Diff}^{\mathrm{formal}}\mleft(\mathbb R^d,0\mright) $$
where $\mathrm{Diff}^{\mathrm{formal}}\mleft(\mathbb R^d,0\mright)_{\geq k} $ consists of formal diffeomorphisms $\Phi $ that coincide at $0$ with the identity map up to order $k $ at least, \textit{i.e.}\ 
$\Phi (F) -F \in \mathbb R[[x_1, \dots, x_d]]_{\geq k} $ for all $ F \in \mathbb R[[x_1, \dots, x_d]]$.
The quotients of 
$ \mathrm{Diff}^{\mathrm{formal}}\mleft(\mathbb R^d,0\mright)/\mathrm{Diff}^{\mathrm{formal}}\mleft(\mathbb R^d,0\mright)_{\geq k} $ are finite-dimensional manifolds (of $k$-jets) and are therefore Lie groups. 
Last, for any $\Phi \in \mathrm{Diff}^{\mathrm{formal}}\mleft(\mathbb R^d,0\mright)$, the map:
\begin{align*}
\Phi_*: \mathfrak X_0^{\mathrm{formal}}  &\longrightarrow   \mathfrak X_0^{\mathrm{formal}}\\ X & \mapsto   \Phi^{-1} \circ X  \circ \Phi
\end{align*}
is a Lie algebra isomorphism.

\begin{blue}

There is a natural diffeology  \cite{Iglesias-Zemmour}  on $\mathrm{Diff}^{\mathrm{formal}}\mleft(\mathbb R^d,0\mright)$: plots are maps to $\mathrm{Diff}^{\mathrm{formal}}\mleft(\mathbb R^d,0\mright)$ whose projections on the Lie group $ \mathrm{Diff}^{\mathrm{formal}}\mleft(\mathbb R^d,0\mright)/\mathrm{Diff}^{\mathrm{formal}}\mleft(\mathbb R^d,0\mright)_{\geq k} $  are smooth for all $k$. There is also a natural diffeology on $ \mathfrak X_0^{\mathrm{formal}}$: Plots are  maps from $ \mathbb R^n$ 
valued on  $ \mathfrak X_0^{\mathrm{formal}}$, and are  expressions of the forms:
\begin{equation}\label{smoothinX0} g(y_1, \dots,y_n) := \sum_{\substack{i_1 + \dots +i_d \geq 0 \\ i=1, \dots, d}} f^i_{i_1, \dots, i_d}(y_1, \dots,y_n) ~ x_1^{i_1} \dots x_d^{i_d} \frac{\partial}{\partial x_i} 
\end{equation}
for some smooth real-valued functions $
 f^i_{i_1, \dots, i_d}(y_1, \dots,y_n)$ defined in a neighborhood of $0$ in $ \mathbb R^n$ and such that $f^i_{0, \dots, 0}$ vanish at~$0$. 

\begin{rem}
\label{rem:1cocycles}
Now that we have a diffeology on $\mathrm{Diff}^{\mathrm{formal}}\mleft(\mathbb R^d,0\mright)$, one can state the following result: Given a manifold $L$, there is a one-to-one correspondence between formal neighborhoods of $L$ and principal $\mathrm{Diff}^{\mathrm{formal}}\mleft(\mathbb R^d,0\mright)$-bundles over $L$. In particular, a formal neighborhood of $L$ is induced by any $1$-cocycle, i.e.\ by the data of an open covering $ (U_{i})_{i\in I}$ of $L$ and smooth maps $g_{ij} \colon U_{ij}\to \mathrm{Diff}^{\mathrm{formal}}\mleft(\mathbb R^d,0\mright) $ such that $ g_{ij} \circ g_{jk}=g_{ik}$ and $g_{ij} = \mleft( g_{ji} \mright)^{-1}$, where $U_{ij} \coloneqq U_i \cap U_j$.
To such cocycles $g_{ij}\colon U_{ij}\to \mathrm{Diff}^{\mathrm{formal}}\mleft(\mathbb R^d,0\mright)$ we associate the formal neighborhood whose global formal sections are elements in $ \coprod_{i\in I} C^\infty(U_i)[[x_1,\dots,x_d]]$ such that for any pair of indices $i,j\in I$ the restrictions to $U_{ij}$ of the components in $C^\infty(U_i)[[x_1,\dots,x_d]]$ and $C^\infty(U_j)[[x_1,\dots,x_d]]$ correspond one to the other through the formal diffeomorphism $g_{ij}$. The diffeology has been precisely chosen so that this construction makes sense. 
\end{rem}

\begin{lemma}
\label{lem:form}
Any smooth map $f$ from a manifold $M$ to $\mathrm{Diff}^{\mathrm{formal}}\mleft(\mathbb R^d,0\mright)$ admits, near every point $m \in M$, a coordinate neighborhood with coordinates  $(y_1, \dots,y_n)$ centered at $m$ on which it is  of the form
\begin{equation}
\label{eq:shqpeExpDiffeology} f (y_1, \dots,y_n)=f(0,\dots,0)  \circ  {\mathrm{exp}} \left(g(y_1, \dots,y_n) \right) 
\end{equation}
for some smooth map $g(y_1, \dots,y_n) $ as in Equation \eqref{smoothinX0} with smooth real-valued functions $
 f^i_{i_1, \dots, i_d}(y_1, \dots,y_n)$ defined in a neighborhood of $0$ in $ \mathbb R^n$ and such that $f^i_{0, \dots, 0}(0)=0$. These functions are unique on some neighborhood of $ 0$. 
 \end{lemma}
 \begin{proof}
 The map $ {\mathrm{exp}} $ has a formal inverse at $0$ that we denote by ${\mathrm{ln}} $. The desired expression is then obtained as  $${\mathrm{ln}}\mleft( f(0,\dots,0)^{-1}  \circ f (y_1, \dots,y_n)\mright).$$
 It is of the desired form by construction. This completes the proof.
 \end{proof}
\end{blue}

\subsection{Formal singular foliations near a point}\label{sec:T000}
Consider a \emph{formal singular foliation $\mathcal T_0 $ near $0 \in \mathbb R^d $}, \textit{i.e.}\ a sub-module of $ \mathfrak X^{\mathrm{formal}} $ stable under Lie bracket and made of vector fields that vanish at $0$.
Since formal functions are a Noetherian ring, every sub-module is in fact finitely generated.
 Hence, $\mathcal T_0 $ is
 generated over $ \mathbb R[[x_1, \dots,x_d]]$ by formal vector fields $ X_1, \dots, X_r$ on $ \mathbb R^d$, 
 and there exists $ (h_{ij}^k)_{i,j,k=1,\dots,r} \in  \mathbb R[[x_1, \dots,x_d]] $ such that:
 \begin{equation}
 \label{eq:coeffs}
     [X_i,X_j]=\sum_{k=1}^r h_{ij}^k X_k
 \end{equation} 
 for all $i,j =1, \dots, r$. \begin{blue}
To start with, let us state a lemma which is only true because we consider formal singular foliations and would not be true for smooth singular foliations (see \cite{LLR}, Chapter 1, Remark 1.3.3., where a counter-example, provided by Georges Skandalis, is given). This lemma describes the natural diffeology of $ \mathcal T_0$ induced by the one of $\mathfrak{X}^{\mathrm{formal}}_0$.

\begin{lemma}\label{lem:diffeoT0}
Any smooth map $f$ from a manifold $M$ to $\mathfrak{X}^{\mathrm{formal}}_0$ valued in  $\mathcal T_0 $  admits, near every point $m$, a coordinate neighborhood with coordinates  $(y_1, \dots,y_n)$ centered at $m$ on which it is  of the form
$$ f (y_1, \dots,y_n)=   \sum_{\substack{i_1 + \dots +i_d \geq 0 \\ i=1, \dots, r}} f^i_{i_1, \dots, i_d}(y_1, \dots,y_n) x_1^{i_1} \dots x_d^{i_d}X_i   $$
for some smooth real-valued functions $
 f^i_{i_1, \dots, i_d}(y_1, \dots,y_n)$ defined in a neighborhood of $0$ in $ \mathbb R^n$. 
 Here $X_1, \dots,X_r $ are any set of generators of $ \mathcal T_0$.
\end{lemma}
\begin{proof}
The functions $
 f^i_{i_1, \dots, i_d}(y_1, \dots,y_n)$ are constructed by a sort of recursion on  $i_1 + \dots +i_d$, but with a subtlety, that involves the Artin-Rees lemma.  Let $ \mathcal I$ be the maximal ideal of formal functions with $n$ variables vanishing at the origin. 
For every $k$, let $p_k$ stand for the projection 
$$ p_k :  \mathfrak{X}^{\mathrm{formal}} \longrightarrow \frac{ \mathfrak  X^{\mathrm{formal}}}{  \mathcal I^{k+1}  \mathfrak  X^{\mathrm{formal}} }.$$
Since $ p_k \circ f$ belongs for all $ (y_1, \dots, y_n)$ to the subspace generated by $$  \left\{ p_k \mleft(x_1^{i_1} \dots x_d^{i_d}X_i\mright) ~ \mid ~ i_1 + \dots + i_d \leq k, ~ i=1, \dots, r \right\}, $$ 
there  exists smooth functions $  f^{(k),i}_{i_1, \dots, i_d}(y_1, \dots,y_n)$ 
 for $  0 \leq i_1 + \dots +i_d \leq k $ and $ i=1, \dots, r$ such that
 \begin{equation}
 \label{eq:orderk}
f (y_1, \dots,y_n)-   \sum_{\substack{0 \leq i_1 + \dots +i_d \leq k \\ i=1, \dots, r}} f^{(k),i}_{i_1, \dots, i_d}(y_1, \dots,y_n)\,  x_1^{i_1} \dots x_d^{i_d}X_i   
\end{equation}
is valued in  $ \mathcal I^{k+1}   \mathfrak X^{\mathrm{formal}} \cap \cT_0$. To prove the claim , we need to take the limit as $ k $ goes to infinity of the previous sums.
However, since the function $  f^{(k),i}_{i_1, \dots, i_d} $  depend on $k$, we are not allowed to. To be allowed to, we need to make sure that the sequence of function $ \left( f^{(k),i}_{i_1, \dots, i_d} \right)_{k\geq 0}$ is constant after a certain rank for all given $i_1, \dots,i_d,i$. Here is the point where we need to use Artin-Rees lemma, which applies since formal functions form a Noetherian ring, and states that there exists a natural number $\nu$ such that for all $k\geq \nu$:
 $$ \mathcal I^{k} \mathfrak X^{\mathrm{formal}} \cap \mathcal T_0 \subset \mathcal I^{k  - \nu }\mathcal T_0 .$$
 The idea now consists in constructing the functions $ f^{(k),i}_{i_1, \dots, i_d} $ by recursion as follows. 
 First, we choose of all of them to be zero for $ k=0$. Assume that we have constructed  $ f^{(k),i}_{i_1, \dots, i_d} $  that satisfy \eqref{eq:orderk} for a given $k$. Then the expression given in \eqref{eq:orderk}  belongs to $\mathcal I^{k+1}   \mathfrak X^{\mathrm{formal}}\cap \mathcal T_0$ by construction. In view of Artin-Rees Lemma, it belongs to $\mathcal I^{k  +1- \nu }\mathcal T_0$ and therefore reads 
as $$  \sum_{\substack{0 \leq i_1 + \dots +i_d \leq k+1 \\ i=1, \dots, r}} g^{(k),i}_{i_1, \dots, i_d}(y_1, \dots,y_n)\,  x_1^{i_1} \dots x_d^{i_d}X_i  + \mathcal I^{k+2} \mathfrak X^{\mathrm{formal}} \cap \cT_0 $$
for some smooth functions $ g^{(k+1),i}_{i_1, \dots, i_d}(y_1, \dots,y_n)$
which are identically equal to $0$  provided $  i_1 + \dots + i_d \leq k-\nu $. We also define them to be zero for 
 $  i_1 + \dots + i_d \geq k+2 $ by convention.

We can now define $$ f^{(k+1),i}_{i_1, \dots, i_d}(y_1, \dots,y_n) :=  f^{(k),i}_{i_1, \dots, i_d}(y_1, \dots,y_n)+  g^{(k),i}_{i_1, \dots, i_d}(y_1, \dots,y_n).$$
 Equation \eqref{eq:orderk} is now valid as order $ k+1$ instead of $k$. The sequence
 $$  \left( f^{(k),i}_{i_1, \dots, i_d}(y_1, \dots,y_n) \right)_{k \geq 0} $$ 
is constant for $ k \geq i_1 + \dots + i_d +\nu +1 $. This completes the proof.
\end{proof}
\end{blue}

\subsection{Inner and outer symmetries}\label{sec:InnerAndOuterSym}
To a formal singular foliation $ \mathcal T_0$ near $0$ on $ \mathbb R^d$, as in Section \ref{sec:T000}, we first associate three Lie algebras. Let $\mathrm{sym} (\mathcal T_0)$ be the Lie subalgebra of $ \mathfrak X_0^{formal}$ of all \emph{infinitesimal symmetries of $\mathcal T_0 $}, \textit{i.e.}\ formal vector fields $X$ on $\mathbb{R}^d$ vanishing at $0$ such that
$  
[X , \mathcal T_0 ]  \subset \mathcal T_0
$. Since $\mathcal T_0 $ is closed under the Lie bracket, the Lie algebra $\mathrm{sym} (\mathcal T_0)$ contains $ \mathcal T_0$ as a Lie ideal. We denote by $\mathrm{out}(\mathcal T_0)$ the quotient Lie algebra:
 $$ \xymatrix{\mathcal T_0 \ar@{^(->}[r] & \mathrm{sym} (\mathcal T_0) \ar@{->>}[r] &  \mathrm{out} (\mathcal T_0)  }  .$$

 To these Lie algebras, there correspond three groups (as in \cite{zbMATH01867165}), which are of a crucial importance in this paper and which we now introduce.
We call \emph{symmetries of $ \mathcal T_0$}
 the sub-group of all $\Phi \in  \mathrm{Diff}^{\mathrm{formal}}\mleft(\mathbb R^d,0\mright)$ that preserve $\mathcal T_0 $, \textit{i.e.}\ such that $\Phi_* (\mathcal T_0) = \mathcal T_0$.
  We denote this group by $\mathrm{Sym}(\mathcal T_0) $.

Now, it is a non-obvious theorem that for any $X \in \mathcal T_0$, $ {\mathrm{exp}}(X) \in  \mathrm{Sym}(\mathcal T_0)$; see \textit{e.g.}\ the chapter about the existence of leaves in \cite{LLR} or Proposition 1.6 in \cite{AS}, or the initial reference \cite{Hermann}. Of course, the given references have to be adapted to the formal case, which is routine. 
 
\begin{enumerate}
\item 
We call \emph{inner symmetries of $ \mathcal T_0$} the sub-group of 
$\mathrm{Sym}(\mathcal T_0) $ generated by $\mathrm{exp} (X)$ with $X \in \mathcal T_0$ (this group is considered in \cite{GV,AZ2}). 
We denote this group by ${\mathrm{Inner}}(\mathcal T_0) $.
 \item Since inner symmetries form a normal subgroup of the subgroup $\mathrm{Sym}(\mathcal T_0) $ of symmetries of $\mathcal T_0 $ (see \cite{LLR}, chapter 2),    their quotient  is a group that we call \emph{outer symmetries of $\mathcal T_0 $}.  We denote it by $\mathrm{Out}(\mathcal T_0) $.
 \end{enumerate}

We have by construction a commutative diagram as follows:
 $$ \xymatrix{\mathcal T_0 \ar@{^(->}[r] \ar[d]^{\mathrm{exp}} & \mathrm{sym} (\mathcal T_0) \ar@{->>}[r] \ar[d]^{\mathrm{exp}} & \ar[d]^{}   \mathrm{out} (\mathcal T_0)  \\
 \mathrm{Inner} (\mathcal T_0) \ar@{^(->}[r] & \mathrm{Sym} (\mathcal T_0) \ar@{->>}[r] &  \mathrm{Out} (\mathcal T_0) }  .$$
 Also, the upper line can be considered as the tangent space at the identity map of the lower line. To state it,  one needs a diffeology on these groups - that we now introduce.
The natural diffeology \cite{Iglesias-Zemmour} on $\mathrm{Diff}^{\mathrm{formal}}(\mathbb R^d,0) $ was already defined as follows: the plots are maps from $\mathcal V \subset \mathbb R^m $ to $\mathrm{Diff}^{\mathrm{formal}}(\mathbb R^d,0) $ whose image through the quotient in $\mathrm{Diff}^{\mathrm{formal}}(\mathbb R^d,0)/ \mathrm{Diff}^{\mathrm{formal}}(\mathbb R^d,0)_{\geq k} $ are ordinary smooth maps for all $k \geq 2$. There is an induced diffeology on both $\mathrm{Inner}(\mathcal T_0)$ and $\mathrm{Sym}(\mathcal T_0) $. The plots are maps from $\mathcal V $ to these subgroups which are smooth when composed with the natural inclusions into $\mathrm{Diff}^{\mathrm{formal}}(\mathbb R^d,0) $.
Last, for the group ${\mathrm{Out}}(\mathcal T_0) $ of outer symmetries, there is a natural diffeology whose plots are maps which are, locally, the quotients of plots valued in ${\mathrm{Sym}}(\mathcal T_0) $.

\begin{example}\label{ex:ConcentricCirclesInnAndOut}
\normalfont
Consider the formal singular foliation  $\mathcal T_0 $ on $\mathbb R^2 $ (called "concentric circles" in the introduction) generated by $ x \frac{\partial}{\partial y}- y \frac{\partial}{\partial x} $. 
In this case, $\mathrm{Inner}(\mathcal T_0) $ is the group of all formal diffeomorphisms $\Phi $ that preserve the distance to the origin,  \textit{i.e.}\ that satisfy
  $$ \Phi ( x^2+y^2 )=x^2+y^2,  $$
  and with positive determinant of the tangent map/total derivative $\rmD\Phi $ at $(0, 0)$.
  The Lie algebra ${\mathrm{sym}}(\mathcal T_0) $ also contains the Euler vector field ${\mathrm{Eul}}:=x \frac{\partial}{\partial x}+ y \frac{\partial}{\partial y} $ and all formal fields of the type: 
   $f(x^2+y^2) {\mathrm{Eul}}$ with $f$ a formal function in one variable. The group $\mathrm{Sym}(\mathcal T_0)  $ is   the group of all formal diffeomorphisms $\Phi $ that  satisfy
  $$ \Phi ( x^2+y^2 )=g_\Phi(x^2+y^2 ) $$
  for some formal diffeomorphism $g_\phi$ of $\mathbb R$ preserving $0$ and with positive derivative at $0$. If $g_\Phi (t) =t$, and if the determinant of the total derivative $\rmD\Phi $ at $(0, 0)$ is positive, then we recover the definition of inner symmetries. The quotient group $ \mathrm{Out}(\mathcal T_0) $ is therefore made of all formal diffeomorphisms of $\mathbb R $ preserving~$0$ and with positive derivative at the origin, times $\mathbb Z/2\mathbb Z $. That is, ${\mathrm{Out}}(\mathcal T_0)= {\mathrm{Diff}^{\mathrm{formal}}}(\mathbb R,0)_+ \times \mathbb Z/2\mathbb Z $.  
 Geometrically, this can be understood as follows. The leaves of the singular foliation $\mathcal T_0 $ are concentric circles plus the origin. Inner symmetries are those mapping each one of these circles to itself without changing the orientation, while symmetries allow to exchange circles, and to exchange the global orientation of these circles. 
\end{example}

There is also a natural family of sub-groups of $\mathrm{Inner}(\mathcal T_0)  $ to consider.
To start with, for every $k \geq 2$, let  $(\mathcal T_0)_{\geq k}   $ be the Lie subalgebra of $\mathcal T_0 $ made of formal vector fields in $ \mathcal T_0$ that vanish at order $\geq k$ at $0$. 
This Lie algebra admits a natural diffeology: the plots are maps whose image in $(\mathcal T_0)_{\geq k}/  (\mathcal T_0)_{\geq l}$ is an ordinary smooth map for all $l \geq k$.  
Let $\mathrm{Inner}(\mathcal T_0)_{\geq k}$ be the sub-group of  $\mathrm{Inner}(\mathcal T_0)$ generated by $\mathrm{exp}\left((\mathcal T_0)_{\geq k}\right) $.  It inherits a diffeology from its ambient space. 

\begin{example}
\normalfont
\label{ex:outis}
Let $\mathcal T_0 $ be now the singular foliation made of \emph{all} formal vector fields on $ \mathbb R^d$ that vanish at $0 $, 
$ {\mathrm{Sym}}((\mathcal T_0)_{\geq k})$ is then the group $ {\mathrm{Diff}}^{\mathrm{formal}}(\mathbb R^d, 0)_{\geq k}$.
${\mathrm{Inner}}(\mathcal T_0)$ is made of formal diffeomorphisms of $\mathbb R^d $ which are the products of flows of vector fields vanishing at $0$. Those are exactly formal diffeomorphisms whose differential at $0$ has a positive determinant.
For $ k\geq 2$, notice that for two formal diffeomorphisms $\phi,\psi$ fixing $0$ and having the same $k-1$-jet,  $\phi \circ \psi^{-1} $ is a diffeomorphism that coincides with the identity up to order $k$. Hence ${\mathrm{Inner}}(\mathcal T_0)_{\geq k}$ is in the image of $(\mathcal T_0)_{\geq k} $ through the exponential map by Lemma \ref{lem:exp2} below. This implies that $\mathrm{Out}((\mathcal T_0)_{\geq k})$ is the group of $(k-1)$-jets of diffeomorphisms fixing a point in $ \mathbb R^d$. In particular, for $ k=2$, $\mathrm{Out}((\mathcal T_0)_{\geq 2})$ is the group ${\mathrm{GL}}_d(\mathbb R) $.
\end{example}

The existence of a diffeology allows to state results as this lemma.

\begin{lemma} \label{lem:exp2}
For every $k \geq 2$, the group $\mathrm{Inner}(\mathcal T_0)_{\geq k} $ is contractible, and the exponential map $ (\mathcal T_0)_{\geq k} \to \mathrm{Inner}(\mathcal T_0)_{\geq k}$ is a bijection compatible with the respective diffeologies of these sets.
\end{lemma}
\begin{proof}
The exponential map, restricted to formal vector fields vanishing at least quadratically at $0$, is injective and surjective onto ${\mathrm{Diff}}^{\mathrm{formal}}(\mathbb R^d,0)_{\geq 2} $.  Hence, the restriction of the exponential map to $ (\mathcal T_{0})_{\geq k}$ is bijective onto its image, which is  $\mathrm{Inner}(\mathcal T_0)_{\geq k}$. In particular, they are contractible. 
\end{proof}

\subsection{Principal bundles over symmetries of a singular foliation}

In this section, $\mathcal T_0$ is a formal singular foliation near $0$ as in the previous section, and we are given a discrete group $K$ and an injective group morphism $\mathfrak i \colon K \to {\mathrm{Out}}(\mathcal T_0)  $.
Let us associate to it an exact sequence
\begin{equation}
\label{eq:defH}
\begin{aligned}
\xymatrix{ \mathrm{Inner} (\mathcal T_0) \ar@{^(->}[r] \ar[d]^{=} &  \ar@{^(->}[d] H \ar@{->>}[r] & K \ar@{^(->}[d]^{\mathfrak i}\\ 
 \mathrm{Inner} (\mathcal T_0) \ar@{^(->}[r] & \mathrm{Sym} (\mathcal T_0) \ar@{->>}[r] &  \mathrm{Out} (\mathcal T_0) }  
\end{aligned}
\end{equation}
where $H $ is given as the inverse image of $\mathfrak i (K) $ through the natural projection \linebreak $\xymatrix{\mathrm{Sym} (\mathcal T_0) \ar@{->>}[r] &  \mathrm{Out} (\mathcal T_0).}$
Such a group $H$ comes again with a natural diffeology, given by the inclusion into $\mathrm{Sym} (\mathcal T_0) $. Since $K$ is discrete,  the connected components of the diffeology of $H$  are naturally isomorphic to $\mathrm{Inner}(\mathcal T_0) $.  

We are going to state a lemma about the diffeology of $H$.

\begin{blue}
\begin{lemma}
\label{lem:diffeologyH}
The diffeology of $H$, induced by the one of  $\mathrm{Sym} (\mathcal T_0) $,  is such that
any smooth map $g$ from a manifold $M$ to $H$ admits, near every point $m$, a coordinate neighborhood with coordinates  $(y_1, \dots,y_n)$ centered at $m$ on which it is  of the form
\begin{equation}
\label{eq:gy1yr}
g (y_1, \dots,y_n)=g(0,\dots,0)  \circ  {\mathrm{exp}} \left(  \sum_{\substack{i_1 + \dots +i_d \geq 0 \\ i=1 \dots, r}} f^i_{i_1, \dots, i_d}(y_1, \dots,y_n) \, x_1^{i_1} \dots x_d^{i_d}  \,\, X_i   \right) 
\end{equation}
for some smooth real-valued functions $
 f^i_{i_1, \dots, i_d}(y_1, \dots,y_n)$
 defined in that neighborhood and vanishing at $0$.  
 \end{lemma}
 \begin{proof}
 Since $K$ is discrete, it suffices to prove the lemma for $ H=\mathrm{Inner} (\mathcal T_0) $. 
 In this case, the lemma is an obvious consequence of Lemma 
 \ref{lem:diffeoT0} and of the fact that the exponential map  $\mathcal T_0 \to \mathrm{Inner} (\mathcal T_0)  $ is invertible near $0$.
 \end{proof}
\end{blue}

Now, let us state a Proposition about $H$-principal bundles.
For every $k \geq 2$, the group $\mathrm{Inner}(\mathcal T_0)_{\geq k}$
is a normal subgroup of $\mathrm{Inner}(\mathcal T_0)$.
The quotient
\begin{equation}\label{eq:quotientLG} 
\frac{H}{\mathrm{Inner}(\mathcal T_0)_{\geq k }}    \end{equation}
admits a  Lie group structure (by Cartan's theorem), with Lie algebra the finite dimensional quotients $\mathcal T_0 / \mathcal T_{\geq k} $. 

As a consequence, for any diffeological principal $H$-bundle $P \to \leaf $ and any $k \geq 2 $, the quotient of $ P$ by  ${\mathrm{Inner}}(\mathcal T)_{\geq k} $ is a finite dimensional principal bundle with respect to the  Lie group as in \eqref{eq:quotientLG}.

\begin{prop}
\label{prop:simplifyBundle}
Let $\leaf $ be a manifold and $ \mathcal T_0$
 be a formal singular foliation near $0 \in \mathbb R^d$. Let $ H \subset {\mathrm{Sym}}(\mathcal T_0) $ be as in Equation \eqref{eq:defH}. Then
 there is a one-to-one correspondence between:
\begin{enumerate}
\item[(i)]  diffeological principal $H$-bundles over $ \leaf$, and
\item[(ii)] ordinary\footnote{I.e., finite-dimensional} principal $ H / {\mathrm{Inner}}(\mathcal T_0)_{\geq 2} $-bundles over $\leaf $.
\end{enumerate}
\end{prop}
\noindent
The crucial point of the proof is that $\mathrm{Inner}(\mathcal T)_{\geq 2} $ is a contractible normal subgroup in view of Lemma \ref{lem:exp2}. Since we work with diffeological bundles, we have to adapt classical differential geometric proofs.
For every $ k \geq 2$, the quotient space
$(\mathcal T_0)_{\geq k}/ (\mathcal T_0)_{\geq k+1} $ is finite dimensional, and is acted upon by the Lie group
$ H / {\mathrm{Inner}}(\mathcal T_0)_{\geq k}$.
The following lemma is therefore obvious, since the sheaf  is a module over smooth functions on $L$, so that it admits partitions of unity.

\begin{lemma}
\label{lem:flatsheaf}
Let $ P_k \to \leaf$ be an $ H / {\mathrm{Inner}}(\mathcal T_0)_{\geq k}$-principal bundle over $\leaf$.
The \v{C}ech cohomology of sections of the associated bundle $$ \frac{P_k \otimes (\mathcal T_0)_{\geq k}/ (\mathcal T_0)_{\geq k+1}}{H / {\mathrm{Inner}}(\mathcal T_0)_{\geq k}} $$
is zero in every positive degree. 
\end{lemma}

\begin{proof}[Proof of Proposition \ref{prop:simplifyBundle}]
The map from item (i) to item (ii) consists in mapping an $H$-bundle $P$ to $ P/{\mathrm{Inner}(\mathcal T_0)_{\geq 2}} $.
We have to show that this map is into and onto.
We start with "onto". Consider a  principal $H/{\mathrm{Inner}}(\mathcal T_0)_{\geq 2}$-bundle $ P_2$ as in item (ii). Assume that it already admits an extension to a principal $H/{\mathrm{Inner}}(\mathcal T_0)_{\geq k}$-bundle $P_k \to L$. Let us show that $ P_k$ admits an extension $P_{k+1}$ to a principal $H/{\mathrm{Inner}}(\mathcal T_0)_{\geq k+1}$-bundle. Let $(\mathcal U_a)_{a \in A} $ be an open cover  of $ \leaf$ by contractible open sets with contractible intersections. The bundle $P_k \to \leaf $  admits on each $\mathcal U_a $ a section $\sigma_a^{(k)} $. 
The smooth functions $\sigma_{ab}^{(k)} \colon \mathcal U_{ab} \longrightarrow H/{\mathrm{Inner}}(\mathcal T_0)_{\geq k} $ that satisfy $\sigma_{ab}^{(k)}\cdot \sigma_b^{(k)} = \sigma_a^{(k)} $ can be lifted to smooth maps $ \tilde{\sigma}_{ab} \colon \mathcal U_{ab} \to H   $. We can assume $ \tilde{\sigma}_{ab}=\tilde{\sigma}_{ba}^{-1}$. Since the smooth functions $\sigma_{ab}^{(k)}$ form a $1$-cocycle 
in $ H/ {\mathrm{Inner}}(\mathcal T_0)_{\geq k}$, their lifts satisfy that 
for every $a,b,c \in A$ the map
$\tilde{\sigma}_{ab} \circ  \tilde{\sigma}_{bc} \circ \tilde{\sigma}_{ca} $
  is valued in ${\mathrm{Inner}}(\mathcal T_0)_{\geq k} $. By Lemma \ref{lem:exp2}, there exists therefore a uniquely defined family of maps $\tau_{abc}  \in (\mathcal T_0)_{\geq k} $ such that: 
$$   \tilde{\sigma}_{ab} \circ  \tilde{\sigma}_{bc} = {\mathrm{exp}}\left( \tau_{abc}\right) \circ \tilde{\sigma}_{ac}  $$
Let $\tau \mapsto \bar{\tau} $ be the projection from the sheaf $(\mathcal T_0)_{\geq k} $ to $(\mathcal T_0)_{\geq k}/(\mathcal T_0)_{\geq k+1} $.  One can check by writing $ \tilde{\sigma}_{ab} \circ \tilde{\sigma}_{bd} \circ {\mathrm{exp}}\left( \tau_{bcd}\right)$ in two different manners, that is, on one hand
 \begin{eqnarray*} \tilde{\sigma}_{ab} \circ  \tilde{\sigma}_{bc} \circ \tilde{\sigma}_{cd} &=&   {\mathrm{exp}}\left( \tau_{abc}\right)  \circ  \tilde{\sigma}_{ac} \circ \tilde{\sigma}_{cd} =  {\mathrm{exp}}\left( \tau_{abc}\right)  \circ   {\mathrm{exp}}\left( \tau_{acd}\right)  \circ \tilde{\sigma}_{ad}~, \end{eqnarray*} 
 and on the other hand
\begin{eqnarray*} \tilde{\sigma}_{ab} \circ  \tilde{\sigma}_{bc} \circ \tilde{\sigma}_{cd}  & =&  \tilde{\sigma}_{ab}  \circ {\mathrm{exp}}\left( \tau_{bcd}\right) \circ \tilde{\sigma}_{bd} \\ &= &  {\mathrm{exp}}\left(  \tilde{\sigma}_{ab} (\tau_{bdc})\right) \circ \tilde{\sigma}_{ab} \circ \tilde{\sigma}_{bd} \\
 &=&    {\mathrm{exp}}\left(  \tilde{\sigma}_{ab} (\tau_{bcd})\right) \circ  {\mathrm{exp}}\left(  \tau_{acd}\right) \circ \tilde{\sigma}_{ad}~. \end{eqnarray*} 
 
Hence $${\mathrm{exp}}\left( \tau_{abc}\right)  \circ   {\mathrm{exp}}\left( \tau_{acd}\right)  
 =   {\mathrm{exp}}\left(  \tilde{\sigma}_{ab} (\tau_{bcd})\right) \circ  {\mathrm{exp}}\left(  \tau_{acd}\right).$$
This relation, taken modulo $   (\mathcal T_0)_{\geq k+1}$, implies  by Baker-Campell-Hausdorf that
$$ \bar \tau_{acd}+\bar\tau_{abc} =  \sigma_{ab}^{(k)} \cdot  \bar\tau_{bcd}+\bar\tau_{abd} $$
i.e.,
the family $(\bar{\tau}_{abc})_{a,b,c \in A}$ is a cocycle in the \v{C}ech cohomology in Lemma \ref{lem:flatsheaf}.
 Lemma \ref{lem:flatsheaf}
implies therefore that  the family $ (\bar{\tau}_{abc})_{a,b,c\in A}$ is a $2$-coboundary, \textit{i.e.}\ there exists a family $\theta_{ab} \colon \mathcal U_{ab} \longrightarrow   (\mathcal T_0)_{\geq k} $ such that 
 $$  \bar{\tau}_{abc} = \sigma_{ab}^{(k)}\cdot  \bar{\theta}_{bc}  - \bar{\theta}_{ac} 
 + \bar{\theta}_{ab}. $$
 This implies that the family of maps $\mathcal U_{ab}\to H/{\mathrm{Inner}}(\mathcal T_0)_{\geq k+1} $ defined by $$ \sigma_{ab}^{(k+1)} := {\mathrm{exp}}(-\theta_{ab}) \circ \tilde{\sigma}_{ab}  \hbox{ modulo ${\mathrm{Inner}}(\mathcal T_0)_{\geq k+1} $ }$$ is a $1$-cocycle. 
 It therefore corresponds to a principal $H/{\mathrm{Inner}}(\mathcal T_0)_{\geq k+1}$-bundle $ P_{k+1}$. Moreover, 
 applying successively this method starting at $k=2$,  one obtains a family of maps $\left( \sigma_{ab}^{(k)} \right)_{k \geq 1} $ that converge with respect to the considered diffeologies, and  whose limit is a principal $H$-bundle $P$ over $\leaf $. 

 Now, we have to check that it is "into", \textit{i.e.}\ that two $H$-principal bundles $P$ and $\tilde{P}$ inducing the same $ H/\mathrm{Inner}(\mathcal T_0)_{\geq 2}$-principal bundle are isomorphic. Again, let us show that if their induced $ H/\mathrm{Inner}(\mathcal T_0)_{\geq k}$-bundles $P_k$ and $\tilde{P}_k $ are isomorphic for some $ k\geq 2$, so are their induced principal $ H/\mathrm{Inner}(\mathcal T_0)_{\geq k+1}$-bundles $P_{k+1}$ and $\tilde{P}_{k+1} $. This comes from the fact that if the associated $1$-cocycles $(\sigma_{ij})$ and $  (\tilde{\sigma}_{ij})$ (computed with the help of local sections $(\sigma_a)_{a \in A}$ and $(\tilde{\sigma}_a)_{a \in A} $ respectively) match up to order $k$, then for all indices $a,b \in A$, Lemma \ref{lem:exp2} implies that $\sigma_{ab}  \circ(\tilde{\sigma}_{ab})^{-1} = {\mathrm{exp}}\left({\tau_{ab}}\right)$ for some $\tau_{ab}$ valued in $ (\mathcal T_0)_{\geq k}$. Again, the family $\bar{\tau}_{ab} \in ({\mathcal T}_0)_{\geq k}/({\mathcal T}_0)_{\geq k+1} $ form a $1$-cocycle for the  \v{C}ech cohomology as in Lemma \ref{lem:flatsheaf}. It is therefore a coboundary for some $(\eta_a)_{a \in A} $.
 In turn, the family of local sections $\mathrm{exp}(\eta_a) \sigma_a  $ over $ \mathcal U_a$ of $P \to L$ defines a $1$-cocycle that matches the one of $\tilde{P}$ up to order $k+1 $, and hence defines an isomorphism $P_{k+1} \simeq \tilde{P}_{k+1} $ of principal $H/{\mathrm{Inner}}(\mathcal T_0)_{\geq k+1} $-bundles.  Applying the argument recursively starting at $k=2$, one gets a family of local sections of $P \to M$ and  $\tilde P \to M$ that define the same $1$-cocycle. These sections converge by definition of the diffeology, and provide by construction the desired isomorphism $ P \simeq \tilde{P}$.
\end{proof}

Proposition \ref{prop:simplifyBundle}
applied to $H= \mathrm{Inner}(\mathcal T_0)  $  gives the following interesting particular case.

\begin{cor}
\label{coro:simplifyBundle}
Let $\leaf $ be a manifold and $ \mathcal T_0$
 be a formal singular foliation near $0 \in \mathbb R^d$. 
 There is a one-to-one correspondence between:
\begin{enumerate}
\item[(i)]  diffeological $ \mathrm{Inner}(\mathcal T_0) $-principal bundles over $ \leaf$ and
\item[(ii)] ordinary $  \mathrm{Inner}(\mathcal T_0)/ \mathrm{Inner}(\mathcal T_0)_{\geq 2} $-principal bundles over $\leaf $.
\end{enumerate}
\end{cor}

\begin{rem}
\label{rem:concentric}
Notice that there is a natural group morphism:
 $$  \mathrm{Inner}(\mathcal T_0)/\mathrm{Inner}(\mathcal T_0)_{\geq 2} \longrightarrow  \mathrm{Inner}(\mathcal T_0)_{\mathrm{lin}}  $$
which is a local diffeomorphism. Here $\mathrm{Inner}(\mathcal T_0)_{\mathrm{lin}} \subset {\mathrm{GL}}_d(\mathbb R^d) $ is the subgroup obtained by considering the differential at $0$ of all elements in $\mathrm{Inner}(\mathcal T_0)$. This local diffeomorphism may not be an isomorphism in general.
Here is a counter example: take $\mathcal T_0 $ to be the formal singular foliation on $\mathbb R^2 $ (called "spirals" in the introduction) generated by $$ x \frac{\partial}{\partial y} - y \frac{\partial}{\partial x} + (x^2+y^2) \left(x \frac{\partial}{\partial x} + y \frac{\partial}{\partial y}\right).$$ Then 
$\mathrm{Inner}(\mathcal T_0)_{\mathrm{lin}}\simeq  S^1$ while $\mathrm{Inner}(\mathcal T_0)/\mathrm{Inner}(\mathcal T_0)_{\geq 2} \simeq \mathbb R$. 
\end{rem}

\begin{blue}
We intend to prove a second proposition that explains how  a principal $ H$-bundle with $ H$ a group as in \eqref{eq:defH} induces a formal singular foliation along the leaf $L$.

We start with a lemma. Let $\cU$ be a manifold (it will be an open subset of $\leaf$ below, hence the notation). 
Consider the trivial formal neighborhood of $\cU$ (see Example \ref{ex:obvious}). 
Any smooth map 
\begin{align*}
 g \colon   \cU & \to  {\mathrm{Diff}}^{\mathrm{formal}} \left(\mathbb R^d, 0\right) 
 \\ 
 u & \mapsto g_u 
 \end{align*}
  can be seen as a formal diffeomorphism of this formal neighborhood: It maps $ (y,x)$ in $ \cU \times \mathbb R^d$ to $ (u,g_u(x))$, or, in terms of sheaves:
\begin{equation}\label{eq:formaldiffeo}
      \sum_{i_1, \dots, i_d \geq 0} f_{i_1, \dots, i_d}(u) ~ x_1^{i_1} \dots x_d^{i_d} \mapsto  \sum_{i_1, \dots, i_d \geq 0} f_{i_1, \dots, i_d}(u) ~ g_u \left( x_1^{i_1} \dots x_d^{i_d} \right). \end{equation}

\begin{lemma}  
\label{lem:simplifyBundle}
Let $\mathcal U $ be a manifold and $ \mathcal T_0$
 be a formal singular foliation near $0 \in \mathbb R^d$. Let $ H \subset {\mathrm{Sym}}(\mathcal T_0) $ be a sub-group as in Equation \eqref{eq:defH}. Any smooth map $$ 
 \begin{array}{rcll} g \colon & \cU &\longrightarrow &  H\\ &y &\mapsto  & g_y\end{array}, $$ 
 seen as the formal diffeomorphism  of $ \mathcal U \times \mathbb R^d$ as in \eqref{eq:formaldiffeo},
 is a symmetry of the trivial singular foliation on $\cU$ with transverse model $  \mathcal T_0$.
 \end{lemma}
 \begin{proof}
 Let $ (y_1, \dots, y_n)$ be local coordinates on $ \mathcal U$ centered at $0$. Let us decompose $g$ as in Equation \eqref{eq:gy1yr} in Lemma \ref{lem:diffeologyH}.
For every $k=1, \dots, n$, the pushforward $ g_*$ of the vector field $\frac{\partial }{\partial y_k}$ through the formal diffeomorphism $g$ is the formal vector field
    \begin{align*}
        g_* \left(\frac{\partial }{\partial y_k}\right)  &= \frac{\partial }{\partial y_k}+ g^{-1} \frac{\partial g_y}{\partial y_k} , \\
        &= \frac{\partial }{\partial y_k}+ g(0,\dots, 0)_* \mleft(  e^{-X} \left.\frac{\rmd}{\rmd t} e^{X + tY} \right|_{t=0}  \mright)
        \end{align*} 
       where we use the notations of Equation \eqref{eq:gy1yr}, namely
\begin{align*}
X & = \sum_{\substack{i_1 + \dots +i_d \geq 0 \\ i=1, \dots, r}} f^i_{i_1, \dots, i_d}(y_1, \dots,y_n) \, x_1^{i_1} \dots x_d^{i_d}  \,\, X_i ,\\ 
 Y &=  \sum_{\substack{i_1 + \dots +i_d \geq 0 \\ i=1, \dots, r}} \frac{\partial f^i_{i_1, \dots, i_d}(y_1, \dots,y_n)}{\partial y_k} \, x_1^{i_1} \dots x_d^{i_d}  \,\, X_i.  
 \end{align*}
 The classical formula for the differential of the exponential map:
    \begin{equation}
    \label{eq:formalexp}
 e^{-X} \left.\frac{\rmd}{\rmd t} e^{X + tY} \right|_{t=0} = \sum_{n=0}^\infty \frac{1}{(n+1)!} \, \text{ad}_X^n(Y)
\end{equation}
 together with Eq. \eqref{eq:coeffs} implies 
 $$ e^{-X} \left.\frac{\rmd}{\rmd t} e^{X + tY} \right|_{t=0}= \sum_{\substack{i_1 + \dots +i_d \geq 0 \\ i=1, \dots, n}} t^i_{i_1, \dots,i_d}(y_1, \dots,y_n) \, x_1^{i_1} \dots x_d^{i_d}  \,\, X_i  , $$
where for all possible indices
$  t^i_{i_1, \dots,i_d}(y_1, \dots,y_n)$ is a smooth function, since it is given by some polynomial expression using finitely many of the functions $\frac{\partial f^i_{i_1, \dots, i_d}(y_1, \dots,y_n)}{\partial y_i} $ and $ f^i_{i_1, \dots, i_d}(y_1, \dots,y_n)  $, and finitely many of the coefficients that appear in \eqref{eq:coeffs}  and their derivatives. In other words, it lies in 
 the trivial singular foliation on $\cU$ with transverse model $  \mathcal T_0$.
 Now, since $ g(0, \dots,0 )$ is a symmetry of $ \mathcal T_0$, there exists formal functions such that:
\begin{equation}\label{eq:pushforwardg0} g(0, \dots,0 )_*(X_i)= \sum_{j=1}^r   K^i_{i_1, \dots,i_d}(y_1, \dots,y_n) \, x_1^{i_1} \dots x_d^{i_d}  \,\, X_i   \end{equation}
for some smooth functions $ K^i_{i_1, \dots,i_d}(y_1, \dots,y_n)$.
  In turn, this allows to write
$$  g^{-1} \frac{\partial g_y}{\partial y_k} =  \sum_{\substack{i_1 + \dots +i_d \geq 0 \\ i=1, \dots, n}} T^i_{i_1, \dots,i_d}(y_1, \dots,y_n) \, x_1^{i_1} \dots x_d^{i_d}  \,\, X_i  , $$
where for all possible indices, where each
$  T^i_{i_1, \dots,i_d}(y_1, \dots,y_n)$ is a smooth function obtained as a polynomial expression in finitely many of the coefficients $ t^i_{i_1, \dots,i_d}(y_1, \dots,y_n)$ and  $ K^i_{i_1, \dots,i_d}(y_1, \dots,y_n) $.
As a conclusion,  $g^{-1} \frac{\partial g_y}{\partial y_k}$ lies in 
 the trivial singular foliation on $\cU$ with transverse model $  \mathcal T_0$.
This implies that the pushforward of $   \frac{\partial }{\partial y_k}$ lies in 
 the trivial singular foliation on $\cU$ with transverse model $  \mathcal T_0$.

We now have to prove that the   push-forward of each one of the vector fields $  X_1, \dots , X_r$ through $g$ lies in 
 the trivial singular foliation on $\cU$ with transverse model $  \mathcal T_0$.
 The push-forward through $ e^X$ with $X  = \sum_{\substack{i_1 + \dots +i_d \geq 0 \\ i=1, \dots, r}} f^i_{i_1, \dots, i_d}(y_1, \dots,y_n) \, x_1^{i_1} \dots x_d^{i_d}  \,\, X_i $ of $X_i$ coincides with $ e^{\mathrm{ad}_X}(X_i)$. It lies in that singular foliation by arguments similar to those above. Also, by the same arguments of above, using the coefficients that appear in Equation \eqref{eq:pushforwardg0}, so is the push-forward of $e^{\mathrm{ad}_X}(X_i)$  through $ g(0, \dots, 0)$. This completes the proof. 
 \end{proof}
\end{blue}

\subsection{The transverse singular foliation} 
\label{sec:transver}

The usual theory of smooth singular foliations extends with minor adaptations to the formal case.
Most of the following definitions and theorems are obvious extensions of the equivalent results for neighborhood of leaves on singular foliations - references will be given along the way. 
 We will give a brief introduction to the matter, and refer to \cite{LLR} for a more detailed study.
Let $\vertical $ be the normal bundle of a formal neighborhood of $\leaf $, and assume that a tubular neighborhood isomorphism as in Lemma \ref{lem:normalstructure} is chosen.
Let $\pi^* \colon C^\infty(\leaf) \longrightarrow C^{\mathrm{formal}} $ be the induced algebra morphism (which is the pull-back map in the context of Example \ref{decomp:step2.5}). 
We say that a formal vector field $X \in \mathfrak X^{\mathrm{formal}} $ tangent to $ \leaf$ is \emph{$\pi$-compatible} if $$  
\xymatrix{  C^{\mathrm{formal}} \ar[r]^{X}& C^{\mathrm{formal}}  \\ C^\infty(\leaf) \ar[u]^{\pi^*} \ar[r]^{X|_\leaf} & C^\infty(\leaf) \ar[u]^{\pi^*} }  
$$
where $X|_\leaf \in \mathfrak X(\leaf)$ is the restriction of $ X$
 to $ \leaf$ which we may also denote by $\underline{X}$. This gives rise to a $C^\infty(L)$-module which we denote by $\mathfrak{X}_\pi^{\mathrm{formal}}(T)$.
 We call \emph{formal Ehresmann connection} a splitting  $\mathbb{H}$ of the following exact sequence
\begin{equation}\label{FConnAsSplitting0}
\begin{tikzcd}
{\mathrm{Vert}}_\pi \arrow[hook]{r}
&
\mathfrak{X}_\pi^{\mathrm{formal}}(T) \arrow[two heads]{r}
&
\arrow[bend right,swap]{l}{\mathbb H}
\mathfrak{X}(L),
\end{tikzcd}
\end{equation}
with ${\mathrm{Vert}}_\pi$ being the Lie subalgebra of vertical formal vector fields, that is, $X \in \mathfrak{X}^{\mathrm{formal}}$ 
that satisfy $X [\pi^* F] =0 $ for every $F \in C^\infty(L)$, i.e.\ $X \in \mathfrak{X}_\pi^{\mathrm{formal}}(T)$ and $\underline{X} = 0$. 
Ehresmann connections always exist, as we will see shortly below.

 Let $\mathcal F $ be a formal singular foliation along the leaf $\leaf $. 
Let us first define the transverse singular foliation.
\begin{enumerate}
\item formal vector fields in $X \in \mathcal F \cap \mathrm{Vert}_\pi$ shall be called \emph{vertical vector fields in $\mathcal F $} and denoted by $ \mathcal T$. By construction, $\mathcal T$ is a $ C^{\mathrm{formal}}$-module stable under Lie bracket. It is locally finitely generated, so is a singular foliation along $\leaf$ that we call the \emph{transverse singular foliation}.
\item For every $\ell \in L $, any formal vector field in $\mathcal {T}$ restricts to yield a formal vector field on the fibre $ T_\ell$ of the normal bundle. The image of this restriction is a formal singular foliation in the finite-dimensional vector space $ T_\ell$ that we denote by $ (T_\ell, \mathcal {T}_\ell)$. By construction, it is made of vector fields that vanish at $0 \in T_\ell$. We call $(T_\ell, \mathcal T_\ell) $ the \emph{transverse singular foliation at $\ell \in \leaf $}. 
\end{enumerate}

We call \emph{formal $\cF$-connection} a formal Ehresmann connection  $\mathbb{H}$ valued in $\mathcal F $, \emph{i.e.} which splits the following short exact sequence of $C^\infty(L)$-modules
\begin{equation}\label{FConnAsSplitting}
\begin{tikzcd}
\cT \arrow[hook]{r}
&
\cF_\pi \arrow[two heads]{r}
&
\arrow[bend right,swap]{l}{\mathbb H}
\mathfrak{X}(L),
\end{tikzcd}
\end{equation}

\begin{example}
\label{decomp:step6}
\normalfont
This example continues Example \ref{decomp:step4}.
Let $\mathcal F $ be a singular foliation on $M $ admitting $ \leaf$ as an embedded leaf.
Any formal $\mathcal F $-connection for $\leaf $ in the sense of \cite{LGR} (\textit{i.e.}\ a tubular neighborhood of $\leaf $ \& an Ehresmann connection whose sections are in $\mathcal F $) yields a $C^\infty(\leaf) $-linear map $\mathbb H \colon \mathfrak X(\leaf) \to \mathfrak X(M)$ with $(\mathrm{D}\pi \circ \mathbb H)(X) = X$ for all $X \in \mathfrak X(\leaf)$. Its composition  with the Lie algebra morphism $\mathfrak{X}(M) \to \mathfrak{X}^{\mathrm{formal}}$ of Example \ref{decomp:step3} gives a formal $\mathcal F $-connection.
\end{example}

Let us list a few facts, analogous to those valid for smooth singular foliations around embedded leaves \cite{LGR}:
 \begin{enumerate}
 \item around each point of $\leaf $, a  flat formal $\mathcal F $-connection exists; see \cite{meinrenkensplitting} for splitting theorems related to such statements.
\item formal $\mathcal F $-connections exist globally (since one can glue the formal flat ones using partitions of unity). 
\end{enumerate}

The following is an immediate consequence of \eqref{FConnAsSplitting}:

\begin{cor}[Curvature of a formal $\cF$-connection]\label{cor:CurvatureOfFConn}
The curvature of a formal $\cF$-connection $\mathbb{H}$ has values in $\cT$, that is,
\begin{equation*}
\bigl[ \mathbb{H}(X), \mathbb{H}(X') \bigr]
    - \mathbb{H}\bigl( [X, X'] \bigr)
\in \cT.
\end{equation*}
for every $X,X' \in \mathfrak X(L)$
\end{cor}

Furthermore, any other formal $\cF$-connection $\mathbb{H}'$ is also a splitting of \eqref{FConnAsSplitting}, and thus:

\begin{cor}[Difference between two different formal $\cF$-connections]\label{cor:DiffOfTwoFConn}
Two different formal $\cF$-connections $\mathbb{H}$ and $\mathbb{H}'$  differ by $C^\infty$-linear maps $\mathfrak{X}(L) \to \cT$, that is,
\begin{equation*}
\mathbb{H}(X) - \mathbb{H}'(X) \in \cT
\end{equation*}
for all $X \in \mathfrak{X}(L)$.
\end{cor}

For every formal Ehresmann connection with respect to the  tubular neighborhood $\vertical$, parallel lifts can be defined, \textit{i.e.}\ for any path $ t \mapsto \gamma(t)$ on $L$ with $ \gamma(0)=\ell$ and $\gamma(1)=\ell' $, one can associate a formal diffeomorphism $ \mathrm{PT}_{\gamma}$ from $T_\ell $ to $T_{\ell'} $.
Exactly as for usual $ \mathcal F$-connections (see also \cite{LLR}), it follows for formal $\mathcal F $-connections from Corollaries \ref{cor:CurvatureOfFConn} and \ref{cor:DiffOfTwoFConn} above that the following three properties, that we shall denote by "$\mathcal P_A$", "$\mathcal P_B$" and "$\mathcal P_C$", hold true. They are the equivalent to the Ambrose-Singer theorem.

\begin{cor}
\label{cor:PAPBPC}
Choose a formal $\mathcal F $-connection with corresponding parallel transport denoted by $\mathrm{PT}$ on a singular foliation $\mathcal F $ along the leaf~$\leaf $.
\begin{enumerate}
\item["$\mathcal P_A$"] 
For any path $\gamma $ on $\leaf $ from $ \ell$ to $ \ell'$, $\mathrm{PT}_{\gamma}$  maps $ \mathcal T_\ell$ to $\mathcal T_{\ell'} $. Therefore, for any loop starting at $ \ell \in \leaf$,
 $\mathrm{PT}_{\gamma}$ is an element of $\mathrm{Sym}(\mathcal T_\ell) $.
\item["$\mathcal P_B$"] For any two homotopic paths $ \gamma_0, \gamma_1$ from $\ell$ to $\ell'$ with respect to some formal $\mathcal F $-connection, 
 $ \mathrm{PT}_{\gamma_0}$ and $\mathrm{PT}_{\gamma_1}$ differ by an inner symmetry, \textit{i.e.}\
$$ \mathrm{PT}_{\gamma_1}^{-1}  \circ \mathrm{PT}_{\gamma_0} \in {\mathrm{Inner}} (\mathcal T_\ell).$$
\item["$\mathcal P_C$"] For two different formal $\cF$-connections, and any path $\gamma $ on $\leaf $, their respective parallel transportation also differ by an inner symmetry. 
\end{enumerate}
\end{cor}

If we worked with local singular foliations, defined near $ L$, with completeformal $\cF$-connections as in \cite{LGR}, one could depict $\mathcal P_A,\mathcal P_B $ in this corollary as in Fig.\ \ref{fig:ParallelTransportChangingCircles}; also recall Ex.\ \ref{ex:ConcentricCirclesInnAndOut}.

\begin{figure}[ht]
    \centering
    \includegraphics[width=\textwidth]{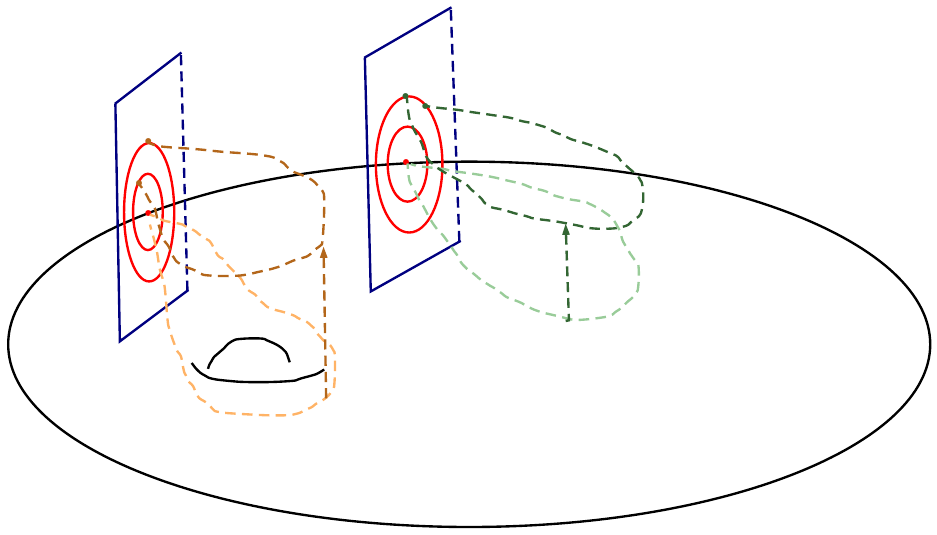}
    \caption{Parallel transport along loops following a foliation $\cF$ with concentric circle as transverse model.
    Whenever the loop is contractible, each horizontal lifts ends in the  circle it started from. Whenever the loop is not contractible, its horizontal lifts might end up in a different circle.}
    \label{fig:ParallelTransportChangingCircles}
\end{figure}

Property "$\mathcal P_A$" implies in particular the following statement:

\begin{prop}[Transverse model singular foliations are isomorphic]\label{prop:IsomorphicTransverseFoliations}
Let $\mathcal F $ be a formal singular foliation along a leaf $\leaf $. 
For any two choices of a formal tubular neighborhood, and any two points $\ell, \ell' $ in $\leaf $, the transverse singular foliations $(T_\ell, \mathcal T_\ell) $ and $(T_{\ell'}, \mathcal T_{\ell'}) $ are isomorphic. 
\end{prop}

This allows to make sense of the following definition.

\begin{deff}[Transverse models of a formal singular foliation]
We call a \emph{transverse model of a formal singular foliation along a leaf $\leaf $} a singular foliation $\mathcal{T}_0$ near $0 \subset \mathbb R^d$ which is formally isomorphic to the transverse singular foliation $(T_\ell, \mathcal T_\ell) $ for one (equivalently all) $l \in \leaf$, computed with respect to an arbitrary formal tubular neighborhood. 
\end{deff}

 \subsection{The  Yang-Mills groupoid}
 \label{sec:YMgroupoid}
\begin{blue}
Corollary \ref{cor:PAPBPC} allows to construct an important transitive groupoid (inspired by a similar construction by Meinrenken \cite{meinrenken2021integration}): The Yang-Mills groupoid below. It has an important interpretation in terms of Yang-Mills connections. For this choose a tubular neighborhood $\vertical$, and let $(T_\ell,\mathcal T_\ell)_{\ell \in \leaf} $ be the induced family of transverse singular foliations as in Section \ref{sec:transver}.
Let $(\mathbb R^d, \mathcal T_0) $ be a transverse model and recall that a 
formal $\mathcal F $-connection equips $T \to L $  with a parallel transport valued in formal diffeomorphisms $$\mathrm{PT}_{\gamma} \colon T_{\ell _0}\to T_{\ell_1} $$ for any path $\gamma $ in $ L$ from $\ell_{0} $ to $\ell_{1} $ and mapping $\mathcal T_{\ell_0} $ to $\mathcal T_{\ell_1} $.
 
Let us call \emph{frame groupoid} the groupoid ${\mathrm{Frame}}(\mathcal F) $ over $\leaf $ whose arrows between two arbitrary points $\ell,\ell'  \in \leaf$
 are made of all formal diffeomorphisms from $\normal_\ell $ to $\normal_{\ell'} $ mapping $\mathcal T_\ell $ to $\mathcal T_{\ell'} $. Proposition \ref{prop:IsomorphicTransverseFoliations} translates as follows.

 \begin{lemma}\label{lem:frame}
The frame groupoid ${\mathrm{Frame}}(\mathcal F) \rightrightarrows\leaf $ of a formal singular foliation along a leaf $ \leaf$ is a transitive diffeological groupoid. 
\end{lemma}

The Yang-Mills groupoid is a sub-groupoid of the frame groupoid  ${\mathrm{Frame}}(\mathcal F) \rightrightarrows\leaf   $ that we now construct. 
Consider for all $\ell_0,\ell_1  \in \leaf$ the subset
  $$  \mathrm{YM}_{\ell_0}^{\ell_1}  \subset {\mathrm{Frame}}(\mathcal F)_{\ell_0}^{\ell_1} $$
made of all formal diffeomorphisms from $(\normal_{\ell_0}, \mathcal T_{\ell_0}) $ to $(\normal_{\ell_1}, \mathcal T_{\ell_1}) $
which are a composition of 
\begin{enumerate}
\item parallel transportation  ${\mathrm{PT}}_\gamma  $ using a formal $\cF$-connection, where $\gamma  $ is in the set $\mathrm{Path}_{\ell_0}^{\ell_1}$ of paths from $\ell_0 $ to $ \ell_1$,
\item an inner symmetry of  $({\normal}_{\ell_1}, \mathcal T_{\ell_1}) $, \textit{i.e.}\ an element in the group bundle $\mathrm{Inner}(\mathcal T_{\ell_1}) $. 
\end{enumerate}
In equation:
\begin{equation}
\label{eq:YMgroupoidDef}
\mathrm{YM}_{\ell_0}^{\ell_1}  =  \left\{ \phi \circ {\mathrm{PT}}_\gamma  \middle| \gamma \in  {\mathrm{Path}}_{\ell_0}^{\ell_1} , \phi \in {\mathrm{Inner}}(\mathcal T_{\ell_1})\right\} .
\end{equation}

\begin{cor}
\label{cor:YMgroupoid}
$\mathrm{YM}(\mathcal F)  = \cup_{\ell_0,\ell_1 \in \leaf}  \mathrm{YM}_{\ell_0}^{\ell_1} $ is a transitive subgroupoid of ${\mathrm{Frame}}(\mathcal F) \rightrightarrows\leaf  $ that we call the Yang-Mills groupoid.

It does not depend on the choice of a formal $\mathcal F $-connection.
\end{cor}
\begin{proof}
It follows from property "$\mathcal P_A$" (see Corollary \ref{cor:PAPBPC}) that such a composition is indeed a formal diffeomorphism from $(\normal_{\ell_0}, \mathcal T_{\ell_0}) $ to $(\normal_{\ell_1}, \mathcal T_{\ell_1}) $, \textit{i.e.}\ is a subset of ${\mathrm{Frame}}(\mathcal F)_{\ell_0}^{\ell_1} $. Property "$\mathcal P_A$" also implies that
  $$
   \begin{array}{rcl} {\mathrm{Inner}}(\mathcal T_{\ell_0}) &\longrightarrow & {\mathrm{Inner}}(\mathcal T_{\ell_1})
   \\
   \phi & \mapsto & {\mathrm{PT}}_\gamma \circ \phi \circ   {\mathrm{PT}}_\gamma^{-1}.
   \end{array}
   $$
   is a group isomorphism. As a consequence,
 $$  \mathrm{YM}_{\ell_0}^{\ell_1}  =  \left\{ {\mathrm{PT}}_\gamma \circ \phi \, \middle| \, \gamma \in  {\mathrm{Path}}_{\ell_0}^{\ell_1} , \phi \in {\mathrm{Inner}}(\mathcal T_{\ell_0})\right\} $$
 and this implies that
  $  \mathrm{YM}(\mathcal F)  \subset {\mathrm{Frame}}(\mathcal F) $
is stable under the inverse map. Stability under product is a direct consequence of the formula
$$  \phi \circ {\mathrm{PT}}_{\gamma_1} \circ \psi \circ {\mathrm{PT}}_{\gamma_2}  = \phi \circ  \left({\mathrm{PT}}_{\gamma_1} \circ \psi \circ \left({\mathrm{PT}}_{\gamma_1}\right)^{-1}\right)  \circ {\mathrm{PT}}_{\gamma_1 \star \gamma_2}   $$
valid for any two paths $\gamma_1,\gamma_2 $ such that $\gamma_1(0)=\gamma_2(1) $, and any $\phi \in \mathcal T_{\gamma_1(1)},\psi \in  \mathcal T_{\gamma_1(0)} $ with the understanding that $ \gamma_1 \star \gamma_2$  is the concatenation of these two paths, starting with $\gamma_2$ followed by $\gamma_1$.

Property "$\mathcal P_C$" implies that the groupoid $\mathrm{YM} $ does not depend on the choice of a particular formal $ \mathcal F$-connection. 
\end{proof}
\end{blue}

\begin{rem}
\label{rem:holonomygroupoid}
One can define the \emph{holonomy groupoid} of a formal singular foliation $\mathcal F $ along a leaf $L$ as the groupoid obtained by taking the quotient of the Yang-Mills groupoid by the connected group subbundle of its isotropy group whose Lie algebra is $ I_\ell \mathcal T_\ell$ for every $ \ell \in L$, with $I_\ell $ the ideal of formal functions on the fiber $T_\ell$ that vanish at $ \ell$. A comparison with the construction in  \cite{AZ13,GV} of the holonomy groupoid of a leaf (as defined by Androulidakis-Skandalis \cite{AS})  justifies this definition, and shows that when the formal singular foliation comes from a smooth singular foliation as in Example \ref{decomp:step5}, there is a surjective groupoid morphism from the holonomy groupoid of $L$  to the holonomy groupoid as defined above. It is not bijective in general: for instance for the singular foliation on $ \mathbb R^2$ generated by
 $$ x \partial_y-y\partial_x +e^{-\frac{1}{x^2+y^2}} (x \partial_x +y \partial_y)  $$
 the Androulidakis-Skandalis holonomy groupoid for the leaf $\{(0,0)\} $ is $\mathbb R $
while it is $ S^1$ for the 
induced formal singular foliation.
\end{rem}

Now let us relate it with the notion of Yang-Mills connections.
The Lie algebra bundle whose fiber over $\ell $ is $(T_\ell,\mathcal T_\ell)$ integrates to a group bundle whose fiber over $\ell $ is  ${\mathrm{Inner}}(\mathcal T_\ell) $, i.e.\ the subgroup of formal diffeomorphisms generated by the formal flows of elements in $\mathcal T_\ell $. The group bundle has fibers isomorphic to ${\mathrm{Inner}} (\mathcal T_0)$ by Proposition \ref{prop:IsomorphicTransverseFoliations}. Also, it is locally trivial (by local existence of flat formal $\cF$-connections), hence it is a diffeological group bundle, and, equivalently, it can be defined as the identity component of the Yang-Mills groupoid's isotropy bundle.
We denote it by ${\mathrm{Inner}} (\mathcal T)$.

By definition, the group bundle acts formally on the normal bundle.
Each fiber of this formal action is isomorphic to the formal action of ${\mathrm{Inner}} (\mathcal T_0)$ on $\mathbb R^d $. This action is therefore faithful (in the formal sense).

As we have just seen in the proof of Corollary \ref{cor:YMgroupoid}: the formal $ \mathcal F$-connection induces an Ehresmann connection on the group bundle ${\mathrm{Inner}} (\mathcal T) $, and its parallel transportation $\mathrm{PT}^{\mathrm{Inner}(\cT)}_\gamma$ is given for any $\gamma $ as above by:
 $$ \mathrm{PT}^{\mathrm{Inner}(\cT)}_\gamma(\phi) =  (\mathrm{PT}_{\gamma}) \circ \phi \circ  \mathrm{PT}_{\gamma}^{-1} $$
 for every $\phi \in  {\mathrm{Inner}} (\mathcal T_{\ell_0}) $. One therefore has a formal Ehresmann connection on $T \to L$ and an Ehresmann connection on ${\mathrm{Inner}} (\mathcal T) $ which we call a compatible pair of Yang-Mills connections, that is, we have

 \begin{equation*}
     \mathrm{PT}_\gamma(\phi \cdot p)
     =
     \mathrm{PT}^{\mathrm{Inner}(\cT)}_\gamma(\phi) \cdot \mathrm{PT}_\gamma(p)
 \end{equation*}

 for all $p \in \cT_{\ell_0}$ and $\phi \in \mathrm{Inner}(\cT_{\ell_0})$, and such that for all contractible loops $\gamma_0$ starting and ending at $\ell_0 \in L$ there is a $\mathrm{Hol}^0(\gamma_0) \in \mathrm{Inner}(\cT_{\ell_0})$, only depending on $\gamma_0$, such that
\begin{equation*}
\mathrm{PT}_{\gamma_0}(p)
=
\mathrm{Hol}^0(\gamma_0) \cdot p~.
\end{equation*}

 \begin{rem}
One can similarly consider the group bundle $\mathrm{Sym}(\cT)$ whose identity component is $\mathrm{Inner}(\cT)$ and its structural Lie group is given by $\mathrm{Sym}(\cT_0)$; the previous argument can be extended to say that one has a pair of compatible Yang-Mills connections on $T \to L$ and on ${\mathrm{Sym}} (\mathcal T) $. Analogously, this holds for any subgroup bundle of $\mathrm{Sym}(\cT)$, containing $\mathrm{Inner}(\cT)$ as identity component, and stable under the conjugation with $\mathrm{PT}_\gamma$. An example of such a group bundle is the Yang-Mills groupoid's isotropy bundle.

Given such a subgroup bundle, due to a trivial center of and the faithfulness of the canonical action of $\mathrm{Sym}(\cT)$ it is straightforward to show that formal $\cF$ connections are 1:1 to compatible pairs of Yang-Mills connections on $T$ and the given subgroup bundle.

$\mathrm{Hol}^0(\gamma_0)$ is in any case an element of $\mathrm{Inner}(\cT_{\ell_0})$ because it comes from parallel transports over contractible loops. 
\end{rem}

\begin{theorem}
\label{thm:isYM} 
Consider a formal singular foliation $\mathcal F $ along a leaf $L $. Any formal $\mathcal F $-connection  equips the group bundle ${\mathrm{Inner}} (\mathcal T) $ and the transverse bundle $T \to L $ with a compatible pair of Yang-Mills connections. 
Moreover, two different formal $\mathcal F $-connections give Yang-Mills connections that differ through an action by an inner symmetry (by the canonical one on $T$ and by conjugation on ${\mathrm{Inner}}(\cT)$). 
\end{theorem}

\begin{rem}
    Analogously, one achieves the same statement by replacing $\mathrm{Inner}(\cT)$ with the isotropy bundle of the Yang-Mills groupoid.
\end{rem}

\begin{proof}[Proof of Theorem \ref{thm:isYM}]
In short: the first part of the statement  is an immediate consequence of the properties of the parallel transportation studied above. The second part is a consequence of properties $\mathcal P_B,\mathcal P_C $.
\end{proof}

In Appendix \ref{app:ClassOfYM}, we point out that our classification carries over to Yang-Mills connections in the centerless case.

\section{Classification I: The Theorem}

\label{sec:ClassI}

This section is devoted to an  abstract characterizations of the type: "formal singular foliations along a leaf $\leaf$  are in one-to-one correspondence with equivalence classes of some-objects-to-be-introduced", see Theorem \ref{thm:classification}. This equivalence will be simplified in Theorem \ref{thm:classificationSimplified} and exploited in Section \ref{sec:ClassII}. 

We start with a definition.

\begin{deff}
\label{def:transversedata}
Let $\leaf$ be a manifold and $(\mathbb R^d, \mathcal T_0) $ a formal singular foliation near $0$.
We call a \data{} a triple $(\tilde L, H, P)$ made of 
\begin{enumerate}
\item[(a)] 
a Galois cover\footnote{That is a connected manifold $ \tilde L$, equipped with a submersion onto $ L$, such the fundamental group of $L$ acts transitively on the fibers. It is a theorem that Galois covers are in one to one correspondence with normal subgroups of the fundamental group of $L$: to such a normal subgroup $K$, one associates the quotient of the universal cover of $L$ by $K$. By construction, $ \tilde L \to L$ is a $\pi_1(\tilde L, L):=\pi_1(L)/K$-principal bundle.} $\tilde{L} \to L $ (we denote by $ \pi_1(\tilde L, L)$ its Galois group),
\item[(b)] a subgroup $H$ of ${\mathrm{Sym}}(\mathcal T_0)$ containing $\mathrm{Inner}(\mathcal T_0) $ and equipped with a surjective group morphism onto $\pi_1(\tilde L, L) $ with kernel $\mathrm{Inner}(\mathcal T_0)$: $$ \xymatrix{  \ar@{^(->}[d] \mathrm{Inner}(\mathcal T_0) \ar@{^(->}[r]&\ar@{^(->}[d] H \ar@{->>}[r]& \ar@{^(->}[d] \pi_1(\tilde L,\leaf) \\ \ar@{^(->}[r]\mathrm{Inner}(\mathcal T_0) &  \ar@{->>}[r]\mathrm{Sym}(\mathcal T_0)& \mathrm{Out}(\mathcal T_0)  } $$
\item[(c)] and a diffeological\footnote{The diffeology here is induced by the exact sequence as in Equation~\eqref{eq:defH} and therefore satisfies  Lemma~\ref{lem:diffeologyH}.} $H$-principal bundle $ P \to L$ which is an extension of $\tilde{L} \to L $, i.e.\ such that $ \tilde L = P/\mathrm{Inner}(\mathcal T_0)$. 
\end{enumerate}

We say that two \datas{} $(\tilde{L_0}, H_0, P_0) $ 
and $(\tilde{L_1}, H_1, P_1)$ are equivalent if $\tilde{L_0}=\tilde{L_1}$ and if there exists $ \psi \in {\mathrm{Sym}}(\mathcal T_0)$ such that $ H_1 = \psi^{-1} H_0 \psi $ and such that $ P_1 $ is the $H_1$ bundle associated to $ P_0$ through the group  isomorphism $H_1 \simeq H_0 $ above. 
\end{deff}

\begin{rem}
\label{rem:Pgivestriple}
\normalfont
The three items defining a \data{} are not independent. 
Given a connected principal $H$-bundle $P \to L$, a Galois cover can be recovered as being $ P/ {\mathrm{Inner}}(\mathcal T_0) $.
In fact, there is a one-to-one correspondence between \datas{} and \underline{connected} principal $H$-bundles $P \to L$ where $H$ is a subgroup of ${\mathrm{Sym}}(\mathcal T_0) $ containing  $ {\mathrm{Inner}}(\mathcal T_0)$ as the connected component at the identity (as in Equation \eqref{eq:defH}).
\end{rem}

\begin{rem}
\Datas{} are in one-to-one correspondence with pairs made of a group morphism $$ \Xi: \pi_1(L)\to {\mathrm{Out}}(\mathcal T_0) $$
together with a principal $H$-bundle $ P \to L $, with $H \subset {\mathrm{Sym}}(\mathcal T_0)$ the inverse image through ${\mathrm{Sym}}(\mathcal T_0) \to {\mathrm{Out}}(\mathcal T_0)  $ of the image of $ \Xi$.
Two \datas{} are equivalent through some $\psi \in {\mathrm{Sym}}(\mathcal T_0)$ if and only if $ (\Xi_0,P_0)$ and $(\Xi_1,P_1) $ are related by $$\Xi_1= \overline{\psi}^{-1} \, \Xi_0 \, \circ \overline{\psi} $$ with $\overline{\psi} \in {\mathrm{Out}}(\mathcal T_0) $ being the class of $ \psi$, while $ P_1  $ is the principal $ H_1 \coloneqq \psi^{-1} H_0 \psi$-bundle associated to the $H_0$-principal bundle $P_0$,
\emph{i.e.}\ $ P_1=  \frac{P_0 \times H_1}{H_0}$.
 \end{rem}

Now, we will explain why leaf data classify formal singular foliations along a leaf $L$.

\subsection{From formal singular foliation to leaf data}\label{sec:Decomposition}
\begin{blue}
 Let $\mathcal F $ be a formal singular foliation  along a leaf $\leaf $. Let $(\mathbb R^d, \mathcal T_0)$ be a representative of its transverse model. We construct an equivalence class of \datas.

\vspace{0.2cm}

Let us choose:
\begin{enumerate}
\item[C1] A point $ \ell $ in $L$.
\item[C2] A tubular neighborhood $\normal \to L $.
\item[C3] An isomorphism\footnote{The fiber $ T_\ell$ is equipped, as in Subsection \ref{sec:transver}, with a formal singular foliation $ \mathcal T_\ell$, which is isomorphic to $\mathcal T_0 $ for all $ \ell \in L$ by Proposition \ref{prop:IsomorphicTransverseFoliations}.} $ (T_\ell,\mathcal T_\ell) \simeq (\mathbb R^d, \mathcal T_0)$.
\item[C4] An $ \mathcal F$-connection.
\end{enumerate}

Let $\pi_1(L,\ell) $ stand for the fundamental group of $ L$, with base point $ \ell$. 
 By property $ \mathcal P_A$ in Corollary \ref{cor:PAPBPC}, for every path $ \gamma$ starting at $\ell \in L $ and ending at some point $ \ell' \in L$, parallel transportation $\mathrm{PT}_\gamma   $ is an isomorphism of singular foliation from  $(\normal_\ell,\mathcal T_\ell) $ to $(\normal_{\ell'},\mathcal T_{\ell'}) $.
 Composing with the isomorphism chosen in (C3), we obtain an isomorphism of singular foliation from  $(\mathbb R^d,\mathcal T_0) $ to $(\normal_{\ell'},\mathcal T_{\ell'}) $.
Let us now consider for any $ \ell'\in L$ the set $P(\mathcal F)_{\ell'}$ of all formal isomorphisms
 $$  \mathbb R^d \simeq \normal_{\ell'}   $$
 obtained by composing:
 \begin{enumerate}
     \item an arbitrary  inner isomorphism of $ \mathcal T_0$,
     \item the isomorphism chosen in item (C3),
     \item the parallel transportation with the help of the $ \mathcal F$-connection chosen in item (C4) with respect to an arbitrary paths $ \gamma$ starting at $ \ell$ and ending at $ \ell'$.
 \end{enumerate}

\begin{rem}
\label{rem:alternative}
This set  $P(\mathcal F)_{\ell'}$ admits the following alternative description. It is obtained by composing:
  \begin{enumerate}
     \item The identification chosen in item (C3),
     \item an arbitrary inner isomorphism of $ \mathcal T_\ell$,
     \item the parallel transportation with the help of the $ \mathcal F$-connection chosen in item (C4) with respect to an arbitrary path starting at $ \ell$ and ending at $ \ell'$.
 \end{enumerate}
\end{rem}
By property $ \mathcal P_A$ in Corollary \ref{cor:PAPBPC}, the set  $P(\mathcal F)_{\ell'}$ is a subset of the set  of isomorphisms of singular foliations:
 $$ (\mathbb R^d,\mathcal T_0) \simeq (\normal_{\ell'},\mathcal T_{\ell'}) . $$
 We denote by $ P(\mathcal F)$ the union over $ \ell'$ of all  $P(\mathcal F)_{\ell'}$.  It fibers over $ L$ through a projection that we denote by $ \pi$:
  $$ \pi \colon  P(\mathcal F) \to L $$

For every loop $ \gamma$ based at $\ell \in L $, parallel transportation $\mathrm{PT}_\gamma   $ is a symmetry of $(\normal_\ell,\mathcal T_\ell) $, \textit{i.e.}\   $$  \mathrm{PT}_\gamma \in  {\mathrm{Sym}}(\mathcal T_\ell)
\simeq  {\mathrm{Sym}}(\mathcal T_0).
$$
The isomorphism above is given by the conjugation by the one given in (C3).
By properties $ \mathcal P_B$, the class of $\mathrm{PT}_\gamma$ in ${\mathrm{Out}}(\mathcal T_\ell) $ depends only on the homotopy class of $\gamma $. The map $\gamma \mapsto \mathrm{PT}_\gamma$ therefore induces a group morphism 
\begin{equation}\label{eq:pairing0} \pi_1(L,\ell) \longrightarrow
{\mathrm{Out}} (\mathcal T_\ell) \simeq {\mathrm{Out}} (\mathcal T_0).  \end{equation}
that we call \emph{\pairing{}} of $\mathcal F $. By property $ \mathcal P_C$, the \pairing{} does not depend on the choice of a formal $\cF$-connection.

\end{blue}
\begin{rem} The \pairing{} was noted by \cite[Definition 7.6]{francis2023singular} for codimension one singular foliations and in \cite{BDW} in the case where the transverse model is the singular foliation of all vector fields on $\mathbb R^d$ vanishing up to order $k+1$ (see  \cite[Definition 1.5]{BDW}). For Lie algebroids, a similar construction was completed by Rui Loja Fernandes (see \cite[Definition 3.1]{fernandes}). None of these constructions are in the formal case, but are identical beside this difference.
Last,  a similar map that appeared in Definition 2.9 in \cite{LGR} as a group morphism $ \pi_1(L) \to {\mathrm{Diff}}(\pi^{-1}(\ell)/\mathcal T_\ell)$ in the smooth case provided that a complete formal $\cF$-connection exists,
so that there is a tubular neighborhood of $L$ on which all transverse singular foliations are isomorphic.
In the formal setting that we use here, these conditions are automatic.
Above, $ {\mathrm{Diff}}(\pi^{-1}(\ell)/\mathcal T_\ell)$
is the group of bijections of the leaf space of the transverse singular foliation at some $\ell \in L $ that comes from a symmetry of $ \mathcal T_\ell$.
This group coincides with the group of outer symmetries of the transversal, and maps to the group of outer symmetries of the $ \infty$-jet of the formal singular foliation associated to $\mathcal T_\ell $. This composition is a group morphism that coincides with the \pairing{} of the formal singular foliation along $L$ considered in Example \ref{decomp:step5}.
\end{rem}

\begin{blue}
Another choice of isomorphism in item (C3) or an other choice of base point in item (C1) above would change the \pairing{} by conjugation with some element in ${\mathrm{Out}} (\mathcal T_0)$.
Also, it does not depend on the choice of a tubular neighborhood up to this conjugation.
In particular, the kernel $K \subset \pi_1(L,\ell)$ of the \pairing{} does not depend on the choices  (C2), (C3), and (C4). 
Consider
 $$  L(\mathcal F) = L^u / K$$
 with $L^u $ the universal cover of $ L$, computed with respect to the base point $\ell$. 
 The bundle  $L(\mathcal F)\to L$
 is a principal bundle with respect to the quotient group $$\pi_1(L(\mathcal F), L) :=\frac{\pi_1(L,\ell)}{K}.$$ 
Recall some vocabulary:  $L(\mathcal F)\to L$ is what we call a \emph{Galois bundle} (= connected principal bundle with respect to a discrete group) and $\pi_1(L(\mathcal F), L)$ is called its  \emph{Galois group}. 
A different choice of a base point $ \ell$ would give a canonically diffeomorphic Galois bundle, hence we cay say that we have
 therefore associated a Galois bundle to $ \mathcal F$.

Let $ H(\mathcal F) \subset  \mathrm{Sym}(\mathcal T_0) $ be the inverse image of the image of the \pairing . By construction, $ L(\mathcal F)$ and $ H(\mathcal F)$ satisfy the properties required in items (a) and (b) of Definition \ref{def:transversedata}.

Let us check that $ P(\mathcal F)$ now respects item (c).
By property $ \mathcal P_C$, an another choice of an $\mathcal F $-connection or another choice of a tubular neighborhood would define the same set $P(\mathcal F) $.
 By the definition of $H(\mathcal F) $ for two different paths $\gamma, \gamma'$ with the same starting and ending points
 $$\left(\mathrm{PT}_{\gamma'}\right)^{-1}  \circ  \mathrm{PT}_\gamma   =\mathrm{PT}_{(\gamma')^{-1}\star \gamma}$$
lies in the image of $H(\mathcal F) $ through the isomorphism chosen in (C3).
As a consequence, 
the fibers of $P(\mathcal F) \to L$ are by construction acted upon transitively by $ H(\mathcal F)$.
This action is also free, because for any two elements of $ P(\mathcal F)$ with associated paths $ \gamma,\gamma'$, 
$$\left(\mathrm{PT}_{\gamma'}\right)^{-1}  \circ  \mathrm{PT}_\gamma   \in \mathrm{Inner}(\mathcal T_\ell) \hbox{ iff }  \mleft[(\gamma')^{-1} \star \gamma\mright] \in K $$
by definition of the kernel $K$.
It is a locally trivial bundle since $ \mathcal F$ is locally trivial.
It is therefore a principal bundle. 

Last, there is a natural map
 $ P(\mathcal F) \to L(\mathcal F)$ that maps an element of $ P(\mathcal F)$ to the homotopy class modulo $K$ of the path $ \gamma$ defined in item (3). By construction, $ P(\mathcal F) \to L(\mathcal F)$ is a principal $ \mathrm{Inner}(\mathcal T_0)$-bundle, and it turns $ P(\mathcal F) \to L$ into an extension of the Galois bundle $ L(\mathcal F) \to L $. 

Let us recapitulate. To a formal singular foliation along a leaf $ L$ with transverse model $(\mathbb R^d, \mathcal T_0) $, we have associated: 
\begin{enumerate}
\item[(a)] a Galois cover $L(\mathcal F) $ of $L$ (with respect to its Galois group $ \pi_1(L(\mathcal F),L)$),
\item[(b)] a sub-group $H(\mathcal F) $ of $ {\mathrm{Sym}}(\mathcal T_\ell)$  making the sequence \eqref{eq:defH} an exact sequence of groups, and
   \item[(c)] an extension of $ L(\mathcal F)$ to a principal $H (\mathcal F) $-principal bundle $P(\mathcal F) $.
\end{enumerate}
It is therefore a \data{}.
Now, a different choice of an $ \mathcal F$-connection in (C4) would give the same set. 
Last,  a different choice of the identification in (C3)   would yield an equivalent \data, and so would a different choice of based point (C1), or a different choice of a formal tubular neighborhood (C2).
%
%
 We can therefore conclude this discussion as follows.

\begin{prop}\label{prop:leafdata}
To any formal singular foliation $\mathcal F $ along a leaf $ \leaf$ with transverse model $(\normal, \cT_0)$, the above construction associates an equivalence class of \datas{}.
\end{prop}

\end{blue}

\begin{example}
\normalfont
\label{ex:vectorBundles}
For formal vector fields on a vector bundle $ E \to M$  tangent to the zero section, the
transverse model is the singular foliation of all formal vector fields on $\mathbb R^d $ that vanish at $0$. By Example \ref{ex:outis}, its outer symmetry group is $\mathbb Z/2\mathbb Z $. 

By spelling out the construction of the associated \data{} $ (L(\mathcal F),H(\mathcal F),P(\mathcal F))$, one finds the following. The \pairing{} is the orientation map $\pi_1(L,\ell) \to \mathbb Z/2\mathbb Z $. If  $E$ is orientable, then the \pairing{} is trivial, the Galois bundle is $L(\mathcal F) = L$, $H(\mathcal F) $ is the group of all formal diffeomorphisms $\mathrm{Diff}^{{\rm formal}}(\mathbb R^d, 0)_+$ whose differential has a positive determinant at $0$. This group contains ${\mathrm{GL}}_d(\mathbb R)_+ $ as a subgroup, and $P(\mathcal F)$ is the $H(\cF)$-bundle induced by the positive frame bundle $ {\mathrm{Fr}}_+(E)$ (= the ${\mathrm{GL}}_d(\mathbb R)_+ $ of basis or positive orientation), i.e.\ $P(\mathcal F)= \frac{{\mathrm{Fr}}_+(E) \times \mathrm{Diff}^{\rm formal}_0 (\mathbb R^d)_+}{\mathrm{GL}_d(\mathbb R)_+}   $.  
If $E$ is not orientable, then the \pairing{} is onto, so that the Galois bundle is the manifold $L(\mathcal F) $ of elements of norm $1$ in the determinant bundle of $E$ (for an arbitrary metric). The group $H(\mathcal F) $ is $\mathrm{Diff}^{{\rm formal}}(\mathbb R^d, 0)$. This group contains $\mathrm{GL}_d(\mathbb R) $ as a subgroup. The bundle $P(\mathcal F) $ is the $H(\mathcal F)$-bundle induced by the frame bundle $ {\mathrm{Fr}}(E)$, \textit{i.e.}\ $P(\mathcal F)= \frac{{\mathrm{Fr}}(E) \times \mathrm{Diff}_0 (\mathbb R^d)}{\mathrm{GL}_d(\mathbb R)}   $.
 \end{example}

\begin{example}
\normalfont
\label{ex:torus}
Consider (following an idea in \cite{Ryvkin2}) the formal singular foliation on
$\mathbb T^2 \times  \mathbb R$ with variables $\theta,\eta,t $ and the formal singular foliation along the leaf $ L=\mathbb T^2 \times  \{0\}  \simeq \mathbb T^2 $ 
is generated by
 $$  \frac{\partial}{\partial \theta} + t^5 \frac{\partial}{\partial t}, \frac{\partial}{\partial \eta} + t^6 \frac{\partial}{\partial t} , t^{10}\frac{\partial}{\partial t} .$$
 The transverse model is the formal singular foliation on $ \mathbb R$ generated by the vector field $t^{10}\frac{\partial}{\partial t}  $ .
 
 Here is the associated \data{} in this case:
\begin{enumerate}
\item[(a)] $\leaf(\mathcal F) = \mathbb R^2 $ is the universal cover of $\mathbb T^2 $.
\item[(b)] the group $H(\mathcal F)$ is the group of all formal diffeomorphims of $\mathbb R $ of the form
 $$t \mapsto t+ \mleft(mt^5+\frac{5}{2}m^2 t^{10}\mright)+ n t^6 + \sum_{k\geq 10} a_k t^k   $$
 for some $n,m \in \mathbb Z$ and $ a_k \in \mathbb R$
\item[(c)] the principal bundle $ P(\mathcal F)$ is the set of  formal diffeomorphisms of $\mathbb R $ of the form
$$ t \mapsto t+((a+m)t^5+\frac{5}{2}(a+m)^2 t^{10})+ (b+n) t^6 + \sum_{k\geq 10} a_k t^k   $$
 for some $n,m \in \mathbb Z$ and $ a,b,a_k \in \mathbb R $
the projection onto $L(\mathcal F) = \mathbb R^2 $ being given by $(a,b)$, the projection on $\mathbb T^2$ is given by the class of $ (a,b)$ modulo $\mathbb Z^2$, and the $H(\mathcal F)$-module structure being given by left composition.
 \end{enumerate}
\end{example}

\begin{example}
\label{ex:ppalbundleSFtriple}
\normalfont
For a formal singular foliation as in Example \ref{ex:ppalbundleSF} with $G$ acting formally and faithfully in $\mathbb R^d $ fixing 0, the \data{} is as follows. 
Since the transverse singular foliation $\mathcal T_0 $ is associated to the infinitesimal formal action on $\mathbb R^d $ of the Lie algebra $\mathfrak g $ of $G$, the formal action of $G$ preserves it. In other words, the formal action induces a group morphism $ \psi \colon G \hookrightarrow {\mathrm{Sym}}(\mathcal T_0)$. Also, the connected component $G_0$ of the unit of $G$ is mapped to ${\mathrm{Inner}}(\mathcal T_0) $, so that this group morphism induces a group morphism $ G/G_0 \to  {\mathrm{Out}}(\mathcal T_0)$. 
Now, as for any principal $G$-bundle $ Q \to L$, there is a natural group morphism $ \pi_1(L,\ell) \to G/G_0$.
The composition with the above morphism  $G/G_0 \to  {\mathrm{Out}}(\mathcal T_0)$ is the \pairing{}. In particular, the Galois cover is a quotient of  $L(\mathcal F) := Q/G_0$. By construction, $H(\mathcal F)= \psi(G) \cdot {\mathrm{Inner}}(\mathcal T_0) \subset {\mathrm{Sym}}(\mathcal T_0)  $. Notice that  $G  $ is a subgroup of $ H(\mathcal F)$. The principal  $H(\mathcal F)$-principal bundle $ P(\mathcal F)$ is induced by the principal $G$-bundle $Q$, i.e.\ $P(\mathcal F) = \frac{Q \times H(\mathcal F)}{G} $.
\end{example}

 \begin{rem}
 There is a natural group morphism from ${\mathrm{Out}} (\mathcal T_0) $ to $\mathbb Z/2\mathbb Z $ that associates to any outer automorphism $\bar{\phi} $  the sign of the determinant of the differential at $0 \in\mathbb R^d $ of any of its representative $\phi $ in ${\mathrm{Sym}} (\mathcal T_0) $.
 For any formal singular foliation along a leaf $L$, the composition of the \pairing{} with this morphism is the orientation group morphism $\pi_1 (L) \to \mathbb Z/2 \mathbb Z $ of the normal bundle $ T \to L$.
 \end{rem}

\begin{rem}
\label{rem:holonomygroupoid2}
There is a relation between \data{} and the holonomy groupoid of Remark \ref{rem:holonomygroupoid} is transitive, it has to be the gauge groupoid of a principal group bundle, see, \emph{e.g.}, \cite[\S 1.3]{MR2157566}. A line by line comparison shows that this  group is $H(\mathcal F)/{\mathrm{exp}}(I_0 \mathcal T_0) $, with $I_0 \subset C^{\mathrm{formal}} $ being the ideal of functions vanishing at $0$, and that the principal bundle in question is $ P(\mathcal F)/{\mathrm{exp}}(I_0 \mathcal T_0)  $.
\end{rem}

To conclude this section, let us present an alternative description of \data. 
It can be seen as a general property for Yang-Mills connections, more precisely from the Yang-Mills groupoid of section \ref{sec:YMgroupoid}; as already observed (infinitesimally) in \cite{My1stpaper, MyThesis}, based on extending Lie algebroids by Lie algebra bundles, \cite[\S 7.2]{MR2157566}; and (integrated) in \cite{SRFCYM}.
Following \cite[\S 1.3]{MR2157566}, to any transitive groupoid $\Gamma \rightrightarrows \leaf $, there is an associated principal bundle as follows:
fix $\ell$, and consider its isotropy group $ \Gamma_\ell^\ell = \Gamma_\ell \cap \Gamma^\ell$. By construction, $\Gamma_\ell$ is a principal $ \Gamma_\ell^\ell $-bundle over $ \leaf$ whose projection is given by the target arrow. 
(Moreover, $\Gamma $ is the Atiyah groupoid (also known as gauge groupoid) of this principal bundle.)  
When we consider the Yang-Mills groupoid ${\mathrm{YM}}(\mathcal F) \rightrightarrows \leaf $ of section \ref{sec:YMgroupoid}, we obtain a principal bundle whose isotropy group we denote by $H(\ell)$, while $P(\ell) $ will stand for $s^{-1}(\{\ell\}) $. By construction also, $H(\ell) $ is a subgroup of ${\mathrm{Sym}}(\mathcal T_\ell) $ containing ${\mathrm{Inner}}(\mathcal T_\ell) $ as a normal sub-group.

\begin{lemma}
\label{lem:His}
The image of $H(\ell)$ 
through the canonical projection $  {\mathrm{Sym}}(\mathcal T_\ell)  \to  {\mathrm{Out}}(\mathcal T_\ell)$ coincides with the image of the \pairing{}.
\end{lemma}
\begin{proof}
The construction of the Yang-Mills groupoid as the set of all possible parallel transportations, \emph{i.e.} Equation \eqref{eq:YMgroupoidDef} for $\ell_1=\ell_2=\ell $, implies that  $$ H(\ell) = \left\{ {\mathrm{PT}}_\gamma \circ \phi \, \middle| \, \gamma \in  {\mathrm{Path}}_{\ell}^{\ell} , \phi \in {\mathrm{Inner}}(\mathcal T_{\ell})\right\}  $$
which is precisely the content of the statement.
\end{proof}

We have associated to a given formal singular foliation along a leaf $\leaf $ upon choosing a point $\ell \in \leaf $:
\begin{enumerate}
\item a group $ H(\ell)$, i.e.\ the inverse image through $  {\mathrm{Sym}}(\mathcal T_\ell)  \to  {\mathrm{Out}}(\mathcal T_\ell)$ of the image of the \pairing{}
    \item a principal $H(\ell) $-bundle $P(\ell) $ over $ L$ whose quotient under $ {\mathrm{Inner}}(\mathcal T_\ell) $ is the Galois bundle associated to the kernel of the \pairing.
\end{enumerate}
Now, let us choose an identification of $ (T_\ell, \mathcal T_\ell)$ with the transverse singular foliation $(\mathbb R^d, \mathcal T_0)  $.  
The above data then becomes a \data.
Different identifications yield equivalent \datas. 

\subsection{Reconstruction}
\label{sec:reconstruction}

Now, we start from
\begin{enumerate}
\item a manifold $\leaf$,
\item a formal singular foliation $(\mathbb R^d, \mathcal T_0)$ near $0$ (that we call the \emph{transverse model}),
\item an equivalence class of \datas{} (see Def.\ \ref{def:transversedata}).
\end{enumerate}
Let $(\tilde L, H, P)$ be a representative of the equivalence class of \datas{}.
\begin{blue}


Let us construct a singular foliation with transverse model $ \mathcal T_0$ along the leaf $\leaf$. 

\textbf{Step 1: Construction of the formal neighborhood.}
The formal neighborhood is ``morally'' the associated space $ \frac{P \times \mathbb R^d}{H}$. But this is not  well-defined for many reasons (the action is formal, and the diffeology may be not clear). So let us define it explicitly.
Choose an open cover $ (U_i)_{i\in I}$ of $ \leaf$, and local sections $ s_i \colon U_i \mapsto P $, smooth with respect to the diffeology of $P$. Then define maps $ s_{ij}\colon U_i \cap U_j \longrightarrow H $ by: 
  $$  s_{i} (\ell)=s_{ij} (\ell) \cdot s_j(\ell) \hspace{1cm} \forall \ell \in U_i \cap U_j .$$
 These maps form a \v{C}ech $1$-cocycle.
In particular, they can be used to glue the formal neighborhoods 
 $(U_i \times \mathbb R^r)_{i\in I} $ of $ (U_i)_{i\in I}$.
More precisely, we use the $ s_{ij}$ to glue the sheaves of formal functions with $r$-variables on $U_i $, to yield a sheaf on $L$, which is by construction a
 formal neighborhood of $ \leaf$ that we denote by 
  $$   \frac{\coprod_{i\in I}  U_i \times \mathbb R^r}{\sim(s_{ij})}~.$$

 \textbf{Step 2: Construction of the formal singular foliation}
 This \v{C}ech $1$-cocycle  is not only valued in $H$, but is also a smooth map with respect to the diffeology of $H$ by definition of a principal bundle.
 By Lemma \ref{lem:simplifyBundle}, for any pair $i,j\in I$, $ s_{ij}$, seen as a formal diffeomorphism of the trivial singular foliation on $U_i\cap U_j  $ with transverse model $  \mathcal T_0$, is a symmetry of the latter. In particular, they can be used to glue the  trivial singular foliation along the leaf $ U_i$ with transverse model $  \mathcal T_0$
 with the  trivial singular foliation along the leaf $ U_j$ with transverse model $  \mathcal T_0$ into a  formal singular foliation along the leaf $ U_i\cup U_j$, with transverse model $  \mathcal T_0$. 
 Repeating the procedure for all indices, one obtains  a formal singular foliation along the leaf $\leaf$ that we denote by $ \mathcal F((s_i)_{i\in I})$ on $  \frac{\coprod_{i\in I}  U_i \times \mathbb R^r}{\sim(s_{ij})}$.

\textbf{Step 3: The previous construction does not depend on the choices made in its construction.}
 Let us show that the latter formal singular foliation along the leaf $ \leaf$ does not depend on the choice of the open cover and of the local sections. Since it is clearly the same for an open cover and local sections $(s_i)_{i\in I} $ and a refinement of that open cover and the restrictions of each $ s_i$ to this refinement, it suffices to show that 
for a given open cover, it does not depend on the choice of local sections. Let us choose a different set of sections, say $( s_i')_{i \in I}$.
There exists smooth maps $t_i: U_i \longrightarrow H $ such that 
 $$  s_{i}' (\ell')=t_{i} (\ell')s_i(\ell') \hspace{1cm} \forall \ell' \in U_i $$ 
Altogether, the maps $t_\bullet \coloneqq (t_{i})_{i \in I} $ are an isomorphism of formal neighborhoods: $$ {\mathrm{id}}  \times t_\bullet \colon \coprod_{i\in I}  U_i \times \mathbb R^r \simeq \coprod_{i\in I}  U_i \times \mathbb R^r  $$
that go the quotient to define 
an isomorphism  of formal neighborhoods of $L$ $$  \frac{\coprod_{i\in I}  U_i \times \mathbb R^r}{\sim(s_{ij})} \simeq \frac{\coprod_{i\in I}  U_i \times \mathbb R^r}{\sim(s_{ij}')} . $$
By Lemma  \ref{lem:simplifyBundle}, each $t_i$ is an isomorphism of singular foliations, hence so is its quotient map.  Last,   equivalent \datas{} yield cocycles which differ by a symmetry $\psi $ in the sense that: 
  $$  s_{ij}' = \psi \circ s_{ij} \circ \psi^{-1} $$
Now,  $ {\mathrm{id}}  \times \psi $ is an isomorphism of singular foliation 
 : $${\mathrm{ id}} \times \psi \colon  \coprod_{i\in I}  U_i \times \mathbb R^r \simeq \coprod_{i\in I}  U_i \times \mathbb R^r  $$
that go the quotient to define 
an isomorphism  of formal neighborhoods of $L$: $$ \frac{\coprod_{i\in I}  U_i \times \mathbb R^r}{\sim(s_{ij})} \simeq \frac{\coprod_{i\in I}  U_i \times \mathbb R^r}{\sim(s_{ij}')} . $$ Hence equivalent \datas{} give isomorphic formal singular foliations. 

Let us make a proper definition, for future reference.

\begin{deff}
\label{def:rec}
We call the formal singular foliation above the formal singular foliation reconstructed out of the \data{}  $(\tilde{L},H,P)$. We denote it by $\cF(\tilde{L}, H, P)$. 
\end{deff} 
 
\begin{lemma}\label{lem:transversedataisassocdata}
The \data{} of the singular foliation  $\mathcal F ( \tilde{L},H,P )$ is $(\tilde{L},H,P) $. 
\end{lemma}

\begin{proof}
 Let $ (s_{ij})$ be an $H$-valued \v{C}ech $1$-cocycle defining $P$.
 Notice that $ (\bar{s}_{ij})_{i,j\in I}  $ is a $1$-cocycle defining the Galois bundle $ \tilde L$.
 
Recall that  $\cF(\tilde{L}, H, P)$ is constructed on a formal neighborhood of $L$ defined as the quotient:
$$\frac{\coprod_{i\in I}  U_i \times \mathbb R^r}{\sim(s_{ij})}.$$

We choose for a given in $\ell\in U_{i_0} $ to identity $ \mathcal T_\ell$ with $ \mathcal T_0$ using the local isomorphism with $ U_{i_0} \times \mathbb R^r$.
Since the above quotient gives on each open subset $U_i$, a local isomorphism between $\cF(\tilde{L}, H, P)$ to the trivial singular foliation along the leaf $ U_i$ with transverse model $  \mathcal T_0$, 
any $\mathcal F$-connection on $\cF(\tilde{L}, H, P)$  induces an  $\mathcal F$-connection on the trivial singular foliation on $ U_i\times \mathbb R^d$. Now, since parallel transportation along a given path for an $ \mathcal F$-connection for a trivial singular foliation is just the flow of a time-dependent vector field in $ \mathcal T_0$, it is always an inner symmetry. Hence, parallel transportation along any path $\gamma_i $ in $ U_i$ becomes, under this isomorphism, an inner symmetry of $ \mathcal T_0$ that we denote by $ \Phi(\gamma_i)$. We can now describe, for any loop $ \gamma$ based at $ \ell \in L$, parallel transportation along that loop  as follows. 
 This loop intersects several open subsets of the chosen open cover, say $ U_{i_0}, U_{i_1}, \dots, U_{i_k}$ with the understanding that $ i_k=i_0$. More precisely, $ \gamma$ is the union of paths $ \gamma_0, \dots, \gamma_k$ valued in $ U_{i_0}, \dots, U_{i_k}$ with the understanding that $ \gamma_0$ starts at $ \ell$ and ends at a point $ \ell_1$ where $ \gamma_1$ also starts, and so on, until $ \gamma_k$ that starts at a point $ \ell_{k-1} $ and ends at $ \ell_{k}=\ell$.
Parallel transportation $\mathrm{PT}_\gamma $ along $ \gamma $  
is, under the previously chosen isomorphism $ \mathcal T_\ell\simeq \mathcal T_0$, a symmetry of $ \mathcal T_0$ given by 
  \begin{equation}\label{eq:lifttoP}  \Phi(\gamma_k) \circ s_{i_ki_{k-1}}(\ell_k)  \circ \Phi(\gamma_{k-1})  \circ \dots \circ  s_{i_2i_1}(\ell_2) \circ \Phi(\gamma_1) \circ  s_{i_1i_0}(\ell_1)  \circ \Phi(\gamma_0) \end{equation}
In particular, it always belongs to $ H$. Now, consider the subset $H'\subset H$ obtained by composing elements as above, computed with respect to any loops on $L$, with inner symmetries of $ \mathcal T_0$. Let us show that $H=H'$. To start with, $H'$ contains  $\mathrm{Inner}(\mathcal T_0) $ by construction. Second,  
for a  loop as above, the class modulo $\mathrm{Inner}(\mathcal T_0) $ of the parallel transportation is
\begin{equation}\label{eq:liftGalois}
    \bar{s}_{i_ki_{k-1}}(\ell_k)  \circ  \dots \circ  \bar{s}_{i_2i_1}(\ell_2) \circ   \bar{s}_{i_1i_0}(\ell_1)
\end{equation}
 where the horizontal bar stands for the quotient modulo inner symmetries.  The set of such elements is precisely the set of all possible lifts in the Galois bundle $ \tilde L\to L$ with respect to all loops in $L$, it is therefore  $\pi_1(\tilde L,L)$ by definition of a Galois bundle. This shows two points: that the Galois bundle associated to  $\cF(\tilde{L}, H, P)$ is $ \tilde L \to L$ since the expression \eqref{eq:liftGalois} is equal to the identity element if and only if the loop $ \gamma$ has a lift to a loop in $ \tilde L$. It also shows that $H=H'$ since $H'$ has the same quotient as $H$ modulo $\mathrm{Inner}(\mathcal T_0) $. 

Let us show that the principal $H$-bundle $ \hat P$ over $L$ that appears in the \data{} of $\cF(\tilde{L}, H, P)$ coincides with $ P$.
Let us compare the sets  of isomorphisms $ P_{\ell'} \simeq P_\ell \simeq H $ obtained using local trivializations and parallel transportations along loops, for both principal bundles. 
For $P$,  we choose $ \gamma $ a fixed path as \eqref{eq:lifttoP}, except that it starts at $ \ell$ and ends at $ \ell'$ but may not be a loop, and we have a natural identification with $H$ given by:
$$  P_{\ell'} \simeq  \left\{  s_{i_ki_{k-1}}(\ell_k)  \circ  \dots  \circ s_{i_2i_1}(\ell_2) \circ   s_{i_1i_0}(\ell_1) \circ h  \, \middle| \, h \in H \right\}   .$$
 For $\hat P$, the fiber over $ \ell'$ is the set
$$  \hat P_{\ell'} \simeq  \left\{    \Phi(\gamma_k) \circ s_{i_ki_{k-1}}(\ell_k)  \circ \Phi(\gamma_{k-1})  \circ \dots \circ  s_{i_2i_1}(\ell_2) \circ \Phi(\gamma_1) \circ  s_{i_1i_0}(\ell_1)  \circ \Phi(\gamma_0) \circ \psi \right\}   $$
where $ \psi $ runs over arbitrary inner automorphisms of $ \mathcal T_0$, $ \gamma$ runs over arbitrary paths from $ \ell$ to $ \ell'$, and $ i_0,\dots,i_k, \ell_0, \dots, \ell_k $ are as in \eqref{eq:lifttoP}.  
Now, these two sets are identical by definition of $ H$.
The claim therefore follows. 
\end{proof}
\end{blue}


\begin{rem}
The above reconstruction procedure can be described in a very easy but not very rigorous way, following the construction of Example \ref{ex:ppalbundleSF}, applied to the $H$ action on $ \mathbb R^d$ and to the principal bundle  $ P \to L$. 
It amounts to to consider the associated bundle for the  $H$ action on $ \mathbb R^d$ over $\leaf$:
\begin{equation}
\label{eq:defT}
T \coloneqq \frac{P \times \mathbb R^d }{H}.
\end{equation}
with $H$ acting on $\mathbb R^d$ through its inclusion into ${\mathrm{Sym}}(\mathcal T_0) \subset {\mathrm{Diff}}(\mathbb R^d, 0) $, and to equip it with the projection on $T$ of the direct product of the singular foliation associated to the $H$-action on $ \mathbb R^d$, i.e.\ $ \mathcal T_0$, with the singular foliation of all vector fields on $P$. 
However, since $H$ only acts formally on $\mathbb R^d $, and is an infinite dimensional diffeological space, it is not trivial to make a rigorous proof out of this idea.
\end{rem}

\subsection{Conclusion}

In the two previous sections, we went from a formal singular foliation to a \data{} and from a \data{} to a formal singular foliation. 


\begin{theorem}\label{thm:classification}
    Let $\leaf$ be a 
 manifold and $(\mathbb R^d, \mathcal T_0) $ be a formal singular foliation near $0$. Then there is a one-to-one correspondence between:
\begin{enumerate}
\item[(i)] formal singular foliations along the leaf $\leaf $ with transverse model $(\mathbb R^d, \mathcal T_0)$,  
\item[(ii)] equivalence classes of \datas{} (see  Definition \ref{def:transversedata}),
\item[(iii)] equivalences classes made of a group morphism $\Xi \colon \pi_1(L) \to {\mathrm{Out}}(\mathcal T_0)  $ and a principal $H$-bundle $P$, with $H \subset {\mathrm{Sym}}(\mathcal T_0)$ the inverse image through the natural projection ${\mathrm{Sym}}(\mathcal T_0) \to {\mathrm{Out}}(\mathcal T_0) $ of $ \Xi(\pi_1(L))$.
\end{enumerate}
\end{theorem}

\begin{proof}
\begin{blue}
The equivalence of items (ii) and (iii) was established in Remark \ref{rem:Pgivestriple}. Proposition \ref{prop:leafdata} describes a map from item (i) to item (ii). Lemma \ref{lem:transversedataisassocdata} states that it is surjective by providing an explicit left inverse of the previous map.  We have to check that this map is also injective.
Let $ \mathcal F,\mathcal F'$ be two formal singular foliations near the leaf $ L$ that have the same \data{} $ (\tilde L,H,P,)$. By construction, $P$ is a set of isomorphisms from a fiber $ \mathcal T_{\ell'}$ of $\mathcal F$ to $ \mathcal T_0$.
But $P$ is also a set of isomorphisms from a fiber $ \mathcal T_{\ell'}'$ of $\mathcal F'$ to $ \mathcal T_0$.
Any choice of local sections of $(s_i)_{i \in I} $ of $P$ gives therefore isomorphisms $  \mathcal T_{\ell'} \simeq \mathcal T_{0}$ and  $  \mathcal T_{\ell'}' \simeq \mathcal T_{0}$ for all $\ell' \in U_i$. Composing the first isomorphism with the inverse of the second one,  we get an isomorphism  of singular foliation $   \mathcal T_{\ell'} \simeq \mathcal T_{\ell'}'$. Now, this isomorphism does not depend on the choice of the local sections $ (s_i)_{i\in I}$, and therefore not on the choice of the index $i$ such that $ \ell' \in U_i$. These local isomorphisms glue therefore to a global isomorphism $ \psi$ from the formal neighborhood of $ L$ on which $ \mathcal F$ is defined to the one on which $ \mathcal F'$ is defined, that intertwines the induced vertical singular foliations. 
Let us show that it intertwines $\mathcal F  $ and $ \mathcal F'$.
The local section $ s_i \colon U_i \to P$ defines in fact two isomorphisms of singular foliations: 
  $$ \mathcal F|_{U_i}  \simeq U_i \times \mathcal T_0 \hbox{ and } \mathcal F|_{U_i}'  \simeq U_i \times \mathcal T_0 .$$
  By construction, $ \psi$ is the composition of the first one with the inverse of the second one.
 Hence, the global isomorphism $\psi$ of formal neighborhoods intertwines $\mathcal F $ and $ \mathcal F'$.\end{blue} 
\end{proof}



\begin{rem}
Given an embedded leaf $L$ of a smooth singular foliation on $M$, the constructions above can be applied to the induced formal singular foliation
of Example \ref{decomp:step5}.
It yields some \data{}. 
Theorem \ref{thm:classification} implies that two smooth singular foliations that have the same leaf $L $ admitting the same \datas{} are formally equivalent along $L$.
It does not mean that they are isomorphic in neighborhoods of these leaves.
\end{rem}

\section{Classification II: Particular Cases}
\label{sec:ClassII}

We now use Theorem \ref{thm:classification} to give an even simpler answer to Question \ref{EssentialQuestions}.

The great advantage of the next description is that the principal bundle is now a principal bundle for a {\textit{finite dimensional}} Lie group, and not a diffeological Lie group.

\begin{theorem}
\label{thm:classificationSimplified}
    Let $\leaf$ be a 
 manifold and $(\mathbb R^d, \mathcal T_0) $ be a formal singular foliation near $0$. Then there is a one-to-one correspondence between:
\begin{enumerate}
\item[(i)] formal singular foliations along the leaf $\leaf $ with transverse model $(\mathbb R^d, \mathcal T_0)$, and 
\item[(ii)] equivalence classes of triples as follows:
\begin{enumerate}
  \item 
a Galois cover $\tilde{L} \to L $ (we denote by $ \pi_1(\tilde L, L)$ its Galois group),
\item a group extension of the form\footnote{Said otherwise: a subgroup $H$ of ${\mathrm{Sym}}(\mathcal T_0)$ containing $\mathrm{Inner}(\mathcal T_0) $ and equipped with a surjective group morphism onto $\pi_1(\tilde L, L) $ with kernel $\mathrm{Inner}(\mathcal T_0)$.} $ \xymatrix{   \mathrm{Inner}(\mathcal T_0) \ar@{^(->}[r]&H \ar@{->>}[r]&  \pi_1(\tilde L,\leaf)  }  $ with $H \subset {\mathrm{Sym}}(\mathcal T_0)$, 
\item  an extension of $ \tilde{L} \to L  $ to an $\tilde{H}_2 := H/ {\mathrm{Inner}}(\mathcal T_0)_{\geq 2} $-principal bundle\footnote{$\tilde{H}_2 $ is of course a finite dimensional Lie group}.
\end{enumerate}
\item[(iii)] equivalence classes  of pairs made of 
\begin{enumerate}
  \item[(a)] a group morphism $\Xi$ from $\pi_1(\leaf) $ to ${\mathrm{Out}}(\mathcal T_0) $ (= the \pairing{})  
\item[(b)]  an extension of $ \tilde{L} \to L  $ to an $ \tilde{H}_2$-principal bundle, where $\tilde{H}_2 $ is the inverse image of the image of $\Xi$ through the natural projection 
$    \frac{\mathrm{Sym}(\mathcal T_0)}{\mathrm{Inner}(\mathcal T_0)_{\geq 2}} \to  {\mathrm{Out}}(\mathcal T_0) $.
 \end{enumerate}
\end{enumerate}
In cases (ii) and (iii) above, the equivalence classes are obtained by considering the natural actions of $\mathrm{Sym}(\mathcal T_0)$ on items (b) and (c) in (ii) and of $\mathrm{Out}(\mathcal T_0)$ on items (a) and (b) of (iii) respectively.
\end{theorem}
\begin{proof} 
Together with Theorem \ref{thm:classification}, Proposition \ref{prop:simplifyBundle} implies the equivalence between (i) and (ii).
The equivalence between (ii) and (iii) follows the same lines as the one in 
Remark \ref{rem:Pgivestriple}.
\end{proof}

\begin{rem}
\label{rem:holonomygroupoid3}
Since $I_0 \mathcal T_0 \subset (\mathcal T_0)_{\geq 2} $, it follows from Remark \ref{rem:holonomygroupoid2}
that the principal bundle in item c) of Theorem \ref{thm:classificationSimplified} is in fact a quotient of the principal bundle associated to the holonomy groupoid of the leaf.
\end{rem}

\begin{rem}\label{rem:CookingRecipe2}
\normalfont
Let us recapitulate the construction of the formal singular foliation $\cF$ along the leaf $L$ associated to a triple $(\tilde{L},\tilde{H}_2,\tilde{P}_2) $ as in item (ii) of Theorem \ref{thm:classificationSimplified}.
The first step consists in building an $H$-principal bundle $ P$ over $L$ whose quotient modulo $ {\mathrm{Inner}}(\mathcal T_0)_{\geq 2} $ is  $\tilde{P}_2 $. This requires infinitely many operations, as in the proof of Proposition \ref{prop:simplifyBundle}, to lift the $1$-cocycle on $ L$ defining $\tilde{P}_2 $ to a converging sequence of $1$-cocycles defining principal $H/{\mathrm{Inner}}(\mathcal T_0)_{\geq n}$-bundles  $\tilde{P}_n $, for all $ n \geq 2$, extending one another. 
When these operations are completed, one obtains a \data{} $ (\tilde{L},H,P)$. 
One can then apply the reconstruction procedure of
Section \ref{sec:reconstruction}. 
\end{rem}

\begin{rem}
\label{rem:througInfinite}
An important consequence of the description given in Remark \ref{rem:CookingRecipe2} is that, in general, there might be no way to reconstruct the singular foliation $\mathcal F $ staying within finite dimensional differential geometry. 
The issue is that there is in general no  $ k \in \mathbb N$ such that a group action of ${\mathrm{Inner}}(\mathcal T_0)/{\mathrm{Inner}}(\mathcal T_0)_{\geq k} $ on $ \mathbb R^d$ exists, and describes $\mathcal T_0  $.  (If yes, then we can apply the construction of Example \ref{ex:ppalbundleSF} for a finite dimensional principal bundle.)
\end{rem}

\subsection{Simply-connected leaves}

Let us now consider the case of simply connected leaves, as in \cite{Ryvkin2}.



\begin{cor}
\label{coro:simply-connected-case}
 Let $\leaf$ be a 
 simply-connected manifold and $(\mathbb R^d, \mathcal T_0) $ be a formal singular foliation near $0$. Then there is a one-to-one correspondence between:
\begin{enumerate}
\item[(i)] formal singular foliations along the leaf $\leaf $ with transverse model  $(\mathbb R^d, \mathcal T_0)$, and 
\item[(ii)] $\mathrm{Inner}(\mathcal T_0)/\mathrm{Inner}(\mathcal T_0)_{\geq 2} $-principal bundles over $\leaf $.
\end{enumerate}
\end{cor}
\begin{proof}
For $\leaf $ a simply connected manifold, for any \data{} $(L(\mathcal F),H(\mathcal F),P(\mathcal F))$ we have $L(\mathcal F)= L $ and $H(\mathcal F) ={\mathrm{Inner}}(\mathcal T_0)$.
The result then follows from Theorem \ref{thm:classificationSimplified}.
\end{proof}

\begin{rem}\label{rem:TrivialOuterHolonomyClassif}
We actually can classify all foliations with a trivial \pairing{}, not just the simply-connected case:
There is a one-to-one correspondence between:
\begin{enumerate}
\item[(i)] formal singular foliations along the leaf $\leaf $ with transverse model  $(\mathbb R^d, \mathcal T_0)$ and with trivial \pairing{}, and 
\item[(ii)] $\mathrm{Inner}(\mathcal T_0)/\mathrm{Inner}(\mathcal T_0)_{\geq 2} $-principal bundles over $\leaf $.
\end{enumerate}
In particular, if $\mathcal T_0$ is made of all vector fields vanishing at $0$, so that $ {\mathrm{Out}}(\mathcal T_0) $ is $\mathbb Z/2\mathbb Z $, the \pairing{} is just the orientation map of $ T$, and we re-obtain the one-to-one correspondence 
 between ${\mathrm{GL}}_{d} (\mathbb R)$-principal bundles over $ L$ and singular foliations along a leaf $L$ admitting vector fields vanishing at $0$ as transverse model on an orientable transverse bundle $T$.  
\end{rem}

\begin{rem}
\label{rem:holonomygroupoid4}
It follows from Remark \ref{rem:holonomygroupoid3}
and Corollary \ref{coro:simply-connected-case}
that formal singular foliations along a simply-connected leaf $L$ with given transverse models are entirely determined by their holonomy groupoids along $L$ (in fact, even by a quotient of their holonomy groupoids).
\end{rem}


\begin{example}\label{ex:IntroQuestionsAbout2SphereAsLeaf}
\normalfont
We can also answer several questions of the introduction. Remark \ref{rem:concentric} and Corollary \ref{coro:simply-connected-case} imply that the only formal singular foliation along the leaf $\mathbb S^2 $ with transverse model the "spirals"  is the trivial product of $\mathbb S^2 $ with "spirals". In particular, there is no formal singular foliation on $\rmT\mathbb S^2 $ admitting the spiral singular foliation as transverse model, for its normal bundle would not be the trivial bundle. This answers question \textbf{Q2b}.
By contrast, there is one for which the transverse model is by concentric circles. For this transverse model, $ {\mathrm{Inner}}(\mathcal T_0)/{\mathrm{Inner}}(\mathcal T_0)_{\geq 2} \simeq S^1 $ and one can consider the usual $ S^1$-bundle given by the unit circles of $\rmT \mathbb S^2 $ for a given metric. The corresponding singular foliation is associated to the $ {\mathrm{SO}}(3) $-action on $ \rmT \mathbb S^2$, equivalently the singular foliation of all vector fields  $ X \in \mathfrak X(\rmT\mathbb S^2)$ such that $X[\phi ]=0$, with $ \phi$ the square of the norm of same metric on $\rmT\mathbb S^2 $. This answers question {\textbf{Q2a}}. \begin{blue} As already mentioned, there are also elementary answers, but we just presented a systematic one.\end{blue}
\end{example}

As a corollary of Corollary \ref{coro:simply-connected-case}, we recover a formal equivalent of the result of Ryvkin and the second author \cite{Ryvkin2}.

\begin{cor}
Let $\leaf$ be a 
 simply-connected manifold and $(\mathbb R^d, \mathcal T_0) $ is made of  vector fields vanishing at least quadratically at zero.
 Then the unique formal singular foliation along the leaf $\leaf $ and transverse model   $(\mathbb R^d, \mathcal T_0) $ is the trivial product of $(L, \mathfrak X(L))$ with $(\mathbb R^d, \mathcal T_0) $.
\end{cor}
\begin{proof}
In this case, $\mathrm{Inner}(\mathcal T_0)/\mathrm{Inner}(\mathcal T_0)_{\geq 2}=\{0\}$. The result is then obvious.
\end{proof}

Here is an other consequence of Corollary \ref{coro:simply-connected-case}, since any principal bundle over a contractible manifold is trivial.

\begin{cor}
Let $\leaf$ be a contractible manifold.
 Any formal singular foliation along the leaf $\leaf $ is the trivial product of $(L, {\mathfrak X}(L))$ with its transverse model.
\end{cor}

In fact, we can unify and generalize the last two corollaries by making use of the fact that a principal bundle over a simply connected manifold is trivial if and only if it admits a flat connection.

\begin{cor}\label{cor:TrivalOutHolAndTrivialFol}
A formal singular foliation $\cF$ along the leaf $\leaf $ with trivial \pairing{} is the trivial product of $(L, {\mathfrak X}(L))$ with its transverse model if and only if the \data{} (that is, the principal $\mathrm{Inner}(\cT_0)$-bundle) is trivial.

If $\leaf$ is simply connected, then $\cF$ is such a trivial product if and only if there is a  flat formal  $\cF$-connection.
\end{cor}

\subsection{Transversally quadratic singular foliations} 

Assume that the transverse model $ \mathcal T_0 $ is made of formal vector fields that vanish at least quadratically at $0$, \textit{i.e.}\  \emph{vanishes quadratically}. In this case, the exponential map  $ \mathcal T_0 \longrightarrow   {\mathrm{Inner}}(\mathcal T_0)$ 
 is one-to-one and is compatible with the diffeologies, so that 
 ${\mathrm{Inner}}(\mathcal T_0)$ is contractible by Lemma \ref{lem:exp2}.

\begin{cor}
\label{cor:transvQuadra}
Let $\leaf $ be a manifold and $(\mathbb R^d,\mathcal T_0) $ be a formal singular foliation near $0$ which vanishes quadratically.
There is a one-to-one correspondence between the three following sets:
\begin{enumerate}
\item[(i)] singular foliations along a leaf $\leaf $ with a transverse model $(\mathbb R^d,\mathcal T_0) $, and
\item[(ii)] equivalence classes of pairs made of a Galois cover $\tilde{L} \to L $ together with a group extension of the form
$$ {\mathrm{Inner}} (\mathcal T_0)\hookrightarrow H \to \pi_1(\tilde{\leaf}, \leaf) $$
where $H \subset {\mathrm{Sym}}(\mathcal T_0) $. The equivalence relation is given through conjugation by an element in ${\mathrm{Sym}}(\mathcal T_0) $.
\item[(iii)] equivalence classes group morphisms $\pi_1(L) \to {\mathrm{Out}}(\mathcal T_0) $. The equivalence relation is given through conjugation by an element of ${\mathrm{Out}}(\mathcal T_0) $.
\end{enumerate} 
\end{cor}
\begin{proof}
The equivalence between (i) and (ii) is an immediate consequence of Theorem \ref{thm:classification} and Proposition \ref{prop:simplifyBundle}, since in this case,  $ H / {\mathrm{Inner}}(\mathcal T_0)_{\geq 2} \simeq \pi_1(\tilde L, L) $, so that connected diffeological $H$-principal bundles over $\leaf $ are in one-to-one correspondence with Galois covers of $\leaf $. The map from (ii) to (iii) consists in mapping $\pi_1(L) $ to its quotient $\pi_1(\tilde{L},L)  $ then in mapping the latter to ${\mathrm{Out}}(\mathcal T_0) $ by using the inclusion $  H \subset {\mathrm{Sym}}(\mathcal T_0)$. The inverse consists to associating to a group morphism $ \phi$ as in (iii) the inverse image $H$ of ${\mathrm{Im}}(\phi) $ in ${\mathrm{Sym}}(\mathcal T_0) $ and the Galois cover $L^u/ {\mathrm{Ker}}(\phi) $, with $L^u$ the universal cover of $L$. 
\end{proof}

 \begin{example}
 \normalfont
For regular foliations, the transverse model is $\mathcal T_0=0 $. Corollary \ref{cor:transvQuadra} implies that there is a one-to-one correspondence between formal regular foliations along the leaf $\leaf $ and   group morphisms from $\pi_1(L) $ to $ {\mathrm{Diff}}(\mathbb R^d,0) $. This corresponds to the map called holonomy in the theory of regular foliations \cite{GroupoidBasedPrincipalBundles}.
 \end{example}

\begin{example}
\normalfont
We recover the description obtained by Francis in \cite{PhD_Francis,francis2023singular} for codimension $1$ foliations. In this case, there exists  $ n\geq 1 $ such that that $\mathcal T_0 $ is generated by $ t^n \partial_t$. 
For $n=1$, it amounts to classify line bundles by Example \ref{ex:vectorBundles}. For $ n \geq 2$, it amounts to consider equivalences classes of group morphisms from $ \pi_1(L)$ to the group of $(n-1)$-jets. 
\end{example}

\begin{example}
\normalfont
We recover the description obtained by Bischoff, del Pino, and Witte in \cite{BDW} for vector fields tangent to a submanifold up to a certain order, \emph{i.e} in the case when the transverse model is made of all vector fields vanishing at order $k+1$ on $\mathbb R^d$. More precisely, we recover Theorem 5.5 in \cite{BDW} (more precisely the formal equivalent of this result).
Both proofs are based on the observation that we cannot reconstruct the singular foliation as a quotient of an action of $\pi_1(L)$, and one has to use some infinite dimensional geometry. For us, the reconstruction consists in considering a group morphism from $ \pi_1(L)$ to the group of $k$-jets as a flat principal $J_k(\mathbb R^d)$ bundle ($J_k(\mathbb R^d)$ being the group of $k$-jets of diffeomorphism to a principal $\mathrm{Diff}(\mathbb R^ d,0)$-bundle), then dividing this bundle by the subgroup $H \subset \mathrm{Diff}(\mathbb R^ d,0)$ of all formal diffeomorphism whose $k$-jet lies in the image of the group morphism above. We base ourselves on the contractibility of $ \mathrm{Diff}(\mathbb R^ d,0)_{\geq k+1}$, which is also used in Proposition 5.8 in \cite{BDW}.
\end{example}





\begin{rem}
\label{rem:iscontractible}
The conclusion of Corollary \ref{cor:transvQuadra} still holds under the weaker assumption that the Lie group $\mathrm{Inner}(\mathcal T_0)/{\mathrm{exp}}(({\mathcal T_0)_{\geq 2}})$ is contractible. 
\end{rem}

\subsection{Singular foliations admitting flat \texorpdfstring{$ \mathcal F$-connections}{F-connections}}

Recall from Section \ref{sec:transver} that formal $\mathcal F $-connections always exist. But they need not be Lie algebra morphisms, \textit{i.e.}\ they might be non-flat. Here is an obvious result:

\begin{lemma}
\label{lem:existenceflat}
A formal singular foliation $ \mathcal F$ along the leaf $\leaf$ and transverse model  $(\mathbb R^d, \mathcal T_0) $ admits a flat $ \mathcal F$-connection if and only if the principal bundle $P(\mathcal F) $ in the \data{} admits a flat principal bundle connection.
\end{lemma}

If a formal singular foliation $ \mathcal F$ along the leaf $\leaf$ and transverse model  $(\mathbb R^d, \mathcal T_0) $ admits a flat $ \mathcal F$-connection, then parallel transportation induces a group morphism:
 \begin{equation}
\label{eq:Psi0}
 \Psi \colon \pi_1(\leaf,\ell) \longrightarrow \mathrm{Sym}(\mathcal T_\ell)
 \end{equation}
 Here  $\mathcal T_\ell $ is the transverse formal  singular foliation at $\ell \in L $. Upon identifying the fiber  $(T_\ell,\mathcal T_\ell) $
 with $(\mathbb R^d, \mathcal T_0) $, we can see this group morphism as a group morphism, defined up to a conjugation by an element in $\mathrm{Sym}(\mathcal T_0)$:
 \begin{equation}
\label{eq:Psi}
 \Psi \colon \pi_1(\leaf) \longrightarrow \mathrm{Sym}(\mathcal T_0)
 \end{equation}
 that lifts the \pairing{} of the singular foliation $\mathcal F $. (Notice that since this map is defined up to conjugation by an element in $\mathrm{Sym}(\mathcal T_0)$, we do not need to say with respect to which point of $ L$ we compute the fundamental group of $L$).

 The singular foliation $\mathcal F $ can be reconstructed as follows: Consider on $L^u \times \mathbb R^d$ ($L^u $ being again the universal cover of $L$) the trivial product of the singular foliation of all vector fields on $ L^u$
  with the formal singular foliation $\mathcal T_0 $ on $\mathbb R^d $. Then the group morphism \eqref{eq:Psi} can be used to let this formal singular foliation go down to 
 $\frac{ L^u \times \mathbb{R}^d }{\pi_1(L)} $.
 This formal singular foliation coincides with $\mathcal F $. 

\begin{prop}\label{prop:FlatFoliationClassification}
A formal singular foliation $ \mathcal F$ along the leaf $\leaf $ admits a  flat formal $\mathcal F $-connection if and only if there
exists a group morphism $\Psi \colon \pi_1(\leaf) \to {\mathrm{Sym}}(\mathcal T_0) $
such that its associated \data{} is of the form:
\begin{enumerate}
\item $L(\mathcal F)$ is $L^u/K $ with $K=\Psi^{-1}({\mathrm{Inner}}(\mathcal T_0))$ and $\leaf^u $  the universal cover of $\leaf$.
\item $H(\mathcal F)$ is the subgroup of $ {\mathrm{Sym}}(\mathcal T_0) $ of all elements of the form $g \circ \Psi(\bar{\gamma}) $ for $g \in {\mathrm{Inner}}(\mathcal T_0)$ and $\bar{\gamma} \in \pi_1(\leaf)$.
\item The principal bundle $P(\mathcal F)$ is given by $$ P (\mathcal F):= \frac{\leaf^u \times H}{\pi_1(\leaf)} $$
 where the right-action of $\pi_1(\leaf) $ is given by
 $$ (\ell, h) \cdot \bar{\gamma}  = \bigl(\ell\cdot \bar{\gamma}, h \Psi(\bar{\gamma})\bigr)$$ 
 for all $ h \in H(\mathcal F), \bar{\gamma} \in \pi_1(\leaf), \ell \in \leaf^u$.
\end{enumerate}
\end{prop}
Notice that, as a group, $H(\mathcal F)$ in item (2) above is the quotient of the semi-direct product $\pi_1(\leaf) \ltimes {\mathrm{Inner}}(\mathcal T_0) $ by the normal sub-group $(k,\Psi(k)) $ with $k \in K$.
\begin{proof}
Assume that a flat formal $ \mathcal F$-connection exists. Then  the induced connection on $P(\mathcal F)$ is also flat (see Lemma \ref{lem:existenceflat}), so that there is a map $ \Psi \colon \pi_1(\leaf) \to H(\mathcal F)$ and $P(\mathcal F)$ is given as in the third item. It is obvious that $ H(\mathcal F) $ and  $ L(\mathcal F)$ are given as in two first items.
\end{proof}

\begin{example}\label{ex:S1asleaf}
\normalfont
If $L=S^1 $, then any $ \mathcal F$-connection is flat, and has to be of the form described in this section. Since $\pi_1(L) =\mathbb Z $, $ \Psi$ is determined by the image of $1 \in \mathbb Z$ in ${\mathrm{Out}}(\mathcal T_0) $. The conjugacy class of this element entirely determines $\mathcal F $. Formal singular foliations along the leaf $S^1 $ with transverse model $(\mathbb R^d, \mathcal T_0) $ are therefore in one-to-one correspondence with conjugacy classes of the group ${\mathrm{Out}}(\mathcal T_0) $.
\end{example}

\begin{rem}
It is \emph{not} true that if the \pairing{} is trivial, then it implies the existence of a flat formal  $\cF$-connection. The remaining principal bundle in the \data{} can still be non-flat in such a situation. 

Non-flat orientable vector bundles (see Example \ref{ex:vectorBundles}) provide examples with trivial \pairing{} but no flat formal $\cF$-connection; also Example \ref{ex:HopfFibration} provides such an example, because, assuming the contrary, a flat formal $\cF$-connection would induce 
a flat connection on the Hopf fibration, giving rise to a trivialisation of the Hopf fibration due to the fact that $\mathbb{S}^4$ is simply connected.
\end{rem}

Last, notice that if the outer holonomy is trivial,  then the group morphism in \eqref{eq:Psi} has values in $\mathrm{Inner}(\cT_0)$, leading to the following corollary.

\begin{cor}
A formal singular foliation $ \mathcal F$ along the leaf $\leaf $ and trivial \pairing{} admits a  flat formal $\mathcal F $-connection if and only if there
exists a group morphism $\Psi \colon \pi_1(\leaf) \to {\mathrm{Inner}}(\mathcal T_0) $
such that its associated \data{} is given by $L(\cF)=L$, $H= \mathrm{Inner}(\cT_0) $ and the flat principal $\mathrm{Inner}(\cT_0)$-bundle $P(\cF)$ given by $(\leaf^u \times \mathrm{Inner}(\cT_0))/\pi_1(\leaf)$.
\end{cor}

\subsection{The torus case}\label{subsec:Torus}

Let $(\mathbb R^d, \mathcal T_0) $ be a formal singular foliation made of formal vector fields vanishing at $0$ on $\mathbb R^d $. 

\begin{deff}
\label{def:torusdata}
We call  \emph{torus triple} the following data
\begin{enumerate}
\item two symmetries $\phi,\psi \in \mathrm{Sym}(\mathcal T_0) $ such that $ \psi \circ \phi \circ \psi^{-1} \circ \phi^{-1} \in {\mathrm{Inner}}(\mathcal T_0)$,
 \item a homotopy class\footnote{Since the fibers of $ \mathrm{Inner}(\mathcal T_0) \to \mathrm{Inner}(\mathcal T_0)/\mathrm{exp}(({\mathcal T_0)_{\geq 2}}) $ are contractible, we could equivalently require to be given a homotopy class of a path from the unit to the image of $ \kappa$ in the quotient group  $\mathrm{Inner}(\mathcal T_0)/\mathrm{exp}(({\mathcal T_0)_{\geq 2}})$, i.e.\ an element of the universal cover of $\mathrm{Inner}(\mathcal T_0)/{(\mathrm{exp}}({\mathcal T}_0)_{\geq 2})$.  } $[\kappa] $ of a smooth curve $ \kappa(t)$ in $ \mathrm{Inner}(\mathcal T_0)$ such that $\kappa(0)= {\mathrm{id}} $ and 
  \begin{equation}
\label{eq:comutator} \kappa(1) = \psi \circ \phi \circ \psi^{-1} \circ \phi^{-1}. \end{equation} 
 \end{enumerate}
We say that two torus triples $ (\phi,\psi,[\kappa])$ and $ (\tilde{\phi},\tilde{\psi},[\tilde{\kappa}])$ are \emph{equivalent} if there exists paths $g(t), h(t) : [0,1] \to \mathrm{Inner}(\mathcal T_0)$ such that $ g(0)=h(0)=1$ and 
$$  \tilde{\phi} = g \circ \phi  , \tilde{\psi} = h \circ \psi , [\tilde{\kappa}] = [c_{\psi^{-1}}(g^{-1}(t)) \circ  h^{-1}(t) \circ \kappa(t) \circ g(t) \circ c_\phi (h(t)) ]     .$$ (Here $ c$ stands for the conjugation.)
\end{deff}

The groups in  the definition above are diffeological groups \cite{zbMATH01867165}, so the notion of smooth curves makes sense.
We intend to prove the following proposition.

\begin{prop}
\label{prop:torus_triple}
Given a formal singular foliation $\mathcal T_0 $ on $\mathbb R^d $, there is a one-to-one correspondence between:
\begin{enumerate}
\item[(i)] Formal singular foliation with transverse model $(\mathbb{R}^d, \mathcal T_0)$ along the leaf $ \mathbb T^2$.
\item[(ii)] Equivalence classes of torus triples. 
\end{enumerate}
\end{prop}
\begin{proof}
Of course, to prove this proposition, we will use Theorem \ref{thm:classification} and show that equivalence classes of torus triples are in one-to-one correspondence with equivalence classes of \datas{}. Without any loss of generality, one can assume that $\mathbb T^2 = \mathbb R^2/ \mathbb Z^2 $ with projection $ (x,y) \mapsto \overline{(x,y)}$.

Let $\mathcal F $ be a formal singular foliation along the leaf $\leaf = \mathbb T^2$.
Let $ (\tilde{L},H,P)$ be its \data{}. 
Choose a principal bundle connection on $P \to \mathbb T^2 $. Parallel transportation along the two fundamental loops\footnote{We acknowledge the use of a code from "user121799" from the website stackexchange for the \LaTeX{} drawing; see \url{https://tex.stackexchange.com/questions/348/how-to-draw-a-torus}.}:
$ C_1 \colon x \mapsto \overline{(x,0)} $ and $ C_2 \colon y \mapsto \overline{(0,y)}  $

\begin{center}
\begin{tikzpicture}[tdplot_main_coords]
\pgfmathsetmacro{\R}{4}
\pgfmathsetmacro{\r}{1}
 \draw[thick,fill=gray,even odd rule,fill opacity=0.3] plot[variable=\x,domain=0:360,smooth,samples=71]
 ({torusx(\x,vcrit1(\x,\tdplotmaintheta),\R,\r)},
 {torusy(\x,vcrit1(\x,\tdplotmaintheta),\R,\r)},
 {torusz(\x,vcrit1(\x,\tdplotmaintheta),\R,\r)}) 
 plot[variable=\x,
 domain={-180+umax(\tdplotmaintheta,\R,\r)}:{-umax(\tdplotmaintheta,\R,\r)},smooth,samples=51]
 ({torusx(\x,vcrit2(\x,\tdplotmaintheta),\R,\r)},
 {torusy(\x,vcrit2(\x,\tdplotmaintheta),\R,\r)},
 {torusz(\x,vcrit2(\x,\tdplotmaintheta),\R,\r)})
 plot[variable=\x,
 domain={umax(\tdplotmaintheta,\R,\r)}:{180-umax(\tdplotmaintheta,\R,\r)},smooth,samples=51]
 ({torusx(\x,vcrit2(\x,\tdplotmaintheta),\R,\r)},
 {torusy(\x,vcrit2(\x,\tdplotmaintheta),\R,\r)},
 {torusz(\x,vcrit2(\x,\tdplotmaintheta),\R,\r)});
 \draw[thick] plot[variable=\x,
 domain={-180+umax(\tdplotmaintheta,\R,\r)/2}:{-umax(\tdplotmaintheta,\R,\r)/2},smooth,samples=51]
 ({torusx(\x,vcrit2(\x,\tdplotmaintheta),\R,\r)},
 {torusy(\x,vcrit2(\x,\tdplotmaintheta),\R,\r)},
 {torusz(\x,vcrit2(\x,\tdplotmaintheta),\R,\r)});
 \foreach \X  in {240}  
 {\draw[thick,dashed] 
  plot[smooth,variable=\x,domain={360+vcrit1(\X,\tdplotmaintheta)}:{vcrit2(\X,\tdplotmaintheta)},samples=71]   
 ({torusx(\X,\x,\R,\r)},{torusy(\X,\x,\R,\r)},{torusz(\X,\x,\R,\r)});
 \draw[thick] 
  plot[smooth,variable=\x,domain={vcrit2(\X,\tdplotmaintheta)}:{vcrit1(\X,\tdplotmaintheta)},samples=71]   
 ({torusx(\X,\x,\R,\r)},{torusy(\X,\x,\R,\r)},{torusz(\X,\x,\R,\r)})
 node[below]{$C_1$};
 }
 \draw[thick] plot[smooth,variable=\x,domain=-140:-85,samples=71]   
 ({torusx(-15+200*cos(\x),0,\R,\r)},
 {torusy(-15+200*cos(\x),0,\R,\r)},
 {torusz(-15+200*cos(\x),0,\R,\r)})
 node[above left]{};
  \draw[thick,dashed] plot[smooth,variable=\x,domain=0:180,samples=71]   
 ({torusx(-15+200*cos(\x),0,\R,\r)},
 {torusy(-15+200*cos(\x),0,\R,\r)},
 {torusz(-15+200*cos(\x),0,\R,\r)})
 node[right=3mm]{$C_2$};
\end{tikzpicture}
\end{center}
yields two elements of $H$ that we denote by $\psi $ and $\phi $, respectively. Now, the commutator $ \psi \circ \phi \circ \psi^{-1} \circ \phi^{-1}$
vcorresponds to parallel transportation through the contractible path $ C_1 \circ C_2 \circ C_1^{-1 } \circ C_2^{-1} $. By Theorem \ref{thm:isYM}, the outcome of this parallel transportation is an element of $\kappa \in {\mathrm{Inner}}(\mathcal T_0) $. Any homotopy between this path and the constant path $ \overline{(0,0)}$
gives, through parallel transportation, a smooth map $\kappa(t) $ valued in $\mathrm{Inner} (\mathcal T_0)$ that interpolates between ${\mathrm{id}} $ and $\kappa $. Hence this gives a torus triple. A  different choice of formal $ \mathcal F$-connection gives an equivalent torus triples. Also, a different homotopy between $C_1 \circ C_2 \circ C_1^{-1 } \circ C_2^{-1}$ and the constant path, being itself homotopic to the first one because $ \pi_2$ is zero for a torus,  gives  equivalent torus triples.

Conversely, assume a torus triple $(\phi,\psi, [\kappa(t)]) $ is given. Let us construct a \data{}. 
First we construct $P$.
Let $H$ be the subgroup containing $ \mathrm{Inner}(\mathcal T_0) $ generated by $ \phi$ and $ \psi$. 
The idea consists in considering $[0,1]^2 \times H$ modulo the relations:
 $$  ((s,0) , h ) \sim  ((s,1) , \psi \circ  h )  \hbox{ and }  
 ((0,t) , h ) \sim ((1,t) ,  \kappa(t) \circ \phi  \circ  h ) $$
  for all $s,t\in  [0,1]^2$ and $h \in H$.
 In particular, one identifies 
 $ ((0,0) \times H) $  and $( (1,1) \times H)$ in two different manners:
 \begin{align*}
       ((0,0) , h )& \sim  ((0,1) ,  \psi \circ h ) \\ 
       & \sim ((1,1) , \kappa(1) \circ \phi \circ  \psi \circ  h ) 
 \end{align*}
 and
  \begin{align*}
       ((0,0) , h )& \sim  ((1,0) ,  \phi \circ h ) \\ 
       & \sim ((1,1) ,  \psi \circ  \phi \circ h ) .
 \end{align*}
 These two identifications have to coincide by Equation \eqref{eq:comutator}, so that the quotient is a well-defined principal $H$-bundle over $\mathbb T^2 $. Since $H$ is generated by $\phi$,$\psi$, and ${\mathrm{Inner}}(\mathcal T_0)$, $P$ is connected. By Remark \ref{rem:Pgivestriple}, it defines a \data{}.
 Equivalent torus triples can be checked to define equivalent \datas{}. 
This proves the claim.
\end{proof}

\begin{rem}
Take $\mathcal T_0 $ be made of vector fields that vanish  or order at least $2$. Consider a torus triple $(\phi,\psi, [\gamma] )$. It may not be true that the symmetries $\phi $ and $\psi $ can be multiplied by inner symmetries and then commute.
In particular, it may not be possible to see $\mathcal F $ as a quotient of $(\mathbb R^2, \mathfrak X(\mathbb R^2)) \times (\mathbb R^d, \mathcal T_0) $ through a diagonal action of $ \mathbb Z^2$. In general, the reconstruction procedure may have to go through infinite dimensional geometry, as noticed already in Remark \ref{rem:througInfinite}.  
\end{rem}

\begin{example}
  \normalfont 
  \label{ex:torustriple}
  Let $\mathcal T_0 $ be the formal singular foliation of concentric circles, \textit{i.e.}\ the formal singular foliation on $\mathbb R^2 $ generated by $$  x \frac{\partial}{\partial y}- y \frac{\partial}{\partial x}.$$
  For any $n \in \mathbb Z $, there is a non-contractible loop $t\mapsto \kappa_n(t) $ in ${\mathrm{Inner}}(\mathcal T_0) $ mapping $t \in [0,1]$ to the rotation of center $0$ and angle $2 n \pi t $.
In particular, for any $n \in \mathbb Z $,  there is a torus triple given by $\phi=\psi={\mathrm{id}} $ and by the loop $\kappa_n(t) $. They are not isomorphic, although, for a well-chosen $ \mathcal F$-connection, the parallel transport along the two fundamental circles of the torus is the identity map of the transversal singular foliation $\mathcal T_0 $.  
This answers questions \textbf{Q3a}, \textbf{Q3c} (concentric circles case) asked in the introduction.
\end{example}

\begin{example}\label{ex:IntroQuestionsAboutTorusAsLeaf}
\normalfont
By Remark \ref{rem:concentric} and Corollary \ref{cor:transvQuadra}, for $\mathcal T_0 $ the spiral singular foliation of the Introduction, there is a one-to-one correspondence between formal singular foliations admitting the torus as a leaf and spirals as a transverse model, and group morphisms from $\mathbb Z^2 $ to outer automorphisms of this transverse model (which contains at least rotations, so contains $S^1$).
This also answers questions \textbf{Q3b} and \textbf{Q3c} (spiral case) in the introduction.
\end{example}

\section*{Acknowledgments}
 This paper is part of S.-R.\ F.'s post-doc fellowship at the National Center for Theoretical Sciences (NCTS, \begin{CJK*}{UTF8}{bkai}國家理論科學研究中心\end{CJK*}), which is why S.-R.\ F.\ also wants to thank the NCTS. S.-R.\ F.\ wants to thank Siye Wu, Mark John David Hamilton and Alessandra Frabetti for their great help and support along my professional way. Also big thanks to my mother, father, Dennis, Gregor, Marco, Nico, Jakob, Kathi, Konstantin, Lukas, Locki, Gareth, Philipp, Ramona, Oguz, Annerose, Michael, Maxim, Anna, and my girlfriend Rachelle for all their love. C.L.-G.\ would like to thank Tsing Hua University \begin{CJK*}{UTF8}{bkai}國立清華大學\end{CJK*} (Taiwan) and National Center for Theoretical Sciences (NCTS, \begin{CJK*}{UTF8}{bkai}國家理論科學研究中心\end{CJK*}) for hosting him in the Fall of 2022 and 2023. We thank River Chiang, Salah Mehdi and Noriaki Ikeda for opportunies to present this work in Jeju, Nancy and Kyoto. 

We acknowledge important discussions with Ruben Louis and Leonid Ryvkin at early stages of the project, as well as discussions with David Miyamoto, Leonid Ryvkin, Hyungrok Kim, and Marco Zambon for suggesting several improvements when the first version was made public.

\appendix
\setcounter{equation}{0}
\renewcommand{\theequation}{\Alph{section}.\arabic{equation}} 

{\color{blue}

\section{Classifying multiplicative Yang-Mills connections in the centerless case}\label{app:ClassOfYM}

In this appendix, we explain how the main construction of this article extends to the broader setting of centerless Lie group bundles endowed with a multiplicative Yang–Mills connection. More precisely, we show that these objects can be classified up to field redefinitions -an equivalence relation as defined in \cite{MyThesis}- in the same fashion as in the classification of singular foliations. Multiplicative Yang-Mills connections were first introduced in \cite{SRFCYM}, as an integrated version of a structure appearing to describe Ehresmann connections and Lagrangians in curved Yang-Mills gauge theories; for more literature on curved gauge theories see \cite{OriginofCYMH, mayer2009lie, CurvedYMH, My1stpaper, MyThesis}, and for Ehresmann  connections in this context see also \cite{castrillon2023principal}. Those connections are a special kind of multiplicative connections on Lie groupoids as in \cite{crainic2003differentiable}, but see also \cite{blazquez2022group} for multiplicative connections on Lie group bundles alone. Similar constructions on the group structure also appear in the studies of higher gauge theories and Yang-Mills-Higgs theories, see \cite{samann2020towards, Kim:2019owc, rist2022explicit, perez2025higher, Chatterjee:2502.02284, fischer2024adjusted}. Infinitesimally those connections appear as ``Lie derivation laws covering a coupling'' as in \cite{MR2157566}.

Let $G$ be a Lie group and $ G_0$ its connected component of the identity.
We recall that a Yang-Mills connection for a Lie group bundle $ \mathcal G\to L$ whose typical fiber is a  Lie group $G$, is a pair $ (\mu,\zeta)$ with $ \mu$ a Lie group bundle connection (= an Ehresmann connection such that parallel transportation is by Lie group automorphisms) and $ \zeta$ a $2$-form valued in the Lie algebra bundle such that the curvature of $\mu $ (which is a $2$-form valued in the Lie algebra of group automorphisms of the fiber) is the inner automorphism associated to $ \zeta$. In equation
 $$ {\mathrm{curv}}(\mu) = {\mathrm{ad}}_\zeta$$
In the centerless case, $ \zeta$ is unique, and it simply means that the curvature is by inner automorphisms, i.e.\  for any contractible loop the parallel transportation in $ \mathcal G$ along a loop $ \gamma$ based at an arbitrary point $ \ell \in L$ is an inner automorphism of $ \mathcal G_\ell$ associated to an element in the connected component of the identity.
In this case, field redefinition just consists in identifying two such connections whose parallel transportation along a path differ by an inner automorphism by an element in the connected component of the identity, see \cite{SRFCYM}.

\begin{cor}[Classification of Yang-Mills bundles with trivial center]
\label{cor:classYMcenterless}
Let $G$ be a Lie group with trivial center and $L$ a connected manifold. We have a natural 1:1 correspondence between:

\begin{itemize}
	\item Pairs $(\cG, [\mathbb{H}^\cG])$ made of Lie group bundles $\cG \to L$ with structural Lie group $G = \mathcal G_\ell$ and a class, up to field redefinition, of multiplicative Yang-Mills connections $\mathbb{H}^\cG$.
	\item Equivalence classes of pairs made of:
    \begin{enumerate}
    \item a group morphism\footnote{Here $G_0$ is seen as a sub-group of $ \mathrm{Aut}(G)$ through conjugation. Notice that if $G$ is connected, $ \mathrm{Aut}(G)/G_0= \mathrm{Out}(G)$ .} $\Xi \colon \pi_1(L,\ell) \longrightarrow \mathrm{Aut}(G)/G_0$, that we call the outer holonomy
    \item a principal $H$-bundle $P \to L $ with $H \subset {\mathrm{Aut}}(G)$ the inverse image through ${\mathrm{Aut}}(G) \to {\mathrm{Aut}}(G)/G_0$ of the image of $\Xi$, naturally inducing that the quotient $P/G_0$ is the Galois bundle $ \tilde L \to L$ associated to the kernel of $ \Xi$. 
    \end{enumerate}
    We identify to such pairs $(\Xi_1,P_1) $ and $ (\Xi_2,P_2)$, with $P_1$ a principal $ H_1$-bundle and $ P_2$ a principal $ H_2 $-bundle, if there exists $ \Psi \in \Aut(G)$ such that $ \Xi_1= \overline{\Psi}\circ \Xi_2 \circ \overline{\Psi}^{-1}$, where $\overline \Psi$ is the class of $\Psi$ in $\mathrm{Aut}(G)/G_0$, and if there exists an fiber preserving diffeomorphism $ \chi \colon P_1 \simeq P_2 $ such that $ \chi (p \cdot g) = \chi(p) \cdot \Psi(g)$ for all $g \in H_1$.
    Notice that $ \Psi(g)$ automatically belongs to $ H_2$, so that this condition makes sense.
    \end{itemize} 
\end{cor}

\begin{proof}
This correspondence goes via associated bundle constructions.
We start from a pair  $(\cG, [\mathbb{H}^\cG])$.
Choose a representative $\mathbb{H}^\cG$.
 Parallel transportation  $\mathrm{PT}^\cG_\gamma$ along a path $ \gamma$ in $L$ satisfies the following:
\begin{itemize}
    \item $\mathrm{PT}^\cG_\gamma$ has values in $\mathrm{Aut}(\cG_l)\cong \mathrm{Aut}(G)$ for every loop based at $l$
    \item For any two homotopic paths $ \gamma_0, \gamma_1$ from $\ell$ to $\ell'$ the parallel transports $ \mathrm{PT}^\cG_{\gamma_0}$ and $\mathrm{PT}^\cG_{\gamma_1}$ differ by an inner symmetry of $G_0$, \textit{i.e.}\
    \begin{equation*}
    \mleft(\mathrm{PT}^\cG\mright)_{\gamma_1}^{-1}  \circ \mathrm{PT}^\cG_{\gamma_0} \in {\mathrm{Inner}} (\cG_{\ell,0}) \cong {\mathrm{Inner}} (G_0) ~.
    \end{equation*}
    \item Two parallel transports of two different multiplicative Yang-Mills connections are related by field redefinitions if and only if they differ by an element of ${\mathrm{Inner}} (G_0)$. 
\end{itemize}

These are properties similar to those of Corollary \ref{cor:PAPBPC} for singular foliations, but now w.r.t.\ Lie group theoretical notions. In particular, we can repeat the constructions of Subsection \ref{sec:YMgroupoid} and Section \ref{sec:ClassI}.
 Fix a point $\ell \in L$. Define the outer holonomy as a group morphism $\Xi \colon \pi_1(L, \ell) \to \mathrm{Aut}(G)/G_0$ as follows: $\Xi$ is the parallel transport $\mathrm{PT}_\gamma^\cG$ over loops $\gamma$ based at $\ell$ followed by the quotient map $\mathrm{Aut}(G) \to \mathrm{Aut}(G)/G_0$ after identifying $\cG_\ell \cong G$. The properties above imply it is a well-defined group morphism. Consider the transitive groupoid $\Gamma$  over $L$ whose elements are Lie group automorphisms $\cG_{\ell_0} \to \cG_{\ell_1}$, identifying the fibers over two points $\ell_0, \ell_1 \in L$.  Of course, the source of such an element is $\ell_0$ and the target is $\ell_1$. 
Now consider the transitive Lie subgroupoid $\mathrm{YM}(\cG, [\mathbb{H}^\cG]) \subset \Gamma$
of all by all compositions of a parallel transports $\mathrm{PT}^\cG$ along any path with composition by an inner automorphisms of the source or the target. 
The isotropy $H$ of this groupoid at a given point $ \ell$ is by construction obtained precisely as in the statement. Also, the source-fiber of this groupoid over $ \ell$ is the principal $H$-bundle $P$. Different choices for the identification $ \mathcal G_\ell \simeq G$ or for the base point $ \ell$ give equivalent pairs. This constructs a map from the first item to the second one. 

Let us construct the inverse map. Start with a pair $ (\Xi,P)$. 
Consider  the associated Lie group bundle $\mathcal G \coloneqq \mleft( P \times G \mright) \Big/ H$, where $H \subset \mathrm{Aut}(G)$ acts canonically on $G$. 
We have to equip it with a Yang-Mills connection. To start with, let us equip $ P$ with an Ehresmann connection $A$ as a principal $H$-bundle. This connection $A$ induces a connection $\mathbb H^\cG$ on the associated group bundle, such that
parallel paths are paths of the form $m(t) := [p(t),g]$ 
with $ p(t)$ a parallel path in $ P$, where $[\cdot, \cdot]$ denotes the associated equivalence class.
Parallel transportation is then by group morphisms, since the products of two parallel paths $[p_1(t),g] $ and $[p_2(t),g'] $ over the same path in $L$ is a parallel path $ [p_1(t),g h(g')]$ where $h \in H$ is the unique element such that $ p_1(t) \cdot h = p_2(t)$.  
In particular, parallel paths $p \colon [0,1] \to P $ over a loop in $L$ satisfy the property that there exists $ h \in H$ such that  $p(1) = p(0) \cdot h $. So that $m(1)=[p(1),g]= [p(0) \cdot h,g]= [p(0) ,h ( g)]$, so that parallel transportation along a loop is given as an action by an element of $H$. For a contractible loop, it has to be an inner automorphism of $G_0$, i.e.\ an element in $G_0$ acting by conjugation. This implies that $\mathbb H^\cG$ is a Yang-Mills connection. Now, assume we have two different connections $A_1$ by $ A_2$ on $P$, parallel paths $ p_1(t)$ and $ p_2(t)$ over the same path on $L$ and starting at the same point differ by a path in  $h(t) \in H$, i.e.\ $ p_1(t) \cdot h(t) = p_2(t)$. For connectedness reasons and by $h(0) = e$, $ h(t)$ is in fact valued in ${\mathrm{Inner}}( G_0) \simeq G_0$.  Hence parallel paths $ [p_1(t),g]$ and $ [p_2(t),g]$ in $ \mathcal G$ are related by: 
  $$  [p_2(t),g] =  [p_1(t) \cdot h(t), g] = [p_1(t), h(t) \cdot g \cdot h(t)^{-1}] =  [p_1(t), h(t)] \cdot [p_1(t),g] \cdot [p_1(t), h(t)]^{-1} $$
So they differ by an inner automorphism of $G_0$, i.e.\ by a field redefinition.  

We have to check that equivalent pairs $(\Xi_1,P_1) $ and $ (\Xi_2,P_2)$, related by some $ \Psi \in \mathrm{Aut}(G)$ and some diffeomorphism $ \chi$, yield canonically isomorphic group bundles $ \mathcal G_1$ and $ \mathcal G_2$, and Yang-Mills connections that differs by a field redefinition.
To start with, an explicit Lie group bundle isomorphism is given by:
 \begin{align*} \mathcal G_1 \coloneqq  \mleft( P_1 \times G \mright) \Big/ H_1   & \longrightarrow \mathcal G_2 \coloneqq \mleft( P_2 \times G \mright) \Big/ H_2 \\ [p,g]& \mapsto   [\chi(p), \Psi(g) ] \end{align*}
which is well defined since $ \Psi(H_1)=H_2$.
Now, for any connection $ A_1$ on $ \mathcal G_1$, we can equip $P_2$ with the image of that connection through $\chi$. If we do so, the diffeomorphism $ \mathcal G_1 \simeq  \mathcal G_2$  intertwine their respective parallel paths. In particular, it intertwines their Yang-Mills connection. 
This construct a map from the second item to the first one.

Last, we have to show that both maps are inverse one to the other, so that this correspondence is 1:1. Choose $(\mathcal G,[\mathbb H^\cG]) $ as in the first item. Fix a point $\ell \in L$, and, starting with $\cG$, we construct $P$ as described above. By construction an element $p\in P_{\ell_1}$ acts on $g \in G  \simeq \mathcal G_\ell$  to yield an element $p(g)\in\cG_{\ell_1}$. In particular, we have a canonical and well-defined isomorphism of Lie group bundles
\begin{align*}
    (P \times G ) \Big/ H &\to \cG ~,\\
    [p, g] &\mapsto p(g)~.
\end{align*}
As before this isomorphism intertwines the multiplicative Yang-Mills connections and their equivalence classes on $(P \times G)/H$ and $\cG$, such that the maps from the first to the second item in the corollary is injective because we have shown that the map from the second to the first item is a left inverse. 

We prove surjectivity of the map from the first item to the second item by showing that the map from the second to the first item is a right inverse. Start from a pair $(\Xi,P) $.
By definition,  $P$ is principal $H$-bundle, it therefore comes comes a principal $H $-bundle connection. Consider  $\cG \coloneqq (P \times G)/H$ equipped with an associated connection as multiplicative Yang-Mills connection. It suffices to prove that the Atiyah groupoid of $P$ coincides with the Yang-Mills groupoid of $ \mathcal G$: this implies that the principal $H$-bundle associated to the Yang-Mills connection is precisely $ P$, and completes the proof. 
Let us construct a groupoid  morphism from the Yang-Mills groupoid to the Atiyah groupoid.
On the one hand, the Atiyah groupoid $(P \times P)/H$ of $P$ acts canonically on $\cG$. On the other hand, elements ${\rm{PT}}_\gamma^\cG \circ [p,q]$ of $\mathrm{YM}(\cG, [\mathbb{H}^\cG])$ also act on $[p, g] \in \cG$, where we fix again a curve $\gamma$ from $\ell_0$ to $\ell_1$, and $[p, q] \in \cG_{\ell_0,0} \cong {\rm{Inner}}(\cG_{\ell_0,0})$. By construction and trivial center, ${\rm{PT}}_\gamma^\cG$ is 1:1 to the parallel transport ${\rm{PT}}_\gamma^P$ of an Ehresmann connection on $P$, that is, we have
\begin{equation*}
    \mleft({\rm{PT}}_\gamma^\cG \circ [p,q]\mright)([p, g])
    =
    \mleft[ {\rm{PT}}_\gamma^P(p) \cdot q, g \mright]~.
\end{equation*}
%
%
which motivates the smooth map $\xi$ given by
\begin{align*}
    \mathrm{YM}\mleft(\cG, [\mathbb{H}^\cG]\mright) &\to (P \times P)/H~,\\
    {\rm{PT}}_\gamma^\cG \circ [p,q] &\mapsto \mleft[  {\rm{PT}}_\gamma^P(p) \cdot q , p \mright]~.
\end{align*}
This is well-defined due to the mentioned 1:1 correspondence of parallel transports, and a different choice of $p' = p \cdot h \in P_{\ell_0}$ gives on the right hand side
\begin{equation*}
    \mleft[  {\rm{PT}}_\gamma^P(p \cdot h) \cdot h^{-1}(q) , p \cdot h \mright]
    =
    \mleft[  {\rm{PT}}_\gamma^P(p) \cdot \mleft( h \circ h^{-1}(q) \mright) , p \cdot h \mright]
    =
    \mleft[  {\rm{PT}}_\gamma^P(p) \cdot q , p \mright]~,
\end{equation*}
where the last equality is due to 
\begin{equation*}
    \mleft( h \circ h^{-1}(q) \mright)(g)
    =
    q h(g) q^{-1}
    =
    \mleft( q \circ h \mright)(g)
\end{equation*}
for all $g \in G$. Furthermore, for another curve $\gamma'$ from $\ell_1$ to $\ell_2$ and $q' \in G_0$ we have
\begin{equation*}
    \xi\mleft( {\rm{PT}}_{\gamma'}^\cG \circ \mleft[{\rm{PT}}_\gamma^P(p) \cdot q,q'\mright] \mright) \cdot \xi\mleft( {\rm{PT}}_\gamma^\cG \circ [p,q] \mright)
    =
    \mleft[ {\rm{PT}}_{\gamma' \star \gamma}^P(p) \cdot q \cdot q'  , p \mright]~,
\end{equation*}
while we also have
\begin{equation*}
    \mleft[{\rm{PT}}_\gamma^P(p) \cdot q,q'\mright]\circ{\rm{PT}}_{\gamma}^\cG
    =
    {\rm{PT}}_\gamma^\cG\mleft(\mleft[p \cdot q,q'\mright]\mright)\circ{\rm{PT}}_{\gamma}^\cG
    =
    {\rm{PT}}_{\gamma}^\cG\circ\mleft(\mleft[p \cdot q,q'\mright]\mright)~,
\end{equation*}
and thus, by $\mleft[p \cdot q,q'\mright] = \mleft[p,qq'q^{-1}\mright]$,
\begin{equation*}
    \xi \mleft( {\rm{PT}}_{\gamma'}^\cG \circ \mleft[{\rm{PT}}_\gamma^P(p) \cdot q,q'\mright] \circ {\rm{PT}}_\gamma^\cG \circ [p,q] \mright)
    =
    \xi\mleft( {\rm{PT}}_{\gamma'\star\gamma}^\cG \circ \mleft[p, qq'\mright] \mright)
    =
    \mleft[ {\rm{PT}}_{\gamma' \star \gamma}^P(p) \cdot q \cdot q'  , p \mright]~,
\end{equation*}
which proves that $\xi$ is indeed a morphism of Lie groupoids because source and targets are naturally intertwined. In fact, this is an isomorphism: $P$ is a principal $G_0$-bundle over the Galois bundle $\tilde L$, in particular $( P \times P)/H$ is an extension of the Atiyah groupoid of $\tilde L$, that is, a given class $[p', p]$ projects to some homotopy class of a curve (modulo the kernel of $\Xi$) which we choose to be $\gamma$. The projection of $P$ to $\tilde L$ is via the quotient of $G_0$, and as usual the parallel transport in $P$ over a loop at $\ell$ modulo $G_0 \cong H_0$ gives a group morphism $\eta \colon \pi_1(L,l) \to H/G_0 \cong \mathrm{Im}(\Xi)$, defined on the homotopy class of the corresponding loop. $\eta$ is surjective because $P$ is connected as an extension of $\tilde L$ by $G_0$, in particular there is a group isomorphism of $\pi_1(L, \ell)$ modulo the kernel of $\eta$ with $\pi_1(L, \ell)$ modulo the kernel of $\Xi$; in particular, under this isomorphism (and canonically extending by the kernels) $\eta$ is $\Xi$. We can now conclude that the projection of $\mleft[\mathrm{PT}^P_\gamma(p), p\mright]$ to the Atiyah groupoid of $\tilde L$ is given by the corresponding class of $\gamma$.

Then there is a unique $q \in G_0$ such that $p' = {\rm{PT}}_\gamma^P(p) \cdot q$, which proves surjectivity. Any other choice $\gamma'$ homotopic to $\gamma$ (modulo the kernel of $\Xi$) similarly comes with a unique $q' \in G_0$ such that $p' = {\rm{PT}}_{\gamma'}^P(p) \cdot q'$, but we also know that there is a $\hat q \in G_0$ with 
\begin{equation*}
    {\rm{PT}}_{\gamma'}^\cG = {\rm{PT}}_{\gamma}^\cG \circ \mleft[ p, \hat q \mright]~,
\end{equation*}
and thus $p' = {\rm{PT}}_{\gamma}^P(p) \cdot \hat q \cdot q'$ under $\xi$ which implies $\hat q q' = q$, proving injectivity by implying
\begin{equation*}
    {\rm{PT}}_{\gamma'}^\cG \circ \mleft[ p, q' \mright]
    =
    {\rm{PT}}_{\gamma}^\cG \circ \mleft[ p, q \mright]~.
\end{equation*}
Taking an arbitrary curve $\gamma'$ outside the homotopy class (modulo the kernel of $\Xi$) of $\gamma$ implies that ${\rm{PT}}_{\gamma'}^P(p)$ projects to a different class in $\tilde L$ and thus sits in a different fiber over $\tilde L$, which proves injectivity.

Since $ \xi$ is a Lie groupoid isomorphism, both  the Atiyah groupoid and the Yang-Mills groupoids have isomorphic isotropies and associated principal bundles. It implies that the map from the first item to the second item gives back the $H$-principal bundle $ P$. Also, the Atiyah groupoid of $P$ is an extension of the Atiyah groupoid of the Galois bundle $ \tilde L$, it means that the outer holonomy obtained in the map from the first item to the second item gives back $\tilde L $. It therefore have to match $ \Xi$ by construction: As we have shown earlier, over a loop $\gamma$ based at $\ell$, the quotient of ${\rm{PT}^P_\gamma(p)}$ by $G_0$ gives a canonical group morphism $\pi_1(L, \ell) \to H/G_0$ which one can identify as $\Xi$. There is a unique $h \in H$ such that ${\rm{PT}^P_\gamma(p)} = p \cdot h$, and by construction $\Xi([\gamma]) = \overline h$, where $[\gamma]$ is the homotopy class of $\gamma$, and $\overline h$ the equivalence class of $h$ in $\mathrm{Aut}(G)/G_0$. The outer holonomy induced by the map from the first to the second item is the quotient of ${\rm{PT}}^\cG_\gamma$ by $G_0$, but by construction
\begin{equation*}
    {\rm{PT}}^\cG_\gamma ([p, g])
    =
    \mleft[ \mathrm{PT}^P_\gamma(p), g \mright]
    =
    \mleft[p, h(g) \mright]
\end{equation*}
for all $[p, g]\in \cG$; thus, the quotient of ${\rm{PT}}^\cG_\gamma$ by $G_0$ maps $[\gamma]$ to $\overline h$, and henceforth aligns with $\Xi$. This finishes the proof.
\end{proof} 

\begin{example}
    If $G$ corresponds to a simple and centerless Lie group (by definition, one then has $G = G_0$), then $(\cG, [\mathbb{H}^\cG])$ is in 1:1 correspondence to principal $G$-bundles because $\mathrm{Out}(G)$ is trivial. Furthermore, there is just one equivalence class $[\mathbb{H}^\cG]$, because the difference of parallel transports of two Yang-Mills connections can only differ by inner automorphisms. Thus, in other words, if $\cG$ is a simple and centerless Lie group bundle admitting a multiplicative Yang-Mills connection, then there is just one multiplicative Yang-Mills connection up to field redefinition and a unique principal $G$-bundle up to principal bundle automorphisms.
\end{example}

}

\bibliography{Biblio}

\end{document}